\documentclass[10pt,a4paper,reqno]{amsart}
\usepackage{etex}

\pdfoutput=1

\usepackage{amsmath,amstext,amssymb,mathrsfs,amscd,amsthm}
\usepackage[foot]{amsaddr}
\usepackage{mathrsfs}
\usepackage{amsfonts}
\usepackage[all]{xy}
\usepackage{booktabs}
\usepackage{verbatim}
\usepackage{enumerate}
\usepackage{units}
\usepackage{graphicx}
\usepackage{adjustbox}
\usepackage{tikz}
\usetikzlibrary{arrows,calc,positioning,decorations.pathreplacing,shapes.multipart}
\usepackage[noadjust]{cite}
\usepackage{mhsetup}
\usepackage{mathtools}
\usepackage{bbm}
\usepackage{xspace}
\usepackage{booktabs}
\usepackage{caption}
\usepackage[textsize=scriptsize]{todonotes}
\newcommand{\rednote}[1]{}

\newcommand{\redinote}[1]{}
\newcommand{\blueinote}[1]{}
\newcommand{\greeninote}[1]{}
\usepackage[hyperfootnotes=false,pagebackref,colorlinks,citecolor=blue,linkcolor=blue,urlcolor=blue,filecolor=blue]{hyperref}

\renewcommand*{\backref}[1]{}
\renewcommand*{\backrefalt}[4]{%
  \ifcase #1 %
    \relax
  \or
    {\small [cited on p.~#2]}%
  \else
    {\small [cited on pp.~#2]}%
  \fi%
}
\usepackage{breakcites}
\usepackage{etoolbox}
\makeatletter
\patchcmd{\@citex}{,}{;}{}{}
\makeatother
\usepackage{enumerate}
\usepackage[bottom]{footmisc}
\renewcommand{\footnoterule}{%
  \kern -3pt
  \hrule width \textwidth height 0.4pt
  \kern 2.6pt
}
\newenvironment{itemize2}%
{\begin{itemize}

\setlength{\itemsep}{1pt}
\setlength{\parskip}{0pt}
\setlength{\parsep}{0pt}}%
{\end{itemize}}

\numberwithin{equation}{section}
\numberwithin{figure}{section}
\numberwithin{table}{section}
\newtheoremstyle{italicised}
        {0.5em}{0.5em}  % space above % space below
        {\itshape}  % body font
        {}  % indent (empty = no indent)
        {\bfseries}  % header font
        {}  % punctuation after header
        {1ex}  % space after header
        {}  % manually specify the header
\theoremstyle{italicised}
\newtheorem{lemma}{Lemma}[section]
\newtheorem{proposition}[lemma]{Proposition}
\newtheorem{theorem}[lemma]{Theorem}
\newtheorem{corollary}[lemma]{Corollary}

\newtheorem{fact}[lemma]{Fact}
\newtheorem*{claim}{Claim}
\newtheorem{atheorem}{Theorem}

\newtheorem{acorollary}[atheorem]{Corollary}

\newtheorem*{theorem*}{Theorem}

\newtheoremstyle{upright}
        {0.5em}{0.5em}  % space above % space below
        {\upshape}  % body font
        {}  % indent (empty = no indent)
        {\bfseries}  % header font
        {}  % punctuation after header
        {1ex}  % space after header
        {}  % manually specify the header
\theoremstyle{upright}
\newtheorem{df}[lemma]{Definition}
\newtheorem{remark}[lemma]{Remark}

\newtheoremstyle{italicised-restate}
        {0.5em}{0.5em}  % space above % space below
        {\itshape}  % body font
        {}  % indent (empty = no indent)
        {\bfseries}  % header font
        {}  % punctuation after header
        {1ex}  % space after header
        {\thmname{#1}\thmnote{ \bfseries #3}}  % manually specify the header
\theoremstyle{italicised-restate}
\newtheorem{rthm}{Theorem}

\makeatletter
\renewcommand*{\paragraph}{\@startsection{paragraph}{4}{\z@}%
  {2.5ex \@plus 1ex \@minus 0.2ex}%
  {-1.5ex \@plus -0.2ex}%
  {\normalfont\normalsize\bfseries}}
\makeatother

\renewcommand\r[1]{\rho_{#1}}			%Replication map
\font\maljapanese=dmjhira at 1.75ex % the intrinsic replication map.
\newcommand\ro[1]{\textrm{\maljapanese\char"4D}_{#1}}			%Replication map with labels
\newcommand\ssnl[1]{\mathfrak{s}^{#1}}				%Non-linear scanning map
\newcommand\ssl[1]{\mathscr{S}^{#1}}						%Linear scanning map
\def\q{\mathbf{q}}
\newcommand\W[1]{C_{#1}^\delta(M;\theta)}                      %Configurations with vectors.
\newcommand\Wo[1]{C_{#1}(M;\theta)}                      %Configurations with vectors.
\newcommand\val[2]{(#2)_{#1}}                         %#1-adic valuation of #2.
\def\Spec{\operatorname{Spec}}

\def\Id{\mathrm{Id}}

\def\endo{\mathrm{end}}

\newcommand{\sr}{\ensuremath{\mu}}
\newcommand{\srscan}{\ensuremath{\nu}}
\newcommand{\hconn}{\ensuremath{h\mathrm{conn}}}
\newcommand{\mbar}{\ensuremath{{\,\,\overline{\!\! M\!}\,}}}
\newcommand{\cyl}{\ensuremath{\bR^n\smallsetminus\{0\}}}

\newcommand{\RP}{\mathbb{RP}}
\renewcommand{\geq}{\geqslant}
\renewcommand{\leq}{\leqslant}

\newcommand{\cB}{\mathcal{B}}

\newcommand{\id}{\mathrm{id}}

\newcommand{\incl}[3][right]%
{%
\draw[<-,>=#1 hook] #2 to ($ #2!0.5!#3 $);
\draw[->,>=stealth'] ($ #2!0.5!#3 $) to #3;%
}
\newcommand{\inclusion}[5][right]%
{%
\draw[<-,>=#1 hook] #4 to ($ #4!0.5!#5 $) node[#2,font=\small]{#3};
\draw[->,>=stealth'] ($ #4!0.5!#5 $) to #5;%
}

\newcommand\mnote[1]{}
\newcommand\mmnote[1]{}

\newcommand{\bD}{\mathbb{D}}

\newcommand{\bF}{\mathbb{F}}

\newcommand{\bN}{\mathbb{N}}

\newcommand{\bQ}{\mathbb{Q}}
\newcommand{\bR}{\mathbb{R}}

\newcommand{\bZ}{\mathbb{Z}}

\newcommand\lra{\longrightarrow}

\newcommand\End{\mathrm{End}}

\newcommand\Th{\mathrm{Th}}
\newcommand\colim{\operatorname*{colim}}

\newcommand{\map}{\mathrm{Map}}

\SelectTips{cm}{10}
\tikzset{>=to}

\title[Stability for configurations in closed manifolds]{On homological stability for configuration spaces on closed background manifolds}

\author{Federico Cantero}
\thanks{Both authors were funded by Michael Weiss' Humboldt professor grant. The first author was partially supported by the Spanish Ministry of Economy and Competitiveness under grants MTM2010-15831 and MTM2013-42178-P}
\email{\texttt{fcant{\_}01@uni-muenster.de}}
\address{{\normalfont Mathematisches Institut, Universit{\"a}t M{\"u}nster, Einsteinstra{\ss}e 62, 48149 M{\"u}nster, Germany}}

\author{Martin Palmer}
\email{\texttt{mpalm{\_}01@uni-muenster.de}}

\subjclass[2010]{55R80, 55P60, 55R25}
\keywords{Homological stability, configuration spaces, replication map, scanning map, closed background manifolds.}

\date{}

\begin{document}
\begin{abstract}
We introduce a new map between configuration spaces of points in a background manifold -- the \emph{replication map} -- and prove that it is a homology isomorphism in a range with certain coefficients. This is particularly of interest when the background manifold is closed, in which case the classical stabilisation map does not exist.

We then establish conditions on the manifold and on the coefficients under which homological stability holds for configuration spaces on closed manifolds. These conditions are sharp when the background manifold is a two-dimensional sphere, the classical counterexample in the field. For field coefficients this extends results of Church \cite{Church} and Randal-Williams \cite{RW-hs-for-ucs} to the case of odd characteristic, and for $p$-local coefficients it improves results of Bendersky--Miller \cite{BM}.
\end{abstract}

\maketitle

\renewcommand{\setminus}{\smallsetminus}
\newcommand{\point}{\ensuremath{\{*\}}}

\section{Introduction}
Let $M$ be a smooth, connected manifold without boundary of dimension $n$, and with Euler characteristic $\chi$, and denote by $C_k(M)$ the unordered configuration space of $k$ points in $M$:
\[
C_k(M) := \{\q\subset M\mid |\q| = k\},
\]
which is topologised as a quotient space of a subspace of $M^n$. After removing a point $*$ from $M$ one can define a map
\[
C_k(M\setminus\point)\lra C_{k+1}(M\setminus\point),
\]
called the \emph{stabilisation map}, which expands the configuration away from $*$ and adds a new point near to it. More generally, one can define such a stabilisation map $C_k(M)\to C_{k+1}(M)$ using any properly embedded ray in $M$ to bring in a point from infinity (such a ray exists if and only if $M$ is non-compact).

Let us assume from now on that the manifold is endowed with a Riemannian metric with injectivity radius bounded below by $\delta>0$. Define $C_k^\delta(M)\subset C_k(M)\times (0,\delta)$ to be the space of pairs $(\q,\epsilon)$, where $\q$ is a configuration whose points are pairwise at distance at least $2\epsilon$. The projection to $C_k(M)$ is a fibre bundle with contractible fibres, hence a homotopy equivalence. The main theorem in \cite{McDuff} concerns the \emph{scanning map}
\[\ssl{}\colon C_k^\delta(M)\lra \Gamma_c(\dot{T}M)_k\]
which takes values in the space of degree-$k$ compactly-supported sections of the fibrewise one-point compactification of $TM$ (see \S \ref{section:scanning}).

\begin{df}\label{dStableRange}
Given a properly embedded ray in $M$ and an abelian group $A$, define the function $\sr = \sr[M]\colon \bN\to \bN$ to be the pointwise maximum $f\colon \bN\to \bN$ such that the stabilisation map $C_k(M) \to C_{k+1}(M)$ induces isomorphisms on $H_*(-;A)$ in the range $*\leq f(k)$. This is called the \emph{stable range of the stabilisation map}.

Given a Riemannian metric on $M$ with injectivity radius bounded below by $\delta>0$ and an abelian group $A$, the function $\srscan = \srscan[M]\colon \bN\to \bN$ is defined to be the pointwise maximum $f\colon \bN\to\bN$ such that the scanning map $\ssl{}\colon C_k^\delta(M)\to \Gamma_c(\dot{T}M)_k$ induces isomorphisms on $H_*(-;A)$ in the range $*\leq f(k)$. This is called the \emph{stable range of the scanning map}.
\end{df}

Henceforth the term ``stable range'' will by default refer to the stable range $\srscan$ of the scanning map.

\begin{theorem*}[\cite{McDuff}] For non-compact $M$ the function $\sr[M]$ diverges and
\[
\srscan[M](k) = \min_{j\geq k} \{\sr[M](j)\}.
\]
The inequality $\srscan[M] \geq \srscan[M\smallsetminus\point]$ holds for all $M$, so the function $\srscan[M]$ diverges for all $M$.
\end{theorem*}

No explicit lower bound for $\sr[M]$ was given in \cite{McDuff}, but the following lower bounds have since been proved:
\begin{itemize2}\label{pageone}
\item $\sr[M](k)\geq \frac{k}{2}$ if $A=\bZ$ and $\dim(M)\geq 2$, by \cite{SegalRational,RW-hs-for-ucs}.
\item $\sr[M](k)\geq k$ if $A=\bQ$ and either $\dim(M)\geq 3$ or $M$ is non-orientable, by \cite{RW-hs-for-ucs,Knudsen2014}.
\item $\sr[M](k)\geq k-1$ if $A=\bQ$ and $M$ is orientable, by \cite{Church,Knudsen2014}.
\item $\sr[M](k)\geq k$ if $A=\bZ[\frac12]$ and $\dim(M)\geq 3$, by \cite{KupersMiller2014}.
\end{itemize2}
See also Propositions \ref{pImprovedRange} and \ref{pImprovedRangeTwisted}. Further improvements to the lower bound are possible under extra hypotheses (\cite[Proposition 4.1]{Church} and \cite[Remark~4.5]{KupersMiller2014}). Some of these results can be also deduced from \cite{LM:conf, BCT:conf, FT:conf}.

McDuff's theorem says that the homology of configuration spaces $C_k(M)$ on a non-compact manifold $M$ \emph{stabilises}, i.e., is independent of $k$ in a diverging range of degrees. For closed manifolds $M$ stabilisation maps do not exist -- this leaves open the question of when the homology of configuration spaces on closed background manifolds stabilises.

\paragraph{Stability for $p$-torsion.}
Let $\dot{T}M$ denote the fibrewise one-point compactification of the tangent bundle of $M$ and let $\Gamma_c(-)$ denote the space of compactly-supported sections. By the main result in \cite{Moller:nilpotent}, for each $k\in\bZ$ the localisation of the path-component $\Gamma_c(\dot{T}M)_k$ at a prime $p$ is homotopy equivalent to the path-component $\Gamma_c(\dot{T}M_{(p)})_k$ of the space of compactly-supported sections of the fibrewise localisation of $\dot{T}M$.\footnote{To ensure that the localisation of $\Gamma_c(\dot{T}M_{(p)})_k$ exists, we need to assume here that $M$ has the homotopy type of a finite complex. However, for the purpose of proving homological stability results, we may assume this without loss of generality; see \S\ref{ss:localisation-etc}.} In \cite{BM} Bendersky and Miller proved the existence of homotopy equivalences
\begin{equation}\label{eq:933}\Gamma_c(\dot{T}M_{(p)})_k\lra \Gamma_c(\dot{T}M_{(p)})_j\end{equation}
whenever 
\begin{itemize}
\item $p\geq \frac{n+3}{2}$ and $M$ is odd-dimensional,
\item $p\geq \frac{n+3}{2}$ and $\frac{2k-\chi}{2j-\chi}$ is a unit in $\bZ_{(p)}$, 
\item $\dot{T}M$ is trivial and $\frac{2k-\chi}{2j-\chi}$ is a unit in $\bZ_{(p)}$.
\end{itemize}
Using McDuff's theorem one obtains a zigzag of $\bZ_{(p)}$-homology isomorphisms in the stable range:
\begin{equation}\label{eq:932} C_k(M)\lra \Gamma_c(\dot{T}M)_k\lra \Gamma_c(\dot{T}M)_j\longleftarrow C_j(M).\end{equation}

We will show that linearly independent pairs of sections of $TM\oplus\epsilon$ give rise to families of fibrewise homotopy equivalences of $\dot{T}M$ after localisation, and hence maps as in \eqref{eq:933} for certain $k$ and $j$, from which we are able to extend the results of Bendersky and Miller to all odd primes and under certain conditions to the prime $2$. For a number $k\in \bZ$, we denote by $\val{p}{k}$ the $p$-adic valuation of $k$, and observe that $\frac{j}{k}$ is a unit in $\bZ_{(p)}$ if and only if $\val{p}{k} = \val{p}{j}$. If $\ell$ is a collection of primes, the $\ell$-adic valuation is the sequence of all $p$-adic valuations with $p\in \ell$.

\begin{atheorem}\label{thm:a}
Let $M$ be a closed, connected, smooth manifold. If $M$ is odd-dimensional, there are zigzags of maps as in \eqref{eq:932} inducing isomorphisms in the stable range\textup{:}
\begin{align}
H_*(C_k(M);\bZ)&\cong H_*(C_{k+1}(M);\bZ) \quad\text{if}\; \dim M = 3,7 \label{eq:thmaodd-37} \\
H_*(C_k(M);\bZ)&\cong H_*(C_{k+2}(M);\bZ) \label{eq:thmaodd-Z} \\
H_*(C_k(M);\bZ[\tfrac{1}{2}])&\cong H_*(C_{k+1}(M);\bZ[\tfrac{1}{2}]). \label{eq:thmaodd-Z12}
\end{align}
If $M$ is even-dimensional with Euler characteristic $\chi$, then for each set $\ell$ of primes \textup{(}assuming $2\not\in\ell$ if $\chi$ is odd\textup{)} there are zigzags of maps 
%If $M$ is even-dimensional and its Euler characteristic $\chi$ is even \textup{(}resp.\ odd\textup{)}, then for each set $\ell$ of primes \textup{(}resp.\ odd primes\textup{)} there are zigzags of maps
as in \eqref{eq:932} inducing isomorphisms in the stable range\textup{:}
\begin{equation}%\label{eq:thmaeven}
H_*(C_k(M);\bZ_{(\ell)})\cong H_*(C_j(M);\bZ_{(\ell)}) \quad\text{if}\; \val{\ell}{2k-\chi}=\val{\ell}{2j-\chi}.
\end{equation}
In particular there are integral homology isomorphisms between $C_k(M)$ and $C_{\chi-k}(M)$ in the stable range.
\end{atheorem}

Observe that since these isomorphisms are induced by zigzags of maps, they also give isomorphisms between the cohomology rings of configuration spaces. In Proposition \ref{prop:unstability-odd} we show that when $M$ is an odd-dimensional sphere, this method cannot be used to improve Theorem \ref{thm:a}. In Proposition \ref{prop:unstability-even} we prove that our theorem is sharp when $M$ is an even-dimensional sphere: if $n$ is even and $k$ is in the stable range with respect to homological degree $n-1$, then 
\begin{equation}\label{eq:counterexample}
H_{n-1}(C_k(S^n);\bZ)\cong \tau H_{n-1}(\Omega^n_0S^n;\bZ)\oplus \bZ/(2k-2),
\end{equation}
where $\tau G$ is the torsion of $G$. In particular, if $j$ is also in the stable range:
\[H_{n-1}(C_k(S^n);\bZ_{(\ell)})\cong H_{n-1}(C_j(S^n);\bZ_{(\ell)})\Leftrightarrow (2k-\chi)_\ell = (2j-\chi)_\ell.\]
This generalises the computation of $H_1(C_k(S^2);\bZ)$ (which follows from the presentation of $\pi_1(C_k(S^2))$ given by \cite{FadellVan1962}).

\paragraph{Replication maps.}
Our next result involves a new map between configuration spaces, defined whenever $M$ admits a non-vanishing vector field, which induces some of the homology isomorphisms of Theorem \ref{thm:a}. This map (or rather its effect on $\pi_1$) has been considered before in the case $M=\bR^2$ in the context of the Burau representations of the classical braid groups \cite{BlanchetMarin2007}. It has also appeared in \S 7 of \cite{MartinWoodcock2003}. However, to our knowledge its homological stability properties have not previously been studied. A homomorphism $\pi_1(C_k(M)) \to \pi_1(C_{k+1}(M))$ (which is not induced by a map of spaces) was defined using a similar idea in \cite{BerrickCohenWongWu2006} (see page 283), where it was used to show that the collection $\{\pi_1(C_k(M))\}$ is a crossed simplicial group when $M$ admits a non-vanishing vector field.

This map is especially interesting when $M$ is closed, in which case it allows one to compare configuration spaces which do not admit any stabilisation map. It is also useful when $M$ is open: we will use this map in the case of open manifolds to prove Theorem \ref{thm:e}, which concerns closed manifolds.

Let $v$ be a non-vanishing vector field on $M$ of norm $1$. Define the \emph{$r$-replication map} $\r{r} = \r{r}[v]\colon C_k^\delta(M)\to C_{rk}^\delta(M)$ by adding $r-1$ points near each point of the configuration in the direction of the vector field $v$:
\[
\r{r}[v](\q=\{q_1,\dotsc,q_k\},\epsilon) = \bigl(\bigl\lbrace \mathrm{exp}(\tfrac{j\epsilon}{r} v(q_i)) \bigm| \begin{smallmatrix} i=1,\dotsc,k \\ j=0,\dotsc,r-1 \end{smallmatrix} \bigr\rbrace ,\tfrac{\varepsilon}{2r} \bigr).
\]

\begin{atheorem}\label{thm:b}
Let $r\geq 2$. If $M$ admits a non-vanishing vector field $v$ and $\ell$ is a set of primes each not dividing $r$, then the homomorphism induced by $\r{r}[v]$\textup{:}
\[H_*(C_k^\delta(M);\bZ_{(\ell)}) \lra H_*(C_{rk}^\delta(M);\bZ_{(\ell)})\] 
is an isomorphism in the stable range. If $M$ is not closed, then it is always injective.
\end{atheorem}

\begin{remark}
Observe that the map $\r{r}$ does not induce isomorphisms on $r$-torsion in general. For example take $M$ to be simply-connected and of dimension at least $3$. Then $\pi_1(C_k(M)) \cong \Sigma_k$ and $H_1(C_k(M))\cong \bZ/2$, given by the sign of the permutation. The map $\Sigma_k\to\Sigma_{2k}$ induced by $\r{2}$ on $\pi_1$ sends a permutation $\sigma$ to the concatenation $(\sigma,\sigma)$, whose sign is the square of the sign of $\sigma$, therefore zero. Hence the map induced on first homology by $\r{2}$ is zero. In particular this shows that $\r{2}$ cannot be homotopic to a composition of stabilisation maps.
\end{remark}

\paragraph{Configurations with labels and the intrinsic replication map}
Given a fibre bundle $\theta\colon E\to M$ with path-connected fibres, one can define the \emph{configuration space $\Wo{k}$ with labels in $\theta$} by
\[
\Wo{k} = \{ \{q_1,\dotsc,q_k\} \subset E \mid \theta(q_i)\neq \theta(q_j) \text{ for } i\neq j \}.
\]
Configuration spaces with labels admit stabilisation maps, scanning maps and replication maps (see \cite{KupersMiller2014} and Definition \ref{dConfigLabels} in this article for the stabilisation map, and Section \ref{sUpToDim} for the other two maps) which induce homology isomorphisms in a range, which we call the \emph{stable range with labels in $\theta$}.

To define the replication and the scanning map it is more convenient to use the following alternative model:
\[
\W{k} = \{ (\q,\epsilon,s) \mid (\q,\epsilon)\in C_k^\delta(M),\; s\colon B_{\epsilon/2}(\q)\to E \text{ a section of } \theta \},
\]
where $B_{\epsilon/2}(\q)$ means the (disjoint) union of the $(\epsilon/2)$-balls around $q$ for each $q\in\q$. So a point in this space consists of a configuration $\q$ with prescribed pairwise separation, together with a choice of label on a small contractible neighbourhood of each configuration point.

If $\theta\colon E \to M$ factors through the unit sphere bundle of $TM$ with a map $\varphi\colon E\to S(TM)$, then it is possible to define a new map which we call the \emph{intrinsic replication map \label{intrinsicreplication}} $\ro{r}\colon \W{k}\lra \W{rk}$. It sends the labelled configuration $(\q=\{q_1,\dotsc,q_k\},\epsilon,s\colon B_{\epsilon/2}(\q)\to E)$ to the labelled configuration
\[
\bigl(\bigl\lbrace \mathrm{exp}(\tfrac{j\epsilon}{r} \varphi s(q_i)) \bigm| \begin{smallmatrix} i=1,\dotsc,k \\ j=0,\dotsc,r-1 \end{smallmatrix} \bigr\rbrace ,\tfrac{\varepsilon}{2r}, \text{ restriction of } s \bigr).
\]
In contrast with the (extrinsic) replication map of Theorem \ref{thm:b}, this map is defined for every manifold $M$.

\begin{atheorem}\label{thm:c}
Let $r\geq 2$ and let $\ell$ be a set of primes each not dividing $r$. Then the map $\ro{r}\colon \W{k} \to \W{rk}$ induces isomorphisms on homology with $\bZ_{\ell}$-coefficients in the stable range with labels in $\theta$.
\end{atheorem}

\paragraph{An extension for field coefficients.} The homology of configuration spaces with field coefficients is better understood than the torsion of their integral homology. In fact, complete descriptions of the additive structure of $H_*(C_k(M);\bF)$ were given by \cite{LM:conf} when $\bF$ has characteristic $2$ and by \cite{BCT:conf} when either $\bF$ has characteristic $2$ or $M$ is odd-dimensional. The rational structure was further studied by \cite{FT:conf} and more recently by \cite{Knudsen2014}, who gave a complete description of the rational cohomology ring of $C_k(M)$. From their computations, it follows that the homology with field coefficients always stabilises, unless the manifold is even-dimensional and the field has odd characteristic. These results were proven again by \cite{Church} (in the rational case) and \cite{RW-hs-for-ucs} (in all cases) using homological stability methods (also improving the known stable ranges).

\begin{theorem*}[\cite{LM:conf,BCT:conf,FT:conf,Church,RW-hs-for-ucs,Knudsen2014}]
Let $M$ be a connected, smooth manifold of dimension $n$, let $\bF$ be a field of characteristic $p$ and assume that $p(n-1)$ is even. Then in the stable range we have isomorphisms $H_*(C_k(M);\bF) \cong H_*(C_{k+1}(M);\bF)$.
\end{theorem*}

The last part (\S\ref{sPart2}) of this article addresses the question of homological stability when $p(n-1)$ is odd, in other words for even-dimensional (closed) manifolds and with coefficients in fields of odd characteristic. It does not involve section spaces, but rather uses the result of Theorem \ref{thm:b} in the case of open manifolds $M$ together with an argument similar to that of \cite[\S 9]{RW-hs-for-ucs}.

If $M$ is a closed, connected manifold one can choose a vector field on $M$ which is non-vanishing away from a point $*\in M$. This vector field (suitably normalised) therefore induces an $r$-replication map for configuration spaces on $M\setminus\point$, which induces isomorphisms on homology with $\bZ[\frac{1}{r}]$ coefficients in the stable range by Theorem \ref{thm:b}.

We can fit $C_k(M)$ into a cofibre sequence in which the other two spaces are suspensions of configuration spaces on $M\setminus\point$. We can then define stabilisation maps on the other two spaces using the $r$-replication map and the ordinary stabilisation map, which are isomorphisms on homology localised away from $r$ in the stable range. We will therefore have homological stability for $C_k(M)$, with field coefficients of characteristic coprime to $r$, as long as the square formed by this pair of stabilisation maps commutes. In fact it does \emph{not} commute in general, but the obstruction to commutativity on homology is a single homology class whose divisibility we can calculate. Thus we obtain the following theorem, where
\[
\lambda(k)=\lambda[M](k) \coloneqq \mathrm{min}\{ \srscan(k), \srscan(k-1)+n-1, \sr(rk-i) \mid i=2,\dotsc,r \}.
\]
Here $n$ is the dimension of $M$ and $\sr=\sr[M\smallsetminus\point]$ and $\srscan=\srscan[M\smallsetminus\point]$.

\begin{atheorem}\label{thm:e}
Let $M$ be a closed, connected, even-dimensional smooth manifold. Choose a field $\bF$ of positive characteristic $p$ and let $r\geq 2$ be an integer coprime to $p$ such that $p$ divides $(\chi-1)(r-1)$. Then there are isomorphisms
\[
H_*(C_k(M);\bF) \;\cong\; H_*(C_{rk}(M);\bF)
\]
in the range $*\leq\mathrm{min}(\lambda(k),\lambda(rk))$. 
\end{atheorem}

See Remark \ref{rLambda} for an explanation of how the function $\lambda[M]$ arises, and the remark that if $\sr[M\smallsetminus\point]$ is linear with slope $\leq \dim(M)-1$ and $r,k\geq 2$, then $\lambda[M](k)=\sr[M\smallsetminus\point](k)$.

This theorem also generalises to configuration spaces with labels in a fibre bundle over $M$ with path-connected fibres. See \S\ref{ssTheoremDproof} for the proof for configuration spaces without labels and \S\ref{ssDwithlabels} for a sketch of the generalisation to configuration spaces with labels (Theorem \hyperref[thm:dprime]{$\text{D}^\prime$}).

\begin{remark}
When $M$ is odd-dimensional the conclusion of Theorem \ref{thm:e} follows directly from Theorem \ref{thm:a}. Also, we note that our proof in \S\ref{ssTheoremDproof} also works for fields of characteristic zero: in this case we must asssume that $\chi=1$, but the proof then becomes simpler since the square \eqref{eSquare} commutes up to homotopy (not only on homology). Finally, in the case where the fibre bundle over $M$ factors through the unit sphere bundle $S(TM)\to M$, Theorem \hyperref[thm:dprime]{$\text{D}^\prime$} follows from Theorem \hyperref[thm:cprime]{$\text{C}^\prime$}.
\end{remark}

\paragraph{Combining Theorems \ref{thm:a} and \ref{thm:e}}

Theorem \ref{thm:a} says that in odd dimensions there are at most two stable integral homologies, depending on the parity of the number of points $k$. On the other hand, in even dimensions -- even when taking homology with $\bZ_{(p)}$ coefficients -- there may be infinitely many different stable homologies: one for each possible $p$-adic valuation of $2k-\chi$. In fact this is sharp, as the calculation \eqref{eq:counterexample} shows.

However, the situation is simpler when taking $\bF_p$ coefficients. From the calculation \eqref{eq:counterexample} we see that, when $n$ is even and $k$ is in the stable range with respect to degree $n-1$, we have
\begin{align*}
H_{n-1}(C_k(S^n);\bF_p) &\cong \mathrm{Tor}(H_{n-2}(\Omega_0^n S^n),\bF_p) \oplus (\tau H_{n-1}(\Omega_0^n S^n) \otimes \bF_p) \\
&\qquad \oplus (\bZ/(2k-2) \otimes \bF_p).
\end{align*}
Writing $d$ for the dimension of the first two summands on the right-hand side (which is independent of $k$), it follows that $H_{n-1}(C_k(S^n);\bF_p)$ is either $d$- or $(d+1)$-dimensional depending on whether or not $p$ divides $2k-2$, so there are at most two stable $\bF_p$-homologies in this special case. One can combine Theorems \ref{thm:a} and \ref{thm:e} to prove that this phenomenon holds more generally:

\begin{acorollary}\label{coro:stable-homologies}
Let $M$ be a closed, connected, even-dimensional smooth manifold and let $\bF$ be a field of odd characteristic $p$. Then there are canonical \textup{(}additive\textup{)} isomorphisms
\[
H_*(C_k(M);\bF)\cong H_*(C_j(M);\bF)
\]
under either of the following conditions\textup{:}
\begin{itemize}
\item $\mathrm{min}\{(2k-\chi)_p,(\chi)_p+1\} = \mathrm{min}\{(2j-\chi)_p,(\chi)_p+1\}$,
\item $\chi \equiv 1 \text{\upshape\ mod } p$,
\end{itemize}
in the range $*\leq\mathrm{min}(\lambda(k),\lambda(j))$.
\end{acorollary}

This is proved in \S\ref{ssCorollaryE}, where we also partially recover the known homological stability results for odd-dimensional manifolds and fields of characteristic $2$ or $0$ (see Corollary \ref{coro:stable-homologies-extended}).

\paragraph{Number of stable homologies.}
Homological stability (without an explicit range) for configuration spaces with coefficients in a field $\bF$ can be rephrased as the statement that for each degree $i$, the set $\{\mathrm{dim} H_i(C_j(M);\bF) \mid j=k,\ldots,\infty\}$ contains only one element once $k$ is sufficiently large, in other words, the number
\[
\mathrm{nsh}_i(M;\bF) \;\coloneqq\; \mathrm{lim}_{k\to\infty} \lvert \{ \mathrm{dim} H_i(C_j(M);\bF) \}_{j=k}^\infty \rvert \;\in\; \{1,2,3,\ldots,\infty\}
\]
is equal to $1$. As mentioned earlier, we have $\mathrm{nsh}_i(M;\bF)=1$ whenever either $\mathrm{dim}(M)$ is odd or $\mathrm{char}(\bF)$ is even, and we also have the example that $\mathrm{nsh}_1(S^2;\bF_p)=2$. The above corollary can be viewed as proving that whenever $M$ has non-zero Euler characteristic, $\mathrm{nsh}_i(M;\bF)$ is finite and has the explicit upper bound:
\[
\mathrm{nsh}_i(M;\bF) \leq (\chi)_p + 2
\]
where $p=\mathrm{char}(\bF)$ and $\chi$ is the Euler characteristic of $M$. Moreover, we also have
\[
\mathrm{nsh}_i(M;\bF) = 1 \quad\text{when}\quad \chi \equiv 1 \;\mathrm{mod}\; p.
\]
In particular this means that when $\chi(M)=1$ we have $\mathrm{nsh}_i(M;\bF)=1$ for any field $\bF$, in other words homological stability holds for the sequence $\{C_k(M)\}_{k=1}^\infty$ with coefficients in any field.

\paragraph{Homological periodicity.}

A consequence of Corollary \ref{coro:stable-homologies} is that the sequence of homology groups $\{H_i(C_k(M);\bF)\}_{k=1}^\infty$ for fixed $M$, $\bF$ and $i$ is eventually periodic as $k\to\infty$, as long as $\chi\neq 0$. To see this, note that by Corollary \ref{coro:stable-homologies} it is enough to show that
\begin{equation}\label{eq:min-equals-min}
\mathrm{min}\{(2k-\chi)_p,(\chi)_p+1\} = \mathrm{min}\{(2(k+a)-\chi)_p,(\chi)_p+1\}
\end{equation}
for some natural number $a$ independent of $k$. Note that for any two natural numbers $x$, $y$ we have the inequality $(x+y)_p \geq \mathrm{min}\{(x)_p, (y)_p\}$ and a sufficient condition for equality is that $(x)_p$ and $(y)_p$ are distinct. Considering the cases $(2k-\chi)_p > (\chi)_p$ and $(2k-\chi)_p \leq (\chi)_p$ separately, and applying this fact, one can easily show that the equation \eqref{eq:min-equals-min} holds for $a=p^{(\chi)_p+1}$. Hence we have:

\begin{acorollary}\label{coro:periodicity}
Let $M$ be a closed, connected, even-dimensional smooth manifold with Euler characteristic $\chi\neq 0$ and let $\bF$ be a field of odd characteristic $p$. Then the sequence
\begin{equation}\label{eq:sequence}
\{H_i(C_k(M);\bF)\}_{k=1}^\infty
\end{equation}
is eventually periodic in $k$ as $k\to\infty$, with period equal to $p^{\mathrm{e}_i(M;\bF)}$ for some number $\mathrm{e}_i(M;\bF) \leq (\chi)_p + 1$. Equivalently, there are \textup{(}additive\textup{)} isomorphisms
\[
H_i(C_k(M);\bF) \;\cong\; H_i(C_{k+p^{\chi(M)_p + 1}}(M);\bF)
\]
for $k\gg i$ \textup{(}precisely, in the range $i\leq\lambda(k)$\textup{)}.
\end{acorollary}

This is similar to a result of Nagpal~\cite[Theorem F]{Nagpal2015:arxiv}, who also proves that the sequence \eqref{eq:sequence} is eventually periodic in $k$ as $k\to\infty$ and obtains an explicit period of $p^{(i+3)(2i+2)}$. The difference is that his result also holds when $\chi=0$, but on the other hand he assumes that $M$ is orientable, and his upper bound on the period depends on the homological degree.

Note however that Corollary \ref{coro:stable-homologies} is much stronger than homological periodicity: it implies that the number of stable homologies $\mathrm{nsh}_i(M;\bF)$ is bounded above by $(\chi)_p + 2$, whereas Corollary \ref{coro:periodicity} alone only implies an upper bound of $p^{\chi(M)_p + 1}$.

\paragraph{Acknowledgements} We thank Oscar Randal-Williams for careful reading of an earlier draft of this paper and for enlightening discussions. The paper has also benefited from conversations with Frederick Cohen, Mark Grant, Fabian Hebestreit, Alexander Kupers and Jeremy Miller. We would also like to thank the anonymous referee for helpful comments and corrections.

\section{Homological stability via the scanning map}\label{section:2}

\subsection{Sphere bundles, localisation and fibrewise homotopy equivalences}\label{ss:localisation-etc}
Let $M$ be a connected manifold and $E\to M$ a rank $n$ inner product vector bundle. Let $\dot{E}$ be the fibrewise one point compactification of $E$. The topological bundle $\dot{E}$ is isomorphic to the unit sphere bundle $S(E\oplus\epsilon)$ of the Whitney sum of $E$ and a trivial line bundle. We denote by $\infty$ the point at infinity in each fibre. We denote by $\iota$ the section with value $\infty$ and by $z$ the zero section. During the next paragraph we assume temporarily that $M$ is a compact manifold with boundary.

Let $\Gamma_\partial(\dot{E})\subset \Gamma(\dot{E})$ be the subspace of those sections that take value $\infty$ on the boundary of $M$. Since the fibre of $\dot{E}\to M$ is nilpotent and the pair $(M,\partial M)$ has finitely many non-zero homology groups, then by \cite[Theorem~4.1]{Moller:nilpotent}, each connected component of $\Gamma_\partial(\dot{E})$ is also nilpotent. We may therefore consider, for each set of primes $\ell$, the localisation $\Gamma_\partial(\dot{E})_{(\ell)}$. We may also consider the fibrewise localisation $\dot{E} \to \dot{E}_{(\ell)}$, and \cite[Theorem ~5.3]{Moller:nilpotent} implies that the induced map $\Gamma_\partial(\dot{E})\to \Gamma_\partial(\dot{E}_{(\ell)})$ is a localisation in each component, since $(M,\partial M)$ is a finite relative complex. 

If $M$ is an arbitrary manifold, we can write it as a union $M=\bigcup M^i$ of compact codimension-$0$ submanifolds with boundary. Let $\dot{E}^i$ denote the restriction of $\dot{E}$ to the submanifold $M^i$. The map $\Gamma_c(\dot{E})\to \Gamma_c(\dot{E}_{(\ell)})$, induced by the fibrewise localisation $\dot{E}\to \dot{E}_{(\ell)}$, is then the colimit of the maps $\Gamma_\partial(\dot{E}^i)\to \Gamma_\partial(\dot{E}_{(\ell)}^i)$:
\[\xymatrix{
\Gamma_\partial(\dot{E}^i)\ar[r]\ar[d] & \Gamma_\partial(\dot{E}^{i+1})\ar[r]\ar[d] & \cdots\ar[r] &\Gamma_c(\dot{E})\ar[d] \\
\Gamma_\partial(\dot{E}^i_{(\ell)})\ar[r] & \Gamma_\partial(\dot{E}^{i+1}_{(\ell)})\ar[r] & \cdots\ar[r] &\Gamma_c(\dot{E}_{(\ell)}).
}\]
Since the vertical maps induce (componentwise) isomorphisms on homology with $\bZ_{(\ell)}$-coefficients, so does their colimit.

%\[\xymatrix{
%\Gamma_c(TM^i_{(p)})\ar[r]\ar[d]^{f_i} & \Gamma_c(TM^{i+1}_{(p)})\ar[r]\ar[d]^{f_{i+1}} & \ldots &\Gamma_c(TM_{(p)})\ar[d]^{f} \\
%\Gamma_c(TM^i_{(p)})\ar[r] & \Gamma_c(TM^{i+1}_{(p)})\ar[r] & \ldots &\Gamma_c(TM_{(p)})
%}\]
A bundle endomorphism $f$ of $\dot{E}_{(\ell)}$ is \emph{compactly supported} if $f\circ \iota = \iota$ outside a compact subset of $M$. We denote by $\End_c^r(\dot{E}_{(\ell)})$ the space of compactly supported endomorphisms which induce on fibres maps of degree $r$. We denote by $\endo^r(\dot{E}_{(\ell)})$ the bundle of pairs $(x,f_x)$, where $x\in M$ and $f_x\colon (\dot{E}_{(\ell)})_x\to (\dot{E}_{(\ell)})_x$ is a map of degree $r$. By definition $\End_c^r(\dot{E}_{(\ell)}) = \Gamma_c(\endo^r(\dot{E}_{(\ell)}))$. By Theorem 3.3 of \cite{Dold}, if $r$ is a unit in $\bZ_{(\ell)}$, then any endomorphism in $\End_c^r(\dot{E}_{(\ell)})$ admits a fibrewise homotopy inverse. Postcomposition with it induces a homotopy equivalence between path-components
\[\Gamma_c(\dot{E}_{(\ell)})_k\lra \Gamma_c(\dot{E}_{(\ell)})_{[f](k)},\]
where $[f]$ denotes the map induced by $f$ on $\pi_0\Gamma_c(\dot{E}_{(\ell)})$.

We summarize the discussion so far in the following lemma:
\begin{lemma}\label{lemma:localisation} If $r$ is a unit in $\bZ_{(\ell)}$, $f\in \End_c^r(\dot{E}_{(\ell)})$ and $[f](k)$ is an integer, then the zigzag
\[\Gamma_c(\dot{E})_k\lra \Gamma_c(\dot{E}_{(\ell)})_k\lra \Gamma_c(\dot{E}_{(\ell)})_{[f](k)}\longleftarrow \Gamma_c(\dot{E})_{[f](k)},\]
where the middle map is given by post-composition with $f$, induces an isomorphism on homology with $\bZ_{(\ell)}$-coefficients.
\end{lemma}

\begin{remark}
Note that if $\ell = \emptyset$, then $(-)_{(\ell)}$ is rationalisation (also denoted $(-)_{(0)}$), whereas if $\ell = \Spec \bZ$, the set of all primes, this localisation is the identity, i.e., we are not localising at all.
\end{remark}

\subsection{The degree of a section}\label{ss:degree}
Let $\beta$ be a compactly supported section of $\pi\colon \dot{T}M\to M$, and let $\Th(\beta)$ be the Thom class in $H^n(\dot{T}M;\pi^*{\mathcal O})$, where $\mathcal O$ is the orientation sheaf of $M$. The \emph{$\beta$-degree} of a compactly supported section $\alpha$ is
\[\deg_\beta(\alpha) = \alpha^*(\Th(\beta))^\vee\in H_0(M;\bZ),\]
the Poincare dual in $M$ of $\alpha^*\Th(\beta)\in H^n_c(M;\mathcal{O})$. If $M$ is orientable, then $\Th(\beta)$ is the Poincare dual of $\beta_*[M]\in H_n(\dot{T}M;\bZ)$, and $\deg_\beta(\alpha)$ is also equal to the intersection product of $\alpha_*[M]$ and $\beta_*[M]$ in $\dot{T}M$. We will write $\deg$ for $\deg_z$, where $z$ is the zero section of $\dot{T}M$. 

Assume now that $M$ is closed and orientable. The Gysin sequence for the sphere bundle $S^n\stackrel{i}{\to} \dot{T}M\stackrel{\pi}{\to} M$ splits an exact sequence
\begin{equation}\label{eq:362}
0\lra H_n(S^n;\bZ) \overset{i_*}{\lra}  H_n(\dot{T}M;\bZ) \overset{\pi_*}{\lra} H_n(M;\bZ)\lra 0.
\end{equation}
The zero section $z\colon M\to \dot{T}M$ is an inverse of $\pi$, so the group $H_n(\dot{T}M)\cong \bZ\oplus\bZ$ is generated by $i_*[S^n]$ and $z_*[M]$. The fibres over two different points give two disjoint representatives of $i_*[S^n]$, therefore $i_*[S^n]\cap i_*[S^n] = 0$. On the other hand, the intersection of the zero section with itself is the Euler characteristic $\chi$ of $M$. And it is also clear that the intersection of $i_*[S^n]$ and $z_*[M]$ consists of a single point. The intersection products of $4k$ (resp.\ $4k+2$) dimensional manifolds are symmetric (antisymmetric). Therefore we have:
\begin{lemma}\label{lemma:cup-product} If $M$ is connected, closed, orientable and of dimension $n$, then the intersection pairing of $\dot{T}M$ with respect to the above basis is given by 
\begin{align*}
\begin{pmatrix}0 & 1 \\  (-1)^n & \chi\end{pmatrix}. %&\quad \text{if $\dim(M)$ is even} 
%\left(\begin{array}{cc}0 & 1 \\ -1 & \chi(E)\end{array}\right) &\quad \text{if $\dim(M)$ is odd} \\
 \end{align*}
If $\alpha$ is a section of $\pi$, then $\alpha_*[M] = (\deg(\alpha)-\chi,1)$ in this basis.
\end{lemma}
\begin{proof}
For the second claim, observe that $\alpha$ is an inverse of $\pi$ too, so the second component of $\alpha_*[M]$ is the same as the second component of $z_*[M]$. The first component is obtained from the following equation:
\begin{align*}%\label{eq:08}
\deg(\alpha) &= \alpha_*[M]\cap z_*[M] = (a,1)\begin{pmatrix}0 &  1 \\ (-1)^n & \chi\end{pmatrix}\begin{pmatrix} 0\\1\end{pmatrix} = a+\chi.\qedhere
\end{align*}
\end{proof}
The Gysin sequence \eqref{eq:362} applied to the localised bundle $S^n_{(\ell)}\to \dot{T}M_{(\ell)}\to M$ shows that $H_n(\dot{T}M_{(\ell)})\cong \bZ_{(\ell)}\oplus \bZ$. The first factor is generated as a $\bZ_{(\ell)}$-module by the fundamental class of the fibre, and the second factor is generated by the image of the fundamental class of $M$ under the zero section. The following definition extends the notion of degree to sections of fibrewise localised sphere bundles.
\begin{df} The degree of a section $\alpha$ of $\dot{T}M_{(\ell)}$, denoted $\deg(\alpha)$, is the value $a+\chi\in \bZ_{(\ell)}$, where $(a,1) = \alpha_*[M]$ in our preferred basis. 
\end{df}

\subsection{Fibrewise homotopy equivalences of many degrees}

Let $V_2(E\oplus \epsilon)$ be the fibrewise Stiefel manifold of $E\oplus\epsilon$. If $\sigma$ is a section of $\Gamma(V_2(E\oplus\epsilon)_{(\ell)})$ we denote by $\sigma_0$ the image of $\sigma$ under the localisation of the map that forgets the second vector:
\[\varphi_{(\ell)}\colon \Gamma(V_2(E\oplus\epsilon)_{(\ell)})\lra \Gamma(S(E\oplus\epsilon)_{(\ell)}).\]
We denote by $\Gamma_c(V_2(E\oplus\epsilon)_{(\ell)})$ the space of sections $\sigma$ such that $\sigma_{0}$ is compactly supported.

\begin{lemma}\label{lemma:ultimate}
Let $E$ be a real inner product bundle over a manifold $M$. There is a bundle map
\begin{align*}
V_2(E\oplus \epsilon) &\lra \endo^r(\dot{E})
\end{align*}
for each $r\in\bZ$ and therefore, for each set of primes $\ell$, there are maps
\begin{align*}
\Phi_r^\ell\colon \Gamma_c(V_2(E\oplus\epsilon)_{(\ell)}) &\lra \End^r_c(\dot{E}_{(\ell)})
\end{align*}
which are natural with respect to pullback of bundles. If $M$ is closed and $E=TM$, then $\Phi_r^\ell(\sigma)$ sends sections of degree $k$ to sections of degree $r(k-\deg(\sigma_0)) + \deg(\sigma_0)$.\nopagebreak
\end{lemma}
\begin{proof} A $2$-frame in $V_2(E\oplus \epsilon)$ determines a linear embedding $\bR^2\to (E\oplus\epsilon)_x$. If we denote by $V$ its orthogonal complement, we obtain canonical isomorphisms $\bR^2\oplus V\cong (E\oplus\epsilon)_x$ which induce canonical isomorphisms $S^1*S(V)\cong S(E\oplus\epsilon)$. This allows to define a degree $r$ map
\[\xymatrix{S(E\oplus\epsilon)_x \cong S^1*S(V) \ar[rr]^{e^{2\pi i r}*\Id} && S^1*S(V)\cong S(E\oplus\epsilon)_x.}\]
After fibrewise localizing and taking sections, one obtains the second map. Observe that the above map fixes the first vector in the $2$-frame, hence the image of a section in $\Gamma_c(V_2(E\oplus\epsilon)_{(\ell)})$ will fix the section $\iota$ outside a compact subset.

By construction, $f^*(\Phi_r^\ell(\sigma)) = \Phi_r^\ell(f^*(\sigma))$, so these maps are natural. Similarly, observe that
\begin{equation}\label{eq:09}
\End_c^r(\dot{E}_{(\ell)})\times \Gamma(\dot{E}_{(\ell)})\lra \Gamma(\dot{E}_{(\ell)})
\end{equation}
is also natural with respect to pullback of bundles.

Now we describe the effect of $\phi_r:=\Phi_r^\ell(\sigma)$ on components of $\Gamma(\dot{T}M_{(\ell)})$ when $M$ is closed. Assume first that $M$ is orientable, in which case $\dot{T}M$ is also orientable and Lemma \ref{lemma:cup-product} applies. First we identify the induced map $(\phi_{r})_*\colon H_n(\dot{T}M_{(\ell)})\to H_n(\dot{T}M_{(\ell)})$. Since $\phi_r(\sigma_0) = \sigma_0$, we have
\[(\phi_r)_*(\deg(\sigma_0)-\chi,1) = (\deg(\sigma_0)-\chi,1).\]
On the other hand, $\phi_r$ acts on the fibre over a point as a map of degree $r$, hence 
\[(\phi_r)_*(1,0) = (r,0).\]
From this we deduce that $(\phi_r)_*$ has the form
\[\begin{pmatrix}
r & -(r-1)(\deg(\sigma_0)-\chi) \\
0 & 1
\end{pmatrix},\]
hence, for an arbitrary section $\alpha$, we have that 
\begin{align*}
(\deg(\phi_r(\alpha)_*[M])-\chi,1) &= \phi_{r}(\alpha)_*[M] = (\phi_r)_*(\alpha_*[M])\\
&= ( r(\deg(\alpha)-\chi) - (r-1)(\deg(\sigma_0)-\chi),1)
\end{align*}
and so $\deg(\phi_r(\alpha)) = r\deg(\alpha) - (r-1)\deg(\sigma_0) = r(\deg(\alpha)-\deg(\sigma_0)) + \deg(\sigma_0)$.

Assume now that $M$ is non-orientable. We take then the orientation cover $f\colon \tilde{M}\to M$. If $s$ is a section of $\dot{T}M_{(\ell)}$ and $\sigma$ is a section of $V_2(TM\oplus\epsilon)_{(\ell)}$, we can pull back both sections along $f$ to obtain a section $f^*s$ of $\dot{T}\tilde{M}_{(\ell)}$ and a section $f^*\sigma$ of $V_2(T\tilde{M}\oplus\epsilon)_{(\ell)}$. Then, because $f$ is a double cover, $\deg(f^*s) = 2\deg(s)$, and by the naturality of $\phi_r$ and \eqref{eq:09} we have that $\Phi_r^\ell(f^*\sigma)(f^*s) = f^*(\Phi_r^\ell(\sigma)(s))$. On the other hand, since $\tilde{M}$ is orientable, by the previous paragraph we know that $\deg(\Phi_r^\ell(f^*\sigma)(s)) = r\deg(f^*s)-(r-1)\deg(f^*\sigma_0)$. As a consequence:
\begin{align*}
2\deg(\Phi_r^\ell(\sigma)(s)) &= \deg(f^*(\Phi_r^\ell(\sigma)(s))) \\
&= \deg(\Phi_r^\ell(f^*\sigma)(f^*(s)) \\
&= r(\deg(f^*s)-\deg(f^*\sigma_0)) + \deg(f^*\sigma_0) \\
&= r(2\deg(s)-2\deg(\sigma_0)) + 2\deg(\sigma_0).\qedhere
\end{align*}
\end{proof}

We now face the following lifting problem:
\[\xymatrix{
&& V_2(TM\oplus\epsilon)_{(\ell)}\ar[d]^{\varphi_{(\ell)}} \\
M\ar[rr]^{\sigma_0}\ar@{-->}[urr]^{\sigma} && S(TM\oplus\epsilon)_{(\ell)},
}\]
\begin{proposition}\label{prop:lifts}
Let $M$ be closed and of dimension $n\geq 2$. When $n$ is odd every diagram has a lift, and when $n$ is even the diagram has a lift precisely for sections $\sigma_0$ of degree $\chi/2$ \textup{(}which exist whenever $\chi$ is even or $2\notin \ell$\textup{)}.
%whereas when $n$ is even only sections of degree $\chi/2$ \textup{(}whenever they exist\textup{)} have a lift.
\end{proposition}
\begin{proof}
The above problem is equivalent to find a section of the pullback $\eta_{(\ell)}$ of $\varphi_{(\ell)}$ along $\sigma_0$, which is an $S_{(\ell)}^{n-1}$-bundle over an $n$-dimensional manifold. If $n$ is odd, $\eta_{(\ell)}$ has always a section, hence in that case every section $\sigma_0$ admits a lift. If $n$ is even, the complete obstruction (if $M$ is orientable) is the Euler class $e(\eta_{(\ell)})$ of $\eta_{(\ell)}$. We proceed to compute it:

Assume first that $M$ is orientable and $\ell=\Spec \bZ$. The bundle $\eta$ is the unit sphere bundle of $\sigma_0^*T^v(TM\oplus\epsilon)$, whose Euler number can be computed by taking the self-intersection of its zero section in the fibrewise one point compactification of $\sigma_0^*T^v(TM\oplus\epsilon)$, which is precisely $S(TM\oplus\epsilon)$. As the zero section of $\sigma_0^*T^v(TM\oplus\epsilon)$ is $\sigma_0$, we have that (we denote by $x^\vee$ the Poincar\'e dual of $x$)
\begin{align}\label{eq:923}
e(\eta)^\vee &= \sigma_0[M]\cap\sigma_0[M] \\
&= 
(\deg(\sigma_0) - \chi,1)
\begin{pmatrix}0 &  1 \\ 1 & \chi\end{pmatrix} 
\begin{pmatrix}\deg(\sigma_0)-\chi \\  1\end{pmatrix}
= 2\deg(\sigma_0) - \chi.
\end{align}
Hence a section admits a lift if and only if $\deg(\sigma_0) = \chi/2$. 

Let us assume now that $M$ is orientable and $\ell$ is a proper subset of $\Spec \bZ$. In this case, the above computation is no longer valid, as it relies on a geometric interpretation of the Euler class. We will first compute the Euler class $e$ of $\varphi_{(\ell)}$:
\begin{align*}
e(\eta_{(\ell)}) &= \sigma_0^*(e)^\vee = e\frown \sigma_0[M] = e^\vee\cap \sigma_0[M],
\end{align*}
and therefore, if $e^\vee = (a,b)$ in the basis described before, it holds that
\begin{align*}
e(\eta_{(\ell)})^\vee = \sigma_0^*(e)^\vee &= 
(a,b)
\begin{pmatrix}0 &  1 \\ 1 & \chi\end{pmatrix} 
\begin{pmatrix}\deg(\sigma_0)-\chi\\1\end{pmatrix}
= a+b\deg(\sigma_0).
\end{align*}
This, together with \eqref{eq:923} (which holds for integral values), implies that $e^\vee = (-\chi,2)$, and therefore that $\sigma_0^*(e)^\vee = 2\deg(\sigma_0)-\chi$. Hence, after localising we obtain that only sections of degree $\chi/2$ admit a lift.

Finally, let $M$ be non-orientable and let $f\colon \tilde{M}\to M$ be the orientation cover of $M$. Then $ \deg_{f^*\sigma_0}(f^*\sigma_0) = 2\deg_{\sigma_0}(\sigma_0)$ and the Euler characteristic of $\tilde{M}$ is $2\chi$, so $\deg_{\sigma_0}(\sigma_0) = 0$ if and only if $(2\chi)/2 = \deg(f^*\sigma_0) = 2\deg(\sigma_0)$. Hence only sections of degree $\chi/2$ have lifts.
\end{proof}

Combining Lemma \ref{lemma:ultimate} and Proposition \ref{prop:lifts}, we have the following (see immediately above Lemma \ref{lemma:localisation} for the notation $[f]$).

\begin{corollary}\label{coro:existence}
Let $\ell$ be a collection of primes. Suppose that $\dim(M)$ is odd and we are given any $r,d\in\bZ$. Then there exists an endomorphism $f\in\End_c^r(\dot{T}M_{(\ell)})$ with $[f](k) = r(k-d) +d$. Suppose that $\dim(M)$ is even and we are given any $r\in\bZ$. Assume also that $\chi/2 \in \bZ_{(\ell)}$, i.e., either $\chi$ is even or $2\not\in\ell$. Then there exists an endomorphism $f\in\End_c^r(\dot{T}M_{(\ell)})$ with $[f](k) = r(k-\chi/2)+\chi/2$.
\end{corollary}

%\rnote{I wrote ``for $k\in\bZ$'' since we only know this formula for integral degrees of sections, not necessarily sections with arbitrary degree in $\bZ_{(p)}$ (right?).}

\subsection{Proof of Theorem \ref{thm:a}}

As promised in the introduction, in the next three propositions we will provide the middle map in the zigzag \eqref{eq:932}, from which the assertions in the theorem will follow by virtue of Lemma \ref{lemma:localisation}.

\begin{proposition}
If $\dim M$ is odd and $\ell$ is any set of odd primes, then there are homotopy equivalences
\begin{align*}
\Gamma_c(\dot{T}M)_{k}&\overset{\simeq}{\lra} \Gamma_c(\dot{T}M)_{k+2} \\
\Gamma_c(\dot{T}M_{(\ell)})_{k}&\overset{\simeq}{\lra} \Gamma_c(\dot{T}M_{(\ell)})_{k+1}.
\end{align*}
\end{proposition}
\begin{proof}
By Corollary \ref{coro:existence} and Theorem 3.3 of \cite{Dold} (cf.\ the discussion above Lemma \ref{lemma:localisation}), there exist homotopy equivalences
\[\Gamma_c(\dot{T}M_{(\ell)})_{k}\lra \Gamma_c(\dot{T}M_{(\ell)})_{r(k-d)+d}\]
for all integers $r$ and $d$ such that $r\notin \ell\bZ$ (if $\ell = \Spec \bZ$), then the condition becomes that $r=1,-1$). Observe first that if $r$ is odd, then $k$ and $r(k-d)+d$ have the same parity. Hence if $k$ and $j$ have different parity then a homotopy equivalence
$\Gamma_c(\dot{T}M_{(\ell)})_{k}\lra \Gamma_c(\dot{T}M_{(\ell)})_{j}$ as above exists only if $2\notin \ell$. 

Taking $r=-1$ and $d=k+1$ (and $\ell=\Spec \bZ$) we obtain the first map. Taking $r=2$ and $d=k-1$ we obtain the second map.
\end{proof}

\begin{proposition}
Suppose that $\dim M$ is even, $\ell$ is a set of primes and $k$, $j$ are integers such that $(2k-\chi)_\ell = (2j-\chi)_\ell$. If $\chi$ is odd, assume also that $2\not\in\ell$. Then there is a
%$\chi$ is even (resp.\ odd), $\ell$ is a set of primes (resp.\ odd primes) and $k$ and $j$ have the same $\ell$-adic valuation, then there is a
zigzag of homotopy equivalences between $\Gamma_c(\dot{T}M_{(\ell)})_{k}$ and $\Gamma_c(\dot{T}M_{(\ell)})_{j}$. If $j=\chi-k$, there is a homotopy equivalence $\Gamma_c(\dot{T}M)_{k} \to \Gamma_c(\dot{T}M)_{j}$.
\end{proposition}

\begin{proof}
By Corollary \ref{coro:existence} and Theorem 3.3 of \cite{Dold}, there exists, for each integer $r$ with trivial $\ell$-adic valuation, a homotopy equivalence
\[
\Gamma_c(\dot{T}M_{(\ell)})_{k} \longrightarrow \Gamma_c(\dot{T}M_{(\ell)})_{r(k-\chi/2)+\chi/2}.
\]
Equivalently, there exists a homotopy equivalence $\Gamma_c(\dot{T}M_{(\ell)})_{k} \to \Gamma_c(\dot{T}M_{(\ell)})_{j}$ whenever $(2j-\chi)=r(2k-\chi)$ for some $r$ such that $(r)_\ell = 0$.
Now let $k$ and $j$ be the two given integers, let $l=(2k-\chi)_\ell = (2j-\chi)_\ell$ and define $m = \prod_{p\in \ell}{p^{l(p)}}$. Note that the integer $\frac{1}{m}(2k-\chi)(2j-\chi)+\chi$ is always even, so we have
\[
\tfrac{1}{m}(2k-\chi)(2j-\chi) = (2h-\chi)
\]
for some integer $h$. Since $\frac{1}{m}(2j-\chi)$ and $\frac{1}{m}(2k-\chi)$ both have trivial $\ell$-adic valuation, the previous discussion implies that there are homotopy equivalences $\Gamma_c(\dot{T}M_{(\ell)})_{k} \to \Gamma_c(\dot{T}M_{(\ell)})_{h}$ and $\Gamma_c(\dot{T}M_{(\ell)})_{j} \to \Gamma_c(\dot{T}M_{(\ell)})_{h}$.

For the last claim, observe that if $M$ has even Euler characteristic $\chi$, taking $r=-1$ and $\ell=\Spec \bZ$ in Corollary \ref{coro:existence}, we obtain a homotopy equivalence $\Gamma_c(\dot{T}M)\to \Gamma_c(\dot{T}M)$ (without localising) that sends sections of degree $k$ to sections of degree $\chi-k$. In fact, such a homotopy equivalence exists regardless of whether $\chi$ is even or odd: it may be obtained by postcomposition with the antipodal map $\dot{T}M\to \dot{T}M$.
\end{proof}

%\begin{proof}
%We use Proposition \ref{prop:lifts} and Lemma \ref{lemma:ultimate}, to construct, for each integer $r$ with trivial $\ell$-adic valuation, a zigzag of homotopy equivalences 
%\begin{align*}
%\Gamma_c(\dot{T}M_{(\ell)})_{k}&\lra \Gamma_c(\dot{T}M_{(\ell)})_{rk-(r-1)\chi/2}.%\quad \text{if $\chi$ is even and $p\neq r$}
%\end{align*}
%It is clear that $2k-\chi$ and $2(rk-(r-1)\chi/2) - \chi = r(2k-\chi)$ have the same $\ell$-adic valuation if $p\nmid r$ for each $p\in \ell$. If $2k-\chi$ and $2j-\chi$ have the same $\ell$-adic valuation $l$, and $s=(2k-\chi)/\prod_{p\in \ell}{p^{l(p)}}$ and $r=(2j-\chi)/\prod_{p\in \ell}{p^{l(p)}}$, then $rk - (r-1)\chi/2 = sj + (s-1)\chi/2$, and therefore there is a zigzag of $\bZ_{(\ell)}$-homotopy equivalences for all $\ell$ such that $r,s\notin \ell\bZ$ between $\Gamma_c(\dot{T}M_{(\ell)})_k$ and $\Gamma_c(\dot{T}M_{(\ell)})_j$. 
%\end{proof}

If $\dim M$ is odd, Theorem \ref{thm:a} can only be improved when $2\in \ell$ and $k,j$ have different parity. To face this problem using a zigzag as in \eqref{eq:932}, we need to find a fibrewise homotopy equivalence $f$ of $\dot{T}M_{(2)}$ whose action on components of the section space changes the parity. The following proposition deals with this case (cf.\ \cite{Hansen,Hansen-manifolds}). We note that its proof can be used to recover Theorem \ref{thm:a} when $M$ is a sphere, as well as the integral homology isomorphisms of Theorem \ref{thm:a} when $M$ is an arbitrary odd-dimensional manifold (see Remark \ref{rmk:Recovering-thma}).

\begin{proposition}\label{prop:unstability-odd}
For $n$ odd and $2\in \ell$, the fibre bundle $\dot{T}S^n_{(\ell)}$ admits fibrewise homotopy equivalences that change the parity of the sections if and only if $n = 1,3,7$. If $\dim M = 1,3,7$, then $\dot{T}M$ admits a fibrewise homotopy equivalence that sends sections of degree $k$ to sections of degree $k+1$.
\end{proposition}
\begin{proof}
Spheres are stably parallelisable, and therefore $\dot{T}S^n$ is trivial, being the unit sphere bundle of $TS^n\oplus\epsilon$. After choosing a trivialisation, a fibrewise endomorphism $f$ of $\dot{T}S^n_{(\ell)}$ of degree $r$ is the same as a map
\[f^{\mathrm{t}}\colon S^n\lra \map_r(S^n_{(\ell)},S^n_{(\ell)}).\]
Since $n$ is odd, the Euler characteristic $\chi$ is $0$, so we may trivialise $\dot{T}S^n$ so that the zero section corresponds to a trivial section of the product bundle $S^n\times S^n_{(\ell)}$. Hence the matrix of $H_n(f)$ with the basis considered in Lemma \ref{lemma:cup-product} is of the form 
\[\begin{pmatrix}
r & b \\
0 & 1
\end{pmatrix},\]
where $b\in \bZ_{(\ell)}$ is the degree of the composition of $f^{\mathrm{t}}$ with the evaluation map. Such an $f$ sends sections of degree $k$ to sections of degree $rk+b$. For $f$ to be a fibrewise homotopy equivalence, $r$ must be odd (and moreover equal to $\pm 1$ when $\ell = \Spec \bZ$), and therefore $b$ has to be odd as well, because $[f](k) = rk+b$ and we want $k$ and $[f](k)$ to have different parity. Therefore, we need to solve the lifting problem
\[\xymatrix{
&& \map_r(S^n_{(\ell)},S^n_{(\ell)})\ar[d] \\
S^n \ar[rr]^{b} \ar@{-->}[urr]&& S^n_{(\ell)}
}\]
where the vertical map is the evaluation map and the horizontal map is some map of odd degree $b$. The single (and therefore complete) obstruction to the existence of a lift is a class in $H^n(S^n;\pi_{n-1}\Omega_r^{n}S^n_{(\ell)})\cong \pi_{n-1}\Omega_r^{n}S^n_{(\ell)}$. This class is $b$ times the image of the generator of $\pi_nS^n_{(\ell)}$ under the boundary homomorphism in the long exact sequence of homotopy groups. The boundary homomorphism is computed in \cite[Theorem 3.2]{Whitehead-products} with a correction in \cite{Whitehead-certain}, who proved that under the identification $\pi_i(\Omega^n_rS^n_{(\ell)})\cong \pi_{n+i}(S^n_{(\ell)})$, it corresponds to taking the Whitehead product with $-r\iota$, where $\iota$ is a generator of $\pi_n(S^n)$. Therefore our obstruction is $b[-r\iota,\iota]$. Because the Whitehead product is graded-commutative, $[\iota,\iota]$ has order two, so $-br[\iota,\iota] = [\iota,\iota]$. The EHP sequence shows that the vanishing of this class is equivalent to the existence of elements of Hopf invariant one in $\pi_{2n+1}(S^{n+1}_{(\ell)})$, which exist if and only if $n=0,1,3,7$ \cite{Adams}.

If $M$ is an arbitrary manifold of dimension $1$, $3$ or $7$, we first choose an open disc $D$ in $M$. Then we take $\ell=\emptyset, r=1, b=1$, and we consider the one-point compactification $\dot{D}$ of $D$. By the previous part, there is a fibrewise endomorphism $f$ of fibrewise degree $1$ of $\dot{T}\dot{D}$. Without loss of generality, we may assume that its value on the basepoint is the identity. Then we can extend this fibrewise endomorphism to the whole of $M$ by defining $x\mapsto (\Id\colon \dot{T}_xM\to \dot{T}_xM)$ if $x\notin D$. The extension sends sections of degree $k$ to sections of degree $rk+b = k+1$. By Theorem 3.3 of \cite{Dold}, it is a fibrewise homotopy equivalence.
\end{proof}

\begin{remark}\label{rmk:Recovering-thma}
The above proof recovers Theorem \ref{thm:a} in certain cases, as follows. Given some $k\in\bZ$, we take $r=2$ and $b=1-k$ to obtain homotopy equivalences $\Gamma_c(\dot{T}S^n_{(\ell)})_k \simeq \Gamma_c(\dot{T}S^n_{(\ell)})_{k+1}$ with $\ell$ the set of all odd primes. This recovers the isomorphisms \eqref{eq:thmaodd-Z12} when $M=S^n$. For the isomorphisms \eqref{eq:thmaodd-Z}, we take $r=1$ and $b=2$ (and $\ell$ the set of all primes, i.e., we do not localise) and extend the resulting fibrewise homotopy equivalence, defined over a disc in $M$, to the whole of $M$ by the identity. This gives homotopy equivalences $\Gamma_c(\dot{T}M)_k \simeq \Gamma_c(\dot{T}M)_{k+2}$ for any $M$.
\end{remark}

%\rnote{Maybe add a remark that when $p$ is odd, we can take $r=p+1$ and $b=0$. Then we can certianly solve the lifting problem (taking the constant map) to get a fibrewise homotopy equivalence that acts on degrees of sections by $k\mapsto (p+1)k$. In particular it gives a homotopy equivalence between $\Gamma_c(\dot{T}S^n_{(p)})_1$ and $\Gamma_c(\dot{T}S^n_{(p)})_{p+1}$.}

Proposition \ref{prop:unstability-odd} also shows that, when $M$ is an odd-dimensional sphere, Theorem \ref{thm:a} cannot be improved by finding a fibrewise homotopy equivalence of $\dot{T}S^n_{(\ell)}$. The following proposition (cf.\ \cite[Theorem 3.1]{Hansen} and \cite{Hansen-manifolds}) generalises the computation of $H_1(C_k(S^2);\bZ)$ (which follows from the presentation of $\pi_1(C_k(S^2))$ given by \cite{FadellVan1962}) and shows that Theorem \ref{thm:a} is sharp when $M$ is an even-dimensional sphere. 

\begin{proposition}\label{prop:unstability-even}
If $n$ is even and $k$ belongs to the stable range with respect to homological degree $n-1$, then
%\[H_{n-1}(C_k(S^n);\bZ_{(p)})\cong H_{n-1}(C_j(S^n);\bZ_{(p)})\Leftrightarrow (k-\chi/2)_p = (j-\chi/2)_p.\]
\[
H_{n-1}(C_k(S^n);\bZ)\cong \tau H_{n-1}(\Omega^n_0S^n)\oplus \bZ/(2k-2),
\]
%Moreover, their orders differ by a factor of $\frac{(k-\chi/2)_p}{(j-\chi/2)_p}$. 
where $\tau G$ is the torsion of $G$. If $M$ is a closed manifold of even dimension, then any fibrewise endomorphism of $\dot{T}M_{(\ell)}$ of degree $r\neq 0$ sends sections of degree $k$ to sections of degree $r(k-\chi/2)+\chi/2$.
\end{proposition}
\begin{proof}
The target of the scanning map in this case is $\Gamma(\dot{T}S^n)$, and since $S^n$ is stably parallelisable, $\dot{T}S^n$ can be trivialised. The trivialisation gives a homotopy equivalence $\Gamma(\dot{T}S^n)\to \map(S^n,S^n)$ that sends sections of degree $k$ to maps of degree $k-\chi/2$, where $\chi$ is the Euler characteristic of $S^n$ (this corresponds to a change of basis in Lemma \ref{lemma:cup-product}, and is stated explicitly in \cite[Proposition~3.6]{BM}). We let now $r=k-\chi/2$.

The space of maps fits into the evaluation fibration
\begin{equation}\label{eq:510}
\Omega_{r}^nS^n\lra \map_{r}(S^n,S^n)\lra S^n
\end{equation}
for which we may consider the corresponding Wang sequence
\[\cdots\lra H_0(\Omega_r^nS^n)\overset{\delta_r}{\lra} H_{n-1}(\Omega_r^n S^n)\lra H_{n-1}(\map_r(S^n,S^n))\lra 0.\]
The map named $\delta_r$ is the transgression
\[
H_{n}(S^n;H_0(\Omega_r^nS^n)) \lra H_{0}(S^n;H_{n-1}(\Omega_r^n S^n))
\]
in the Serre spectral sequence of the evaluation fibration, and therefore under the identification $H_{n}(S^n;H_0(\Omega_r^nS^n)) = H_n(S^n)$ it fits into the commutative diagram
\[\xymatrix{
\pi_n(S^n)\ar[d]\ar[r]^-{\partial_r} & \pi_{n-1}(\Omega_r^n S^n) \ar[d]\\
H_n(S^n)\ar[r]^-{\delta_r} & H_{n-1}(\Omega_r^n S^n),
}\]
where $\partial_r$ is the boundary homomorphism in the long exact sequence of homotopy groups. As recalled in the previous proof, $\partial_r$ was identified by Whitehead as the adjoint of the operation of taking Whitehead product with $-r\iota$, where $\iota$ is the generator of $S^n$. Additionally, the left vertical arrow is an isomorphism and the rightmost vertical arrow sends the class $[\iota,\iota]$ to the Browder square of the generator of $H_0(\Omega^n_rS^n)$ (see Remark 1.2 in the third chapter of \cite{CLM}) (here we are using a canonical identification of $\Omega_r^nS^n$ and $\Omega_1^nS^n$ to define the Browder square and the Whitehead product). 

We claim, when $n$ is even, that this Browder square has infinite order and is divisible by two (but not by four). To see this, consider the scanning map $C(\bR^n)\to \Omega^nS^n$. Both spaces have an action of the little $n$-discs operad, and the scanning map is equivariant with respect to this action, hence it takes Browder squares to Browder squares. The adjoint $\alpha$ of the class $\iota\in \pi_{n}(S^n)$ is a generator of $\pi_0(\Omega_1^nS^n)$, so the Browder square $\lambda(\beta,\beta)$ of the Hurewicz image $\beta$ of $\alpha$ lives naturally (that is, before using the identification between the different components of $\Omega^nS^n$) in $H_{n-1}(\Omega_2^nS^n)$. We will now describe this class.

The generator $\gamma$ of $H_0(C_1(\bR^n))$ is mapped under the scanning map to $\beta$, hence $\lambda(\gamma,\gamma)$ is mapped to $\lambda(\beta,\beta)$. The class $\lambda(\gamma,\gamma)\in H_{n-1}(C_2(\bR^n))$ corresponds to moving one of the points around the other point in all possible directions, parametrised by $S^{n-1}$. The inclusion of the space $\bR P^{n-1}$ of antipodal points in $S^{n-1}$ into $C_2(\bR^n)$ is a homotopy equivalence, and our class $\lambda(\gamma,\gamma)$ is the image of the fundamental class of $S^{n-1}$ under the double covering map $S^{n-1}\to \bR P^{n-1}$, hence it is twice a generator of the group $H_{n-1}(C_2(\bR^{n}))\cong H_{n-1}(\bR P^{n-1})\cong \bZ$ (cf. the class $\tau$ on page \pageref{def:tau}). Now, since the scanning map is split-injective on homology (see \cite[p.\ 103]{McDuff} and the proof of Corollary \ref{cor:inj} in this article), it follows that $\lambda(\beta,\beta)$ also has infinite order and is divisible exactly by two.

Observe that, when $n$ is odd, the above argument shows that $\lambda(\beta,\beta)$ is zero, since $\bR P^{n-1}$ is non-orientable for $n-1$ even.

By results of Serre \cite{Serre} on the homotopy groups of spheres, and the rational Hurewicz theorem, $H_{n-1}(\Omega_r^n S^n)$ has rank 1, so
\[H_{n-1}(\map_r(S^n,S^n))\cong H_{n-1}(\Omega_r^n S^n)/(-r\lambda(\beta,\beta))\cong \tau H_{n-1}(\Omega_r^n S^n)\oplus \bZ/{2r}.\]
%\rnote{I think this should be $\bZ/2r^e$ on the right, for some number $e\geq 1$ independent of $r$. This is because the image of $[\iota,\iota]$ might be divisible by $4$, $8$, etc.\ for all we know.}
 %and since $H_{n-1}(\delta)$ is injective, $H_{n-1}(\map_r(S^n,S^n))$ is a finite group. Let $A_r$ be the image of $H_{n-1}(\delta_r)$, so that $H_{n-1}(\map_r(S^n,S^n))\cong H_{n-1}(\Omega_r^nS^n)/A_r$. The third isomorphism theorem gives a short exat sequence 
%\[0\lra A_1/A_r\lra H_{n-1}(\Omega_r^nS^n)/A_r\lra H_{n-1}(\Omega_r^nS^n)/A_1,\]
%which is isomorphic to
%\[0\lra A_1/A_r\lra H_{n-1}(\map_r(S^n,S^n))\lra H_{n-1}(\map_1(S^n,S^n)),\]
%and therefore 
%\[\mathrm{ord}(H_{n-1}(\map_r(S^n,S^n))) = r\mathrm{ord}(H_{n-1}(\map_1(S^n,S^n))).\]
The first statement now follows from McDuff's theorem.
%, if $k,j$ are in the stable range and $k-\chi/2$ and $j-\chi/2$ have different $p$-adic valuation, then $H_{n-1}(C_k(S^n);\bZ_{(p)})\not\cong H_{n-1}(C_j;\bZ_{(p)})$.

%If $n$ is even, Theorem 3.1 in \cite{Hansen} says that the group $\pi_n\map_r(S^n,S^n)$ has order $r$ times the order of $\pi_n\map_1(S^n,S^n)$. When $n$ is odd, Theorem 4.1 in that paper says that the group $\pi_{n-1}\map_r(S^n,S^n)$ changes when $r$ changes its parity. If there were any zigzag $C_k(M)\to X_1\leftarrow X_2 \to\ldots \to C_j(M)$ of maps inducing homology equivalences in the stable range, then applying plus construction to each space, and adding the scanning maps, we obtain the diagram
 %\[
%\xymatrix@C=12pt{\map(S^n,S^n)\ar[d]  & \ar[l] C_k(S^n)\ar[r]\ar[d] & X_1\ar[d] & \ar[l]\ar[r]\ar[d] X_2 &\ldots \ar[r] & C_j(S^n) \ar[r]\ar[d]& \map(S^n,S^n)_j\ar[d]\\
%\map(S^n,S^n)^+ & \ar[l] C_k(S^n)^+\ar[r] & X_1^+ & \ar[l]\ar[r] X_2^+ &\ldots \ar[r] & C_j(S^n)^+ \ar[r]& \map(S^n,S^n)_j^+
%}\]
%All vertical maps are homology isomorphisms, and all the horizontal maps in the upper row are $\bZ_{(p)}$-homology isomorphisms in the stable range. Therefore so are the corresponding maps in the lower row. We have therefore a zigzag of $\bZ_{(p)}$-homology isomorphisms between $\map(S^n,S^n)_k$ and $\map(S^n,S^n)_j$. Since all spaces 
%
%
%%C_k(S^n)^+\lra X_1^+\leftlongarrow X_2^+ \lra\ldots \lra C_j(S^n)^+.\]
%Now, by the universal property of the plus construction, the scanning map would 

Let $f\colon \dot{T}M_{(\ell)}\to \dot{T}M_{(\ell)}$ be any fibrewise endomorphism of $\dot{T}M_{(\ell)}$, where we view the bundle $\dot{T}M_{(\ell)} \to M$ as a fibration. We can consider the fibrewise rationalisation $f_{(0)}\colon \dot{T}M_{(0)}\to \dot{T}M_{(0)}$, and observe that $[f](k) = [f_{(0)}](k)$ for $k\in\bZ_{(\ell)}$ (using the canonical inclusion of $\bZ_{(\ell)}$ in $\bQ$). Therefore the function $[f]$, describing the effect of $f$ on degrees of sections of $\dot{T}M_{(\ell)}$, is determined by the function $[f_{(0)}]$.

If $\dim M$ is even, there is a unique fibrewise endomorphism of $\dot{T}M_{(0)}$ of fibrewise degree $r$ up to homotopy. This is because such fibrewise endomorphisms are sections of a bundle over $M$ with fibre $\map_r(S^n_{(0)},S^n_{(0)})$, which, using the evaluation fibration \eqref{eq:510}, is $(2n-2)$-connected since $[\iota,r\iota]\neq 0$. Therefore, if $f$ and $g$ are fibrewise endomorphisms of $\dot{T}M_{(\ell)}$ with the same (non-zero) fibrewise degree, we have that $[f](k) = [f_{(0)}](k) = [g_{(0)}](k) = [g](k)$ for all $k\in\bZ_{(\ell)}$, so they act in the same way on the path-components of the section space.

Let $r=p/q$ be a non-zero rational number, with $p,q\in \bZ\setminus\{0\}$, and let $f_r$ be any fibrewise endomorphism of $\dot{T}M_{(0)}$ of degree $r$. Since $p$ and $q$ are integers, Corollary \ref{coro:existence} implies that there are fibrewise endomorphisms $\phi_p$ and $\phi_q$ of degrees $p$ and $q$ respectively. Let $f_p = \phi_qf_r$, which is a fibrewise endomorphism of degree $p$. By the previous paragraph, there is a unique fibrewise endomorphism of each degree, so it follows that $f_p$ is fibrewise homotopic to $\phi_p$. We therefore have the equation:
\[
[\phi_q][f_r](k) = [\phi_p](k).
\]
The last formula of Corollary \ref{coro:existence} determines the functions $[\phi_p]$ and $[\phi_q]$, from which we deduce that
\[
q([f_r](k) - \chi/2) + \chi/2 = p(k - \chi/2) + \chi/2
\]
and the result follows after solving the equation.
\end{proof}

\section{The extrinsic replication map}
\subsection{Scanning maps}\label{section:scanning}
Let $M$ be a connected manifold, for which we choose a Riemannian metric with injectivity radius bounded below by $\delta>0$. Let $T^1M$ denote the open unit disc bundle of the tangent bundle of $M$, and let $\dot{T}^1M$ and $\dot{T}M$ denote the fibrewise one point compactifications of $T^1M$ and $TM$. Let $\delta>0$ be smaller than the injectivity radius of $M$. Define the \emph{linear scanning map}
\[\ssl{}\colon C_k^\delta(M) \lra \Gamma_c(\dot{T}^1M)_k\]
to the space of degree $k$ compactly supported sections of $\dot{T}^1M$ as
\[\ssl{}(\q,\epsilon)(x) = \begin{cases}
\infty & \text{if $x\notin B_\epsilon(q)~\forall q\in \q$} , \\
\frac{\exp^{-1}_x(q)}{\epsilon} & \text{if $x\in B_\epsilon(q), q\in \q$}.
\end{cases}\]
The degree of a section $s$ is the fibrewise intersection, counted with multiplicity, of $s$ and the zero section (see also \S \ref{ss:degree}).

Let $D$ be the unit $n$-dimensional open disc, let $\dot{D}$ be its one point compactification and define $\psi^\delta(D)$ to be the quotient of $\bigcup_{k}C_k^\delta(\bR^n)$, where two configurations $(\q,\epsilon)$ and $(\q',\epsilon')$ are identified if $\q\cap D = \q'\cap D$ and either $\epsilon = \epsilon'$ or $\q\cap D = \emptyset$. We write $\psi^\delta(T^1M)$ for the result of applying this construction fibrewise to $T^1M$.

Let $\gamma$ be a number smaller than the injectivity radius of $M$. The \emph{radius $\gamma$ non-linear scanning map}\label{ssnl}
\[
\ssnl{}\colon C_k^\delta(M)\to \Gamma_c(\psi^\delta( T^1 M))
\]
sends a configuration $\q$ to $\frac{1}{\gamma}\exp_{x}^{-1}(\q)$ --- which may consist of more than one point.  

There is an inclusion $i\colon \dot{D}\hookrightarrow\psi^\delta(D)$ given by $i(q) = (q,\delta/2)$ as the subspace of configurations with at most one point. This inclusion has a homotopy inverse 
\[
h(\q,\epsilon) = \tfrac{\q}{\q_{\rm second}}
\]
where $\q_{\rm first}$ is the norm of a closest point in $\q$ to the origin, and $\q_{\rm second}$ is defined to be $1$ if $|\q|=1$ and $(\q')_{\rm first}$ otherwise, where $\q'$ is the result of removing a single closest point of $\q$ to the origin. The composite $hi$ is the identity and
$H_t(\q,\epsilon) = \left(\frac{\q}{(1-t) + t\q_{\rm second}},t\delta/2 + (1-t)\epsilon\right)$
gives a homotopy between the identity and $ih$.

Each of $i$, $h$ and $H_t$ is $O(n)$-equivariant, so they can be defined on the vector bundle $TM$, obtaining homotopy equivalences 
\[i\colon \dot{T}^1M \longleftrightarrow \psi^\delta(T^1M)\colon h\]
which induce by composition homotopy equivalences
\[i\colon \Gamma_c(\dot{T}^1 M) \longleftrightarrow \Gamma_c(\psi^\delta(T^1M))\colon h\]
that commute with the linear and non-linear scanning maps:
\begin{equation}\label{eq:11}
\centering
\begin{split}
\begin{tikzpicture}
[x=1mm,y=1mm]
\node (tl) at (0,15) {$C_k^\delta(M)$};
\node (tr) at (40,15) {$\Gamma_c(\psi^\delta T^1 M)$};
\node (br) at (40,0) {$\Gamma_c(\dot{T}^1 M).$};
\draw[->] (tl) to node[above,font=\small]{$\ssnl{\gamma}$} (tr);
\draw[->] (tl) to node[anchor=north east,font=\small]{$\ssl{}$} (br);
\draw[->] (tr) to[out=300,in=60] node[right,font=\small]{$h$} (br);
\draw[->] (br) to[out=120,in=240] node[left,font=\small]{$i$} (tr);
\node at (40,7.5) {$\simeq$};
\end{tikzpicture}
\end{split}
\end{equation}

\subsection{Homological stability}
\begin{rthm}[{\ref{thm:b}}.]
Let $M$ be a connected, smooth manifold and let $v$ be a non-vanishing section of $TM$. Then there exists a map $\phi_r\in \End^r_c(\dot{T}^1M)$ that makes the following diagram commute up to homotopy:
\begin{equation}\label{eq:thmbsquare}
\begin{split}
\xymatrix{
C_k^\delta(M) \ar[r]^-{\ssl{}}\ar[d]^{\r{r}[v]} & \Gamma_c(\dot{T}^1 M\to M)_k\ar[d]^{\phi_r} \\
C_{rk}^\delta(M) \ar[r]^-{\ssl{}}  & \Gamma_c(\dot{T}^1 M\to M)_{rk}.
}
\end{split}
\end{equation}
Hence the $r$-replication map induces an isomorphism on $\bZ_{(\ell)}$-homology in the stable range with $\bZ_{(\ell)}$ coefficients as long as $r$ is not divisible by any prime in $\ell$.
\end{rthm}
\begin{remark} One can prove that the map $\phi_r$ constructed below is homotopic to $\Phi_r^\ell(\iota,v)$ with $\ell = \Spec \bZ$.
\end{remark}
\begin{proof}
The proof has three steps. First, since $C^\delta_k(M)$ is independent of $\delta$ up to homotopy, we let %$2\delta$
$2\delta$ be smaller than the injectivity radius of $M$. We claim that the following diagram commutes:
\[\xymatrix{C_k^\delta(M) \ar[r]^-{\ssnl{2\delta}} \ar[d]^{\r{r}[v]} & \Gamma_c(\psi^\delta(T^1 M)) \ar[d]^{\varsigma_r}  \\
C_{rk}^\delta(M) \ar[r]^-{\ssnl{\delta}} & \Gamma_c(\psi^\delta(T^1 M)) %\ar[r]^f & \Gamma_c(\psi(T^1 M)),
}\]
where $\varsigma_r$ is given by postcomposition with the bundle map $\r{r}[\exp_{2\delta}^*(v)]\colon \psi^\delta(T^1M)\to \psi^\delta(T^1M)$ followed by the expansion $\mathbbm{2}\colon \psi^\delta(T^1M)\to \psi^\delta(T^1M)$ that sends each point $q$ in the configuration to $2q$. Observe that the bundle map $\r{r}[\exp_{2\delta}^*(v)]$ is not continuous but it becomes continuous after composing with $\mathbbm{2}$.

In order to understand this square, we check what happens with the adjoint of the scanning map $M\times C_k(M)\to \psi^\delta(T^1M)$ over each point $x\in M$:
\[\xymatrix{\{x\}\times C_k^\delta(M) \ar[r]^-{\ssnl{2\delta}_x} \ar[d]^{\r{r}[v]} & \psi^\delta(T^1_x M) \ar[d]^{\varsigma_r}&  \\
\{x\}\times C_{rk}^\delta(M) \ar[r]^-{\ssnl{\delta}_x} & \psi(T^1_x M).}\]
The square commutes on the nose unless there exists some $q\in \q$ such that 
\[\varsigma_r(q)\cap B_{\delta}(x)\neq \emptyset, \text{and}~q\notin B_{2\delta}(x).\]
But this is not possible, as $d(\varsigma_r(q),x) \geq d(q,x) - \max_{q'\in \varsigma_r(q)}{d(q,q')} \geq 2\delta - \epsilon \geq \delta$.

Second, observe that since the exponential map is homotopic to the projection $\pi\colon TM\to M$, the maps $\varsigma_r = \mathbbm{2}\r{r}[\exp_{2\delta}^*(v)]$ and $\sigma_r = \mathbbm{2}\r{r}[\pi^*v]$ are homotopic.

Third, consider now the diagram
\[\xymatrix{
\Gamma_c(\psi^\delta(T^1 M)) \ar[d]^{\sigma_r} & \Gamma_c(\dot{T}^1 M) \ar[d]\ar[l]_i \\
 \Gamma_c(\psi^\delta(T^1 M)) \ar[r]^h & \Gamma_c(\dot{T}^1 M)
%\Gamma_c(\psi(T^\delta M))_{2n}\ar[r]& \Gamma_c(\dot{T}^\delta M)
}\]
whose maps are induced by the fibrewise maps which on each fibre are
\[\xymatrix{
\psi^\delta(T^1_x M) \ar[d]^{\sigma_r} & \dot{T}^1_x M \ar[d]\ar[l]_i \\
 \psi^\delta(T^1_x M) \ar[r]^h & \dot{T}^1_x M
%\Gamma_c(\psi(T^\delta M))_{2n}\ar[r]& \Gamma_c(\dot{T}^\delta M)
}\]

Let us denote by $v$ the value of the vector field $v$ at the point $x$. Then $\sigma_r(q,1) = q\cup q+v \cup \ldots \cup q+(r-1)v%=:\q
$ and
\[h\sigma_r i(q) =\begin{cases}
2\frac{q+jv}{\|q+(j-1)v\|} & \text{if $q+jv$ is the closest point and $\langle v,q+jv\rangle > 0$} \\
2\frac{q+jv}{\|q+(j+1)v\|} & \text{if $q+jv$ is the closest point and $\langle v,q+jv\rangle < 0$}.
\end{cases}\]
The inverse image of a point (for instance the origin) consists of $r$ points ($\{-jv\}_{j=0}^{r-1}$), all of them oriented according to the sign of $r$. Hence $h\sigma_r i$ induces a map of degree $r$ on fibres.
\end{proof}

\begin{corollary}\label{cor:inj} If $M$ is a connected open manifold of dimension at least $2$, then the homomorphism induced on $\bZ[\frac{1}{r}]$-homology by the $r$-replication map is split-injective.
\end{corollary}
\begin{proof}
The scanning map is split-injective on homology in all degrees, as may be deduced from \cite{McDuff}. To see this, recall the following facts from the referenced paper: (a) the stabilisation maps $C_k(M)\to C_{k+1}(M)$ are split-injective on homology in all degrees (see p.\ 103), (b) the analogous stabilisation maps for section spaces $\Gamma_c(\dot{T}M)_k \to \Gamma_c(\dot{T}M)_{k+1}$ are homotopy equivalences, (c) the scanning map $C_k(M) \to \Gamma_c(\dot{T}M)_k$ is a homology equivalence in the colimit as $k\to\infty$ and (d) the homology in the colimit is finitely-generated. It then follows from the fact that stabilisation and scanning commute that the homomorphism induced on homology by the scanning map is the composition of a finite sequence of split-injections, and therefore itself a split-injection.
%and that $\colim_k H_*(\Gamma_c(\dot{T}M)_k)$ (written $\lim_k H_*(\Gamma_k(M,\partial M))$ in the notation of that paper) is constant.

Thus, in the commutative square \eqref{eq:thmbsquare}, the composite $\phi_r \circ \ssl{}$ is split-injective on $\bZ[\frac{1}{r}]$-homology, hence $\ssl{} \circ \r{r}[v]$ is split-injective on $\bZ[\frac{1}{r}]$-homology, and so $\r{r}[v]$ is also split-injective on $\bZ[\frac{1}{r}]$-homology. 
\end{proof}

\section{The intrinsic replication map}
\label{sUpToDim}
\subsection{Stabilisation, replication and scanning maps with labels} Let $\theta\colon E\to M$ be a fibre bundle, and define the following spaces (the first two were also defined in the introduction; the third space is defined whenever $\gamma$ is smaller than the injectivity radius of $M$, in particular when $\gamma<\delta$):
\begin{align*}
C_k(M;\theta) &=  \{(\q,f)\mid  \q\in C_n(M), f\in \Gamma(\theta|_{\q})\} \\
C_k^\delta(M;\theta) &= \{(\q,\epsilon,f)\mid (\q,\epsilon)\in C_k^\delta(M), f\in \Gamma(\theta|_{B_\q(\epsilon)})\}\\
C_k^{\delta,\gamma}(M;\theta) &= \{(\q,\epsilon,\{f_q\}_{q\in\q})\mid (\q,\epsilon)\in C_k^\delta(M), f_q\in \Gamma(\theta|_{B_q(\gamma)})\}.
\end{align*}

\begin{lemma}\label{lTwoEquivalences}
The forgetful maps 
\[C_k^{\delta,\gamma}(M;\theta)\lra C_k^\delta(M;\theta)\lra C_k(M;\theta)\]
that restrict the section first from balls of radius $\gamma$ to balls of radius $\epsilon$, and then to the centres of the balls, are weak homotopy equivalences.
\end{lemma}
\begin{proof}
A point in $C_k^\delta(M;\theta)$ consists of a configuration $\q$ with prescribed pairwise separation, together with a choice of label on a small contractible neighbourhood of each configuration point. On the other hand, the pullback $\dot{C}_k^\delta(M;\theta)$ of $C_k^\delta(M)$ along the map $C_k(M;\theta)\to C_k(M)$ which forgets the labels consists of a configuration with prescribed pairwise separation, together with a choice of label just over each configuration point. Since $C_k^\delta(M)\to C_k(M)$ is a fibre bundle with contractible fibres, so is its pullback $\dot{C}_k^\delta(M;\theta) \to C_k(M;\theta)$, which is therefore a weak equivalence. There is also a forgetful map $C_k^\delta(M;\theta)\to \dot{C}_k^\delta(M;\theta)$ which just remembers the label at the centre of each ball. This is also a fibre bundle with contractible fibres, so a weak equivalence. Hence the composition $C_k^\delta(M;\theta)\to C_k(M;\theta)$ which completely forgets the labels (the second map in Lemma \ref{lTwoEquivalences}) is a weak equivalence. The first map in Lemma \ref{lTwoEquivalences} is a fibre bundle with contractible fibres, so it is also a weak equivalence.
\end{proof}

\begin{df} If the manifold $M$ is open, a choice of an embedding of a ray, together with a point $y$ in the ray and a label $f_y\in \theta^{-1}(y)$ of this point defines a \emph{stabilisation map with labels}
\[s^\theta\colon C_k(M;\theta)\lra C_{k+1}(M;\theta)\]
by pushing the configuration outside the ray and adding the labelled point $(y,f_y)$.
\end{df}
\begin{df}
If the manifold $M$ is open or has trivial Euler characteristic, a choice of a non-vanishing vector field defines a \emph{replication map with labels}
\[\r{r}^\theta\colon C_k^\delta(M)\lra C_{rk}^\delta(M;\theta)\]
\[\r{r}^{\theta,\gamma}\colon C_k^{\delta,\gamma}(M;\theta)\lra C_{rk}^{\delta,\gamma/r}(M;\theta)\]
by sending a configuration $((\q,\epsilon),f)$ to the configuration $\r{r}(\q,\epsilon)$ together with the restriction of the section $f$ to the balls of radius $\frac{\epsilon}{2r}$ (and radius $\frac{\gamma}{r}$) centered at the points in the new configuration.
\end{df}

The pullback $\theta^*TM\to E$ is also fibred over $M$, and the fibres are vector bundles. We denote by $\dot{T}^\theta M$ the fibrewise Thom construction of $\theta^*TM$ viewed as a bundle over $M$. The inclusion of the points at infinity define a cofibre sequence over $M$
\begin{equation}\label{eq:261} E\lra \theta^*\dot{T}M \lra \dot{T}^\theta M.\end{equation}
The pullback map $\theta^*\dot{T}M\to \dot{T}M$ factors through the bundle maps 
\begin{equation}\label{eq:99}
\theta^*\dot{T}M\lra \dot{T}^\theta M\stackrel{\xi}{\lra} \dot{T}M.
\end{equation}
We define the \emph{degree} of a section $s$ as the degree of $\xi(s)$. If the fibres of $\theta$ are path connected, then the $n$-skeleton of the fibres of $\dot{T}^\theta M$ is homotopic to $S^n$, therefore the forgetful map
\[\Gamma_c(\dot{T}^\theta M)\lra \Gamma_c(\dot{T}M)\]
induces a bijection on connected components. We write $\dot{T}^{1,\theta}M$ for the analogous construction with $T^1M$.

\begin{df} Define the \emph{linear} scanning map 
\[\ssl{\theta}\colon C_k^\delta(M;\theta)\lra \Gamma(\dot{T}^{1,\theta} M)\]
by sending a configuration $(\q,\epsilon,f)$ to the section whose value at a point $x\in M$ is $\ssl{}(\q,\epsilon)(x)$ together with the label $f|_{x}$ if $\ssl{}(\q,\epsilon)(x)\neq \infty$.
\end{df}

Let $D$ be the unit $n$-dimensional open disc, let $\dot{D}$ be its one-point compactification and let $F$ be a space. Define $\psi^\delta(D;F)$ to be the quotient of the space $\bigcup_k{C_k^\delta(\bR^n;\theta\colon F\times \bR^n\to \bR^n)}$, where two labeled configurations $(\q,\epsilon,f)$ and $(\q',\epsilon',f')$ are identified if $\q\cap D = \q'\cap D$, $f|_{B_\epsilon(q)} = f'|_{B_\epsilon(q)}$ for all $q\in\q\cap D$, and either $\q\cap D=\emptyset$ or $\epsilon = \epsilon'$.

Let $\psi^\delta(T^1M;\theta)$ be the result of applying this construction fibrewise to the unit ball of the tangent bundle of $M$ and the fibre bundle $\theta$, so that 
\[\psi^\delta(T^1M;\theta) = \bigcup_{x\in M} \psi^\delta(T^1_xM;\theta^{-1}(x)).\]
\begin{df} The \emph{non-linear scanning map with labels in $\theta$}
\[\ssnl{\theta,\gamma}\colon C_k^{\delta,\gamma}(M;\theta)\lra \Gamma(\psi^\delta(T^1M;\theta))\]
sends a configuration $(\q,\epsilon,\{f_q\}_{q\in\q})$ to the section that assigns to the point $x\in M$ the triple $(\q',\epsilon',\{f'\}_{q'\in\q'})$, where $(\q',\epsilon') = \ssnl{}(\q,\epsilon)\in \psi^\delta(T_x^1M)$, and $f'_{q'}$ is constant with value $f_{q}(x)$ if $q' = \frac{1}{\gamma}\exp^{-1}(q)$.  
\end{df}

There are fibrewise maps
\begin{align*}
i\colon \dot{T}^{1,\theta} M&\hookrightarrow \psi^\delta(T^1M;\theta) \\
h\colon \psi^\delta(T^1M;\theta) & \lra \dot{T}^{1,\theta} M
\end{align*}
defined by $i(\infty) = \infty$, $h(\infty) = \infty$ and
\begin{align*}
i(q,y)&= (q,\delta/2,f) \qquad\qquad\quad\text{with $f$ constant with value $y$} \\%& i(x,\infty) &= (x,\infty) \\
h(\q,\epsilon,\{f_q\}_{q\in\q})&= \begin{cases}
\left(\frac{\q}{\q_{\rm second}}, f\left(q\right)\right)&\text{ if $\q_{\rm first}<\q_{\rm second}$ and $q\in\q$ with $\|q\| = \q_{\rm first}$,}\\
\infty & \text{ if $\q_{\rm first}=\q_{\rm second}$}
\end{cases}
%& h(x,\infty) &= (x,\infty)
% \text{ and by }(x,\infty)\mapsto (y,\infty)\text{ for any $\theta(y)=x$.}
\end{align*}
which are mutually fibrewise homotopy inverses by the same argument as on page \pageref{ssnl} (where the definition of $\q_{\rm second}$ is also given). 

In Appendix \ref{appendixB} we show that the stabilisation map with labels induces an isomorphism in a range $*\leq \sr[M;\theta](k)$, and we give lower bounds for it in Proposition \ref{pImprovedRangeTwisted} and Remark \ref{rImprovedRangeTwisted}  -- these are either the same or one degree less than the lower bounds given on page \pageref{pageone} for unlabelled configuration spaces. 

In Appendix \ref{appendix2} we give a proof of McDuff's theorem with labels, as follows (where $\srscan[M;\theta]$ is the range in which the scanning map for $M$ is an isomorphism on homology):
\begin{theorem}\label{thm:McDufflabelssection4}
Let $\theta\colon E\to M$ be a fibre bundle with path-connected fibres. If $M$ is non-compact then we have
\[
\srscan[M;\theta](k) = \min_{j\geq k}\{\sr[M;\theta](j)\}.
\]
The inequality $\srscan[M;\theta]\geq \srscan[M\smallsetminus\point;\theta|_{M\smallsetminus\point}]$ holds for all $M$, so the function $\srscan[M;\theta]$ diverges for all $M$.
\end{theorem}

\subsection{Homological stability} The fibrewise homotopy equivalences of Lemma \ref{lemma:ultimate} lift to fibrewise homotopy equivalences of $\theta^*\dot{T}^1M_{(p)}$, which in turn descend to fibrewise homotopy equivalences $\dot{T}^{1,\theta} M_{(p)}$ if and only if they fix the section at infinity (c.f. cofibration \eqref{eq:261}). This implies that $\sigma_0 = \iota$ in that lemma, and therefore each of these fibrewise homotopy equivalences sends sections of degree $k$ to sections of degree $rk$. Hence we only recover part of Theorem \ref{thm:a} and the whole of Theorem \ref{thm:b}:
\begin{rthm}[$\text{A}^\prime$]\label{thm:aprime}
If $M$ is a closed, connected manifold with trivial Euler characteristic, then $H_*(C_k(M;\theta))\cong H_*(C_j(M;\theta))$ in the stable range with labels if $k$ and $j$ have the same $p$-adic valuation.
\end{rthm}

\begin{rthm}[$\text{B}^\prime$]\label{thm:bprime}
If $v$ is a non-vanishing vector field on a connected manifold $M$ and $p\nmid r$, then the $r$-replication map $\r{r}^\theta$ with labels induces isomorphisms in $\bZ_{(p)}$-homology in the stable range with labels.
\end{rthm}
\begin{proof}
The relevant diagram is the following:
\[\xymatrix{
C_k^{\delta,\gamma} \ar[r]^-{\ssnl{\theta,\gamma}} \ar[d]^-{\r{r}} & \Gamma_c(\psi^\delta(T^1M;\theta)) \ar[d] & \Gamma_c(\dot{T}^{1,\theta} M) \ar[d]\ar[l]_-{i} \\
C_{rk}^{\delta,\gamma/r} \ar[r]^-{\ssnl{\theta,\gamma/r}} & \Gamma_c(\psi^\delta(T^1M;\theta)) \ar[r]^-h & \Gamma_c(\dot{T}^{1,\theta} M) 
}\]
The argument in the proof of Theorem \ref{thm:b} generalizes step by step to give that the left hand-side square commutes up to homotopy and that the rightmost vertical arrow is given by postcomposition with a bundle map which on the fibre over $x\in M$ induces a map
\[\xymatrix{
F \ar[r]\ar[d]^\Id & F\times S^n \ar[d]^{\Id\times f_r}\ar[d]\ar[r] & \Sigma^n F_+ \ar[d]\\
F \ar[r] & F\times S^n \ar[r] & \Sigma^n F_+ \\
}\]
where $F = \theta^{-1}(x)$, $f_r$ is a degree $r$ map, hence the rightmost vertical map is a $\bZ[\frac{1}{r}]$-homotopy equivalence.
 degree $r$ between sphere bundles.
\end{proof}

Let $\theta$ be the projection $S(TM)\to M$. Recall the definition of the intrinsic replication map $\ro{r}\colon \W{k}\lra \W{rk}$ from page \pageref{intrinsicreplication}.

\begin{rthm}[{\ref{thm:c}}]
If $M$ is a connected smooth manifold and $\ell$ is a set of primes not dividing an integer $r$, then the map $\ro{r}\colon \W{k} \to \W{rk}$ induces isomorphisms on homology with $\bZ_{(\ell)}$-coefficients in the stable range with labels.
\end{rthm}
\begin{proof}
Define $\sigma_r\colon \psi(TM;\theta) \to \psi(TM;\theta)$ to be the fibrewise version of $\ro{r}$ composed with $\mathbbm{2}$ as in the proof of Theorem \ref{thm:b}. The first square in the following diagram
\begin{equation}
\begin{split}
\label{eq:98}\xymatrix{
C_k^{\delta,\gamma}(M;\theta) \ar[r]^-{\ssnl{\theta,\gamma}}\ar[d]^-{\ro{r}} & \Gamma_c(\psi^\delta(T^1M;\theta))\ar[d]^-{\sigma_r} & \Gamma_c(\dot{T}^{1,\theta} M)_k \ar[l]_-i\ar[d]^-{h\sigma_ri} \\
C_k^{\delta,\gamma/r}(M;\theta) \ar[r]^-{\ssnl{\theta,\gamma/r}} & \Gamma_c(\psi^\delta(T^1M;\theta)) \ar[r]^-h& \Gamma_c(\dot{T}^{1,\theta} M)_{rk}.
}
\end{split}
\end{equation}
commutes, by the same argument as the first step the proof of Theorem \ref{thm:b}. Therefore we obtain a fibrewise map $h\sigma_r i$ on the right hand side. The map $h\sigma_r i$ is obtained by postcomposition with a fibrewise map $g\colon \dot{T}^{1,\theta} M\to \dot{T}^{1,\theta} M$. The map $g_x$ on the fibre over the point $x$ restricts to the identity on the points at infinity, therefore it extends to the following diagram of cofibre sequences
\[\xymatrix{
S(T_x M)\times\{\infty\}\ar[d]^\Id\ar[r]& S(T_xM)\times \dot{T}_x^1M \ar[r]\ar[d]^{f}& \dot{T}^{1,\theta}_x M \ar[d]^{g_x}\\
S(T_x M)\times\{\infty\}\ar[r]& S(T_xM)\times \dot{T}_x^1M \ar[r]& \dot{T}^{1,\theta}_x M \\
}\] 
where the leftmost horizontal maps are the inclusion of the points at infinity. After localizing the diagram at $\ell$, the map $f_{(\ell)}$ is a map of sphere bundles that induces a map of degree $r$ on fibres (as in the proof of Theorem \ref{thm:b}). Since $r$ is a unit in $\bZ_{(\ell)}$, it follows that $f_{(\ell)}$ is a homotopy equivalence, and therefore that $(g_x)_{(\ell)}$ is too. Since $g_{(\ell)}$ induces a homotopy equivalence on fibres, it follows (using Theorem 3.3 of \cite{Dold}) that $g_{(\ell)}$ is a fibrewise homotopy equivalence, and therefore that $(h\sigma_r i)_{(\ell)}$ is a homotopy equivalence.
\end{proof}

If the bundle $\theta\colon E\to M$ factors through $S(TM)$ and has path-connected fibres (for instance the oriented frame bundle of $TM$, if $M$ is orientable), then Theorem \ref{thm:c} generalises to:
\begin{rthm}[$\text{C}^\prime$]\label{thm:cprime}
If $\ell$ is a set of primes, none of them dividing the integer $r$, then the intrinsic replication map with labels 
\[\ro{r}\colon C_k(M;\theta)\lra C_{rk}(M;\theta)\]
induces isomorphisms on homology with $\bZ_{(\ell)}$ coefficients in the stable range with labels.
\end{rthm}

\section{Homological stability via vector fields with exactly one zero}
\label{sPart2}

We now use some different techniques to extend our results a bit further for homology with field coefficients. Section spaces are not involved in this part; instead we apply Theorem \ref{thm:b} (homological stability with respect to the $r$-replication map) to $M\smallsetminus\{*\}$ and classical homological stability for $M\smallsetminus\{*\}$ to obtain Theorem \ref{thm:e}.

%%%%%%%%%%%%%%%%%%%%%%%%%%%%%%%%%%%%%%%%%%%%%%%%%%%%%%%%%%%%%%%%%%%%%%%%%%%%%%%%%%
\subsection{Vector fields}
%%%%%%%%%%%%%%%%%%%%%%%%%%%%%%%%%%%%%%%%%%%%%%%%%%%%%%%%%%%%%%%%%%%%%%%%%%%%%%%%%%

Let $M$ be a closed connected manifold with Euler characteristic $\chi$.

\begin{df}\label{dIndex}
Given a vector field $v\in \Gamma(TM)$ with an isolated zero $z\in M$, define the degree $\deg_v(z)$ of $z$ as follows. Choose a coordinate chart $U\cong\bR^n$, with $z\in U$ corresponding to $0\in\bR^n$, such that $v$ has no other zeros in $U$. The differential of the diffeomorphism $U\cong \bR^n$ is a bundle isomorphism $TM|_U = TU \cong T\bR^n = \bR^n \times \bR^n$. The restriction of $v$ to $U\smallsetminus\{z\}$ therefore determines a map $\cyl\to\cyl$. The degree of this map is by definition $\deg_v(z)$.
\end{df}

A simple observation is that $M$ admits a vector field with at most one zero, which will therefore have index $\chi$ by the Poincar{\'e}-Hopf theorem. Moreover we can choose exactly what this zero looks like locally:

\begin{lemma}\label{lExtendingVF}
Suppose we are given a vector field $v$ on a closed ball $B\subseteq M$ with exactly one zero which lies in its interior and has index $\chi$. Then this extends to a vector field $\hat{v}$ on $M$ which is non-vanishing on $M\smallsetminus B$.
\end{lemma}
\begin{proof}
First choose a vector field $w$ on $M$ which has only isolated (and therefore finitely many) zeros. Choose a larger closed ball $B^\prime\supset B$ and a trivialisation of $TM|_{B^\prime}$. Now homotope $w$ if necessary so that all its zeros lie in $\mathrm{int}(B^\prime)\smallsetminus B$. The restriction of $w$ to $\partial B^\prime$ is a map $\partial B^\prime \to \cyl$, whose degree is the sum of the degrees of all zeros of $w$, which is $\chi$ by the Poincar{\'e}-Hopf theorem. The restriction of the vector field $v$ to $\partial B$ is a map $\partial B\to \cyl$ which also has degree $\chi$ by assumption. Since any two maps $S^{n-1}\to \cyl$ of the same degree are homotopic, there is a map $x\colon B^\prime \smallsetminus \mathrm{int}(B) \cong S^{n-1} \times [0,1] \to \cyl$ agreeing with $w$ on $\partial B^\prime$ and with $v$ on $\partial B$. We can therefore define $\hat{v}$ to be equal to $v$ on $B$, $x$ on $B^\prime \smallsetminus B$ and $w$ on $M\smallsetminus B^\prime$.
\end{proof}

%%%%%%%%%%%%%%%%%%%%%%%%%%%%%%%%%%%%%%%%%%%%%%%%%%%%%%%%%%%%%%%%%%%%%%%%%%%%%%%%%%
\subsection{A cofibre sequence of configuration spaces}\label{ssCofibresequence}
%%%%%%%%%%%%%%%%%%%%%%%%%%%%%%%%%%%%%%%%%%%%%%%%%%%%%%%%%%%%%%%%%%%%%%%%%%%%%%%%%%

Choose a Riemannian metric on $M$ and an isometric embedding $D\hookrightarrow M$ of the closed unit disc $D\subseteq \bR^n$. Following \cite[\S 6]{RW-hs-for-ucs} we define $U_k(M)$ to be the subspace of $C_k(M)$ of configurations which have a unique closest point in $D$ to its centre $0\in D$. There is an open cover of $C_k(M)$ given by the subsets $U_k(M)$ and $C_k(M\smallsetminus\{0\})$, with intersection $U_k(M\smallsetminus\{0\})$. By excision, the induced map of mapping cones
\begin{equation}\label{eExcision}
(U_k(M),U_k(M\smallsetminus\{0\})) \longrightarrow (C_k(M),C_k(M\smallsetminus\{0\}))
\end{equation}
is a homology equivalence. The space $U_k(M)$ decomposes up to homeomorphism as $D^n \times C_{k-1}(M\smallsetminus\{0\})$ and similarly $U_k(M\smallsetminus\{0\}) \cong (D^n\smallsetminus\{0\}) \times C_{k-1}(M\smallsetminus\{0\})$, so the left-hand side of \eqref{eExcision} is homeomorphic to
\[
(D^n,D^n\smallsetminus\{0\})\wedge C_{k-1}(M\smallsetminus\{0\})_+.
\]
Composing with the homotopy equivalence $\partial D^n \to D^n \smallsetminus\{0\}$ we obtain the following diagram:
\begin{equation}\label{eCofibreSequences}
\centering
\begin{split}
\begin{tikzpicture}
[x=1mm,y=1mm,font=\small]
\node (t1) at (0,15) {$C_k(M\smallsetminus\{0\})$};
\node (t2) at (40,15) {$C_k(M)$};
\node (t3) at (80,15) {$(C_k(M),C_k(M\smallsetminus\{0\}))$};
\node (b1) at (0,0) {$S^{n-1}\times C_{k-1}(M\smallsetminus\{0\})$};
\node (b2) at (40,0) {$D^n\times C_{k-1}(M\smallsetminus\{0\})$};
\node (b3) at (80,0) {$\Sigma^n(C_{k-1}(M\smallsetminus\{0\})_+)$};
\draw[->] (t1) to (t2);
\draw[->] (t2) to (t3);
\draw[->] (b1) to (b2);
\draw[->] (b2) to (b3);
\draw[->] (b1) to node[left,font=\footnotesize]{$t_{k-1}$} (t1);
\draw[->] (b2) to (t2);
\draw[->] (b3) to node[right,font=\footnotesize]{$(\star)$} (t3);
\end{tikzpicture}
\end{split}
%\end{adjustbox}
\end{equation}
where the rows are cofibre sequences and the rightmost vertical map $(\star)$ is a homology equivalence. The map $t_{k-1}$ may be described as radially expanding the configuration in $C_{k-1}(M\smallsetminus\{0\})$ away from $0$ until it has no points in $D$, and then adding the point in $S^{n-1} = \partial D$ to the configuration.

\begin{remark}\label{rKuenneth}
Note that the bottom left horizontal map of \eqref{eCofibreSequences} is homotopy split-surjective, so the maps
\[
\Sigma^{n-1}(C_{k-1}(M\smallsetminus\{0\})_+) \dashrightarrow S^{n-1}\times C_{k-1}(M\smallsetminus\{0\}) \longrightarrow D^n\times C_{k-1}(M\smallsetminus\{0\})
\]
induce split short exact sequences on homology, corresponding to the K\"unneth decomposition for $S^{n-1}\times C_{k-1}(M\smallsetminus\{0\})$. The dotted arrow is the connecting homomorphism and only exists on homology.
\end{remark}

The upshot of this discussion is the following lemma, where $\dashrightarrow$ indicates a map which is only defined on homology.

\begin{lemma}\label{lCofibreSequence}
There are maps $C_k(M\smallsetminus\{0\}) \longrightarrow C_k(M) \dashrightarrow \Sigma^n(C_{k-1}(M\smallsetminus\{0\})_+)$ which induce a long exact sequence on homology. The connecting homomorphism is the composite
\[
\Sigma^{n-1}(C_{k-1}(M\smallsetminus\{0\})_+) \dashrightarrow S^{n-1} \times C_{k-1}(M\smallsetminus\{0\}) \longrightarrow C_k(M\smallsetminus\{0\}).
\]
The first of these two maps is the inclusion of a direct summand of the homology of $S^{n-1} \times C_k(M\smallsetminus\{0\})$ and the second is the map $t_{k-1}$ described above.
\end{lemma}

%%%%%%%%%%%%%%%%%%%%%%%%%%%%%%%%%%%%%%%%%%%%%%%%%%%%%%%%%%%%%%%%%%%%%%%%%%%%%%%%%%
\subsection{Configuration spaces on cylinders}\label{ssCylinders}
%%%%%%%%%%%%%%%%%%%%%%%%%%%%%%%%%%%%%%%%%%%%%%%%%%%%%%%%%%%%%%%%%%%%%%%%%%%%%%%%%%

For the remainder of this section $n=\dim(M)$ will always be assumed even.

\paragraph{Some natural homology classes.}
We will need to do some calculations inside the homology group $H_{n-1}(C_k(\cyl);\bZ)$ of punctured Euclidean space. There are certain natural elements of this group which one can write down. For example we have the following elements (see also Figure \ref{fHomologyClasses}):

\begin{enumerate}[(a)]
\item For any $0\leq j\leq k-1$ we have a map $\Delta_j\colon S^{n-1}\to C_k(\cyl)$ which sends $v\in S^{n-1}$ to the configuration $\{v,p_1,\dotsc,p_{k-1}\}$, where $p_1,\dotsc,p_{k-1}$ are arbitrary fixed points in $\cyl$ with $\lvert p_i\rvert <1$ for $i\leq j$ and $\lvert p_i\rvert >1$ for $i>j$.
\item[] By abuse of notation we denote the element $(\Delta_j)_*([S^{n-1}])$ simply by $\Delta_j \in H_{n-1}(C_k(\cyl);\bZ)$. We will systematically use this abuse of notation for maps $S^{n-1}\to C_k(\cyl)$.
\item We also have a map $\pi\colon \RP^{n-1}\to C_k(\cyl)$ which sends $\{v,-v\}\in \RP^{n-1} = S^{n-1}/\sim$ to the configuration $\{\underline{2}+v,\underline{2}-v,p_1,\dotsc,p_{k-2}\}$, where $\underline{2} = (2,0,\dotsc,0)$ and $p_1,\dotsc,p_{k-2}$ are fixed points in $\bR^n\smallsetminus B_1(\underline{2})$. This gives us an element $\pi\in H_{n-1}(C_k(\cyl);\bZ)$.
\item Composing this map with the double covering $S^{n-1}\to \RP^{n-1}$ gives a map representing $2\pi$. This is homotopic to the \label{def:tau} map $\tau\colon S^{n-1}\to C_k(\cyl)$ which sends $v\in S^{n-1}$ to the configuration $\{p_1,s(v),p_2,\dotsc,p_{k-1}\}$, where $s\colon S^{n-1} \to \cyl$ is an embedding so that $p_1$ is in the interior of $s(S^{d-1})$ and $0,p_2,\dotsc,p_{k-1}$ are in its exterior.
\item More generally, for any $1\leq j\leq k-1$ we can define a map $\tau_j \colon S^{n-1} \to \cyl$ which sends $v\in S^{n-1}$ to the configuration $\{p_1,\dotsc,p_j,s(v),p_{j+1},\dotsc,p_{k-1}\}$, where $p_1,\dotsc,p_j$ are in the interior of $s(S^{n-1})$ and $0,p_{j+1},\dotsc,p_{k-1}$ are in its exterior. So $\tau_1 = \tau = 2\pi$.
\end{enumerate}

\begin{figure}[ht]
\begin{tikzpicture}
[x=1mm,y=1mm]
\draw[black!50] (-13,-10) rectangle (53,16);
\draw[black!50] (13,-10) -- (13,16);
\draw[black!50] (27,-10) -- (27,16);
\begin{scope}
\node at (0,0) [draw,circle,inner sep=1pt] {};
\foreach \y in {2,4,6,8,10,12} {\node at (0,\y) [fill,circle,inner sep=1pt] {};}
\draw[-angle 90] (8,0) arc (0:360:8);
\node at (1.5,-13) {$\Delta_3$};
\end{scope}
\begin{scope}[xshift=20mm]
\node at (0,0) [draw,circle,inner sep=1pt] {};
\foreach \y in {2,6,8,10,12,14} {\node at (0,\y) [fill,circle,inner sep=1pt] {};}
\draw[-angle 90]  (2,4) arc (0:180:2);
\draw[-angle 90]  (-2,4) arc (180:360:2);
%\draw[-angle 90] (-0.5,4) to[out=130,in=230] (-0.5,8);
%\draw[angle 90-] (0.5,4) to[out=50,in=310] (0.5,8);
\node at (0,-13) {$\pi$};
\end{scope}
\begin{scope}[xshift=40mm]
\node at (0,0) [draw,circle,inner sep=1pt] {};
\foreach \y in {4,6,8,10,12,14} {\node at (0,\y) [fill,circle,inner sep=1pt] {};}
\draw[-angle 90] (3,5) arc (0:360:3);
\node at (0,-13) {$\tau_2$};
\end{scope}
\end{tikzpicture}
\caption{\small Examples of homology classes in $H_1(C_6(\bR^2\smallsetminus\{0\});\bZ)$. The small circle denotes the puncture $0$ and bullets denote points of the configuration.}\label{fHomologyClasses}
\end{figure}
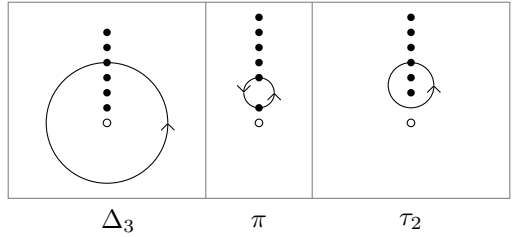

\paragraph{Relations between homology classes.}

Let $P^n$ denote the closed $n$-dimensional disc $\bD^n$ with two open subdiscs (whose closures are disjoint) removed; this is the $n$-dimensional pair-of-pants. Consider the map $r\colon P^n \to C_k(\cyl)$ pictured in Figure \ref{fPairofpants}. The image $r_*([\partial P^n])$ of the fundamental class of its boundary is the class $\Delta_{j+1}-\Delta_j-\tau_1$, which is therefore equal to zero in $H_{n-1}(C_k(\cyl);\bZ)$. Similarly the map $r^\prime \colon P^n \to C_k(\cyl)$ pictured in Figure \ref{fPairofpants} shows that $\tau_{j+1}-\tau_j-\tau_1 = 0$. Hence by induction and the fact that $\tau_1 = 2\pi$ we have
\begin{equation}\label{eRelationA}
\Delta_j = \Delta_0 + 2j\pi \qquad\text{and}\qquad \tau_j = 2j\pi.
\end{equation}

\begin{figure}[ht]
\centering
\begin{tikzpicture}
[x=1mm,y=1mm]

\draw[black!50] (-10,-10) rectangle (122,10);
\draw[black!50] (56,-10) -- (56,10);

\fill[black!20] (13,6) --(0,6) arc (90:270:6) -- (26,-6) arc (270:450:6) -- (13,6);
\fill[white] (26,4) arc (450:90:4);
\fill[white] (8,4) -- (16,4) arc (450:270:4) -- (0,-4) arc (270:90:4) -- (8,4);
\node at (0,0) [draw,circle,inner sep=1pt] {};
\foreach \x in {4,6,14,16,26} {\node at (\x,0) [fill,circle,inner sep=1pt] {};}
\node at (10,0) {$\ldots$};
\draw[->] (13,6) --(0,6) arc (90:270:6) -- (26,-6) arc (270:450:6) -- (13,6);
\draw[->] (26,4) arc (450:90:4);
\draw[->] (8,4) -- (16,4) arc (450:270:4) -- (0,-4) arc (270:90:4) -- (8,4);
\foreach \x in {36,38,46,48} {\node at (\x,0) [fill,circle,inner sep=1pt] {};}
\node at (42,0) {$\ldots$};
\node at (56,-10) [anchor=south east] {$r$};

\begin{scope}[xshift=70mm]
\fill[black!20] (13,6) --(4,6) arc (90:270:6) -- (26,-6) arc (270:450:6) -- (13,6);
\fill[white] (26,4) arc (450:90:4);
\fill[white] (10,4) -- (16,4) arc (450:270:4) -- (4,-4) arc (270:90:4) -- (10,4);
\node at (-6,0) [draw,circle,inner sep=1pt] {};
\foreach \x in {4,6,14,16,26} {\node at (\x,0) [fill,circle,inner sep=1pt] {};}
\node at (10,0) {$\ldots$};
\draw[->] (13,6) --(4,6) arc (90:270:6) -- (26,-6) arc (270:450:6) -- (13,6);
\draw[->] (26,4) arc (450:90:4);
\draw[->] (10,4) -- (16,4) arc (450:270:4) -- (4,-4) arc (270:90:4) -- (10,4);
\foreach \x in {36,38,46,48} {\node at (\x,0) [fill,circle,inner sep=1pt] {};}
\node at (42,0) {$\ldots$};
\node at (52,-10) [anchor=south east] {$r^\prime$};
\end{scope}

\end{tikzpicture}
\caption{\small Pictures of maps $r,r^\prime \colon P^n \to C_k(\cyl)$ such that $r_*([\partial P^n]) = \Delta_{j+1}-\Delta_j-\tau_1$ and $r^{\prime}_*([\partial P^n]) = \tau_{j+1}-\tau_j-\tau_1$. In each case there are $j+1$ fixed points in the bounded white region and $k-j-2$ fixed points in the unbounded white region. The remaining point is in the shaded region; its position is parametrised by $P^n$.}\label{fPairofpants}
\end{figure}
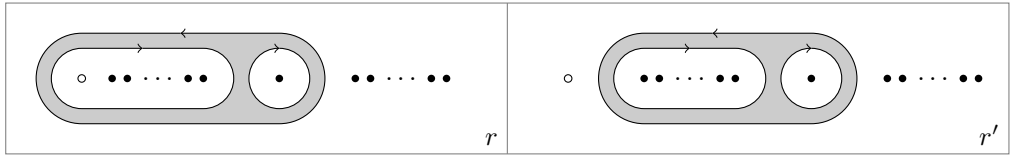

Now let $\hat{\Delta}$ denote the image of the fundamental class under the map $S^{n-1} \to C_k(\cyl)$ which sends $v$ to $\{v,2v,\dotsc,kv\}$. This map can be homotoped to the map
\begin{equation}\label{e:collapsing-equators}
S^{n-1} \longrightarrow S^{n-1}\vee\dotsb\vee S^{n-1} \longrightarrow C_k(\cyl)
\end{equation}
which collapses $k-1$ equators to get a wedge sum of $k$ copies of $S^{n-1}$ and then applies the maps $\Delta_0,\dotsc,\Delta_{k-1}$ on these summands. To see this, note that the $k-1$ equators divide $S^{n-1}$ into $k$ slices (see Figure \ref{f:sphere-slices}) and let $q_i\colon S^{n-1}\to S^{n-1}$ be the quotient map that collapses everything but the $i$th slice of the sphere to a point. Note that the map $q_{\{1,\ldots,k\}} \colon S^{n-1} \to C_k(\cyl)$ defined by sending $v\in S^{n-1}$ to $\{q_1(v), 2q_2(v),\ldots, kq_k(v)\}$ is precisely the map \eqref{e:collapsing-equators} described above. On the other hand, each $q_i$ is homotopic to the identity, and these homotopies assemble to give a homotopy from $q_{\{1,\ldots,k\}}$ to $v\mapsto \{v,2v,\ldots,kv\}$.

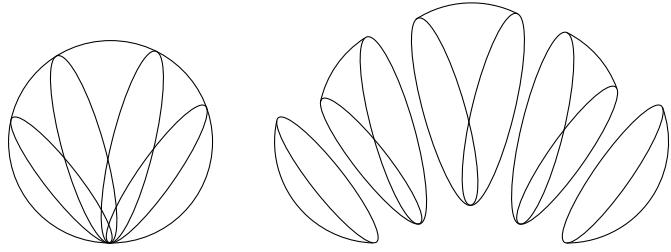
\begin{figure}[ht]
\centering
\begin{tikzpicture}
[x=10mm,y=10mm,scale=0.5]

\def\x{9.826695394059044}
\def\y{-1.3454014474356621}
\def\circleradius{2.6839256230691935cm}
\draw (\x,\y) circle (\circleradius);
% First ellipse
\draw [rotate around={-51.921952297553496:(8.53409490180479,-2.3549729771042025)}] (8.53409490180479,-2.3549729771042025) ellipse (2.093873126062705cm and 0.42737394601423706cm);
% Second ellipse
\draw [rotate around={-74.35019395297131:(9.119136664387407,-1.5442561401718586)}] (9.119136664387407,-1.5442561401718586) ellipse (2.5769635506472355cm and 0.5499969993457372cm);
% Third ellipse
\draw [rotate around={76.40308257125182:(10.399136664387399,-1.484256140171852)}] (10.399136664387399,-1.484256140171852) ellipse (2.615853429092247cm and 0.5568876407585037cm);
% Fourth ellipse
\draw [rotate around={54.892308432369305:(11.042240923024323,-2.1990304093806126)}] (11.042240923024323,-2.1990304093806126) ellipse (2.2121467742791974cm and 0.46375843253128557cm);

\def\centerarc(#1)(#2:#3:#4)% [draw options] (center) (initial angle:final angle:radius)
{ \draw ($(#1)+({#4*cos(#2)},{#4*sin(#2)})$) arc (#2:#3:#4); }

% First slice
\begin{scope}[xshift=7cm]
\draw [rotate around={-51.921952297553496:(8.53409490180479,-2.3549729771042025)}] (8.53409490180479,-2.3549729771042025) ellipse (2.093873126062705cm and 0.42737394601423706cm);
\centerarc(\x,\y)(168:268:\circleradius)
\end{scope}

% Second slice
\begin{scope}[xshift=8.15cm,yshift=5mm]
\draw [rotate around={-51.921952297553496:(8.53409490180479,-2.3549729771042025)}] (8.53409490180479,-2.3549729771042025) ellipse (2.093873126062705cm and 0.42737394601423706cm);
\draw [rotate around={-74.35019395297131:(9.119136664387407,-1.5442561401718586)}] (9.119136664387407,-1.5442561401718586) ellipse (2.5769635506472355cm and 0.5499969993457372cm);
\centerarc(\x,\y)(123:168:\circleradius)
\end{scope}

% Third slice
\begin{scope}[xshift=9.5cm,yshift=1cm]
\draw [rotate around={-74.35019395297131:(9.119136664387407,-1.5442561401718586)}] (9.119136664387407,-1.5442561401718586) ellipse (2.5769635506472355cm and 0.5499969993457372cm);
\draw [rotate around={76.40308257125182:(10.399136664387399,-1.484256140171852)}] (10.399136664387399,-1.484256140171852) ellipse (2.615853429092247cm and 0.5568876407585037cm);
\centerarc(\x,\y)(63:123:\circleradius)
\end{scope}

% Fourth slice
\begin{scope}[xshift=10.8cm,yshift=5mm]
\draw [rotate around={76.40308257125182:(10.399136664387399,-1.484256140171852)}] (10.399136664387399,-1.484256140171852) ellipse (2.615853429092247cm and 0.5568876407585037cm);
\draw [rotate around={54.892308432369305:(11.042240923024323,-2.1990304093806126)}] (11.042240923024323,-2.1990304093806126) ellipse (2.2121467742791974cm and 0.46375843253128557cm);
\centerarc(\x,\y)(20:63:\circleradius)
\end{scope}

% Fifth slice
\begin{scope}[xshift=12cm]
\draw [rotate around={54.892308432369305:(11.042240923024323,-2.1990304093806126)}] (11.042240923024323,-2.1990304093806126) ellipse (2.2121467742791974cm and 0.46375843253128557cm);
\centerarc(\x,\y)(-90:20:\circleradius)
\end{scope}

\end{tikzpicture}
\caption{The slices of $S^{n-1}$ when $k=5$.}\label{f:sphere-slices}
\end{figure}

Hence $\hat{\Delta} = \Delta_0 +\dotsb +\Delta_{k-1}$ and so by \eqref{eRelationA},
\begin{equation}\label{eRelationB}
\hat{\Delta} = k\Delta_0 + k(k-1)\pi.
\end{equation}
Similarly we let $\hat{\tau}$ denote the image of the fundamental class under the map $S^{n-1} \to C_k(\cyl)$ which sends $v$ to $p_1+\{0,v,2v,\dotsc,(k-1)v\}$, where $p_1$ is a fixed point in $\bR^n$ with $\lvert p_1\rvert \geq k$. Just as above, we can homotope this to see that $\hat{\tau} = \tau_1 +\dotsb \tau_{k-1}$ and so by \eqref{eRelationA},
\begin{equation}\label{eRelationC}
\hat{\tau} = k(k-1)\pi.
\end{equation}
One can see this very directly in the case $n=2$. In this case we are talking about $H_1(C_k(\bR^2);\bZ) = \beta_k / [\beta_k,\beta_k] = \bZ\{\pi\}$, where $\beta_k$ denotes the braid group on $k$ strands. Any one of the standard generators $\sigma_1,\dotsc,\sigma_{k-1}$ of $\beta_k$, which interchange two consecutive strands, is sent to the generator $\pi$. The element $\hat{\tau}$ is the image of the full twist of all $k$ strands, which can be written as a product of $k(k-1)$ generating elements, and so after abelianisation we have $\hat{\tau} = k(k-1)\pi$.

We now apply the above discussion to prove the following:
\begin{lemma}\label{lFormula}
For any map $f\colon S^{n-1} \to S^{n-1}$ define $\sigma_f\colon S^{n-1} \to C_k(\cyl)$ by sending $v$ to $\{ v,v+\frac{1}{k}f(v),\dotsc,v+\frac{k-1}{k}f(v) \}$. Denoting the image of the fundamental class under this map also by $\sigma_f$ we have
\begin{equation}\label{eRelationD}
\sigma_f = k\Delta_0 + \deg(f)k(k-1)\pi.
\end{equation}
\end{lemma}

\begin{proof}
Note that if $\deg(f)=1$ then $\sigma_f = \sigma_{\mathrm{id}} = \hat{\Delta}$ so this is just \eqref{eRelationB}. In general this can be seen as follows. Write $d=\deg(f)$ and first assume that $d>0$.

Denote the constant map to the basepoint by $*\colon S^{n-1} \to S^{n-1}$ and the map $S^{n-1} \to S^{n-1} \vee\dotsb\vee S^{n-1}$ which collapses $d-1$ equators by $c_d$. Then $\sigma_f$ can be homotoped to the map
\[
v \mapsto \bigl\lbrace s(v),s(v)+\tfrac{1}{k}g(v),\dotsc,s(v)+\tfrac{k-1}{k}g(v)\bigr\rbrace
\]
where $s=(\id+*+\dotsb +*)\circ c_d$ and $g=(\id +\id +\dotsb +\id)\circ c_d$, which is in turn homotopic to the map $(\hat{\Delta}+\hat{\tau}+\dotsb +\hat{\tau})\circ c_d \colon S^{n-1} \to C_k(\cyl)$. Therefore the homology class $\sigma_f$ is equal to $\hat{\Delta}+(d-1)\hat{\tau}$, which is the claimed formula by \eqref{eRelationB} and \eqref{eRelationC}.

If $d\leq 0$ we can instead take $s=(\id +*+\dotsb +*)\circ c_{2-d}$ and $g=(\id +r+\dotsb +r)\circ c_{2-d}$, where $r$ is a reflection of $S^{n-1}$, to see that $\sigma_f$ is homotopic to the map $(\hat{\Delta}+\hat{\tau}\circ r +\dotsb +\hat{\tau}\circ r)\circ c_{2-d}$. The image of the fundamental class $[S^{n-1}]$ under $\hat{\tau}\circ r$ is just $-\hat{\tau}$, so we again get that the homology class $\sigma_f$ is equal to $\hat{\Delta}+(d-1)\hat{\tau}$.
\end{proof}

\begin{remark}
Rationally, the $(n-1)$st homology of $C_k(\cyl)$ is known to be two-dimensional by the presentation of the bigraded $\bQ$-algebra $H_*(C_*(\cyl);\bQ)$ given in Proposition 3.4 of \cite{RW:tch}. Specifically, it is generated by the elements $\Delta_0$ and $\Delta_1$, corresponding to $[k-1]\cdot\Delta$ and $[k-2]\cdot\Delta\cdot [1]$ in the notation of the cited paper.
\end{remark}

%%%%%%%%%%%%%%%%%%%%%%%%%%%%%%%%%%%%%%%%%%%%%%%%%%%%%%%%%%%%%%%%%%%%%%%%%%%%%%%%%%
\subsection{Proof of Theorem \ref{thm:e}}\label{ssTheoremDproof}
%%%%%%%%%%%%%%%%%%%%%%%%%%%%%%%%%%%%%%%%%%%%%%%%%%%%%%%%%%%%%%%%%%%%%%%%%%%%%%%%%%

Abbreviate $C_l(M\smallsetminus\{0\})$ by just $C_l$. Fix a field $\bF$ of characteristic $p>0$ and write $\widetilde{H}_*(-) = \widetilde{H}_*(-;\bF)$. Recall from Lemma \ref{lCofibreSequence} that we have a long exact sequence on homology
\[
\cdots \longrightarrow \widetilde{H}_*(\Sigma^{n-1}((C_{k-1})_+)) \longrightarrow \widetilde{H}_*(C_k) \longrightarrow \widetilde{H}_*(C_k(M)) \longrightarrow \cdots
\]
and denote the left-hand map above by $T_{k,*}$. By exactness we have:
\begin{align}
\begin{split}\label{eDimensionFormula}
\dim (\widetilde{H}_* (C_k(M))) &= \dim (\mathrm{codomain}(T_{k,*})) + \dim (\mathrm{domain}(T_{k,*-1})) \\
&\phantom{=}\ - \mathrm{rank}(T_{k,*}) - \mathrm{rank}(T_{k,*-1}).
\end{split}
\end{align}
Hence in order to identify $\widetilde{H}_*(C_k(M))$ and $\widetilde{H}_*(C_{rk}(M))$ in a range it suffices to identify the linear maps $T_{k,*}$ and $T_{rk,*}$ in a range.

\begin{proof}[Proof of Theorem \ref{thm:e}]
Fix a positive integer $r\geq 2$ coprime to $p$. We will construct maps $a$, $b$ and $c$ such that the square
\begin{equation}\label{eSquare}
\centering
\begin{split}
\begin{tikzpicture}
[x=1mm,y=1mm]
\node (tl) at (0,15) {$S^{n-1}\times C_{k-1}$};
\node (tm) at (35,15) {$S^{n-1}\times C_{rk-r}$};
\node (tr) at (70,15) {$S^{n-1}\times C_{rk-1}$};
\node (bl) at (0,0) {$C_k$};
\node (br) at (70,0) {$C_{rk}$};
\draw[->] (tl) to node[above,font=\small]{$a$} (tm);
\draw[->] (tm) to node[above,font=\small]{$b$} (tr);
\draw[->] (bl) to node[below,font=\small]{$c$} (br);
\draw[->] (tl) to node[left,font=\small]{$t_{k-1}$} (bl);
\draw[->] (tr) to node[right,font=\small]{$t_{rk-1}$} (br);
\end{tikzpicture}
\end{split}
\end{equation}
commutes on homology with coefficients in $\bF$. Applying $\widetilde{H}_*(-)$ and restricting to a direct summand (see Remark \ref{rKuenneth}) on the top row gives a commutative square
\begin{equation}\label{eSquareOnHomology}
\centering
\begin{split}
\begin{tikzpicture}
[x=1mm,y=1mm]
\node (tl) at (0,15) {$\widetilde{H}_{*+1}(\Sigma^n(C_{k-1})_+)$};
\node (tm) at (40,15) {$\widetilde{H}_{*+1}(\Sigma^n(C_{rk-r})_+)$};
\node (tr) at (80,15) {$\widetilde{H}_{*+1}(\Sigma^n(C_{rk-1})_+)$};
\node (bl) at (0,0) {$\widetilde{H}_*(C_k)$};
\node (br) at (80,0) {$\widetilde{H}_*(C_{rk}).$};
\draw[->] (tl) to node[above,font=\small]{$\alpha$} (tm);
\draw[->] (tm) to node[above,font=\small]{$\beta$} (tr);
\draw[->] (bl) to node[below,font=\small]{$c_*$} (br);
\draw[->] (tl) to node[left,font=\small]{$T_{k,*}$} (bl);
\draw[->] (tr) to node[right,font=\small]{$T_{rk,*}$} (br);
\end{tikzpicture}
\end{split}
\end{equation}
Throughout this section we will abbreviate $\sr=\sr[M\smallsetminus\{0\}]$ and $\srscan=\srscan[M\smallsetminus\{0\}]$. Recall that we defined the function
\[
\lambda(k) = \lambda[M](k) = \mathrm{min}\{ \srscan(k), \srscan(k-1)+n-1, \sr(rk-i) \mid i=2,\dotsc,r \},
\]
where $n$ is the dimension of $M$. We will show that $\alpha$, $\beta$ and $c_*$ are isomorphisms in the range $*\leq\lambda(k)$, therefore identifying the maps $T_{k,*}$ and $T_{rk,*}$ in this range. Hence by \eqref{eDimensionFormula} the vector spaces $\widetilde{H}_*(C_k(M))$ and $\widetilde{H}_*(C_{rk}(M))$ have the same dimension for $*\leq\lambda(k)$, which is Theorem \ref{thm:e}.

\begin{remark}\label{rLambda}
The first two terms of $\lambda(k)$ come from our use of the replication map and Theorem \ref{thm:b}, which tells us that the $r$-replication map induces isomorphisms in the stable range $\srscan$. The remaining terms come from our use of the classical stabilisation map, which by definition induces isomorphisms in the range $\sr$. If we assume that $\sr$ is non-decreasing (so $\srscan=\sr$) and $r,k\geq 2$ then the range $*\leq\lambda(k)$ simplifies to
\[
*\leq \mathrm{min}\{ \sr(k),\sr(k-1)+n-1 \}.
\]
For example if $\sr(k)=ak+b$ then this is
\begin{align*}
* &\leq ak+b &&\text{if } n\geq a+1 \\
* &\leq ak+b - (a+1-n) &&\text{if } n<a+1,
\end{align*}
i.e.\ the same as the stable range, except possibly shifted down by a constant if the manifold is low-dimensional compared to the slope of the stable range.
\end{remark}

\paragraph{Constructing the maps.}
Fix a basepoint $0\in M$. By Lemma \ref{lExtendingVF} we can choose a vector field $v$ on $M$ which is non-vanishing except possibly at $0$. This has an associated one-parameter family of diffeomorphisms $\phi_t$. Define the $r$-replication map
\[
\rho_{r,k} \colon C_k(M\smallsetminus\{0\}) \longrightarrow C_{rk}(M\smallsetminus\{0\})
\]
to take a configuration $c=\{x_1,\dotsc,x_k\}$ to the configuration
\[
\{ \phi_{it/r}(x_1),\dotsc,\phi_{it/r}(x_k) \mid 0\leq i\leq r-1 \},
\]
where $t=t(c)>0$ is sufficiently small that $\phi_s(x_i)\neq \phi_u(x_j)$ for $s,u\in (0,t)$ unless $i=j$ and $s=u$. This agrees up to homotopy with the earlier definition of the $r$-replication map under the identifications $C_k^\delta(M\smallsetminus\{0\}) \simeq C_k(M\smallsetminus\{0\})$ and $C_{rk}^\delta(M\smallsetminus\{0\}) \simeq C_{rk}(M\smallsetminus\{0\})$. We now define
\begin{align*}
a &= \mathrm{id}\times \rho_{r,k-1} \\
b &= (\mathrm{pr}_1, t_{rk-2}) \circ\dotsb\circ (\mathrm{pr}_1, t_{rk-r}) \\
c &= \rho_{r,k}.
\end{align*}
In other words $a$ and $c$ replace each point of the configuration by $r$ copies in the direction determined by the vector field, whereas $b$ adds $r-1$ new points near the missing point $0$ in the direction determined by the vector in $S^{n-1}$.

\paragraph{Isomorphisms in a range.}
The vector field is non-vanishing on $M\smallsetminus\{0\}$, so Theorem \ref{thm:b} tells us that the $r$-replication map $\rho_{r,k}$ induces isomorphisms in the stable range on homology with $\bZ_{(p)}$ coefficients, and hence also with $\bF$ coefficients. Hence $c_*$ is an isomorphism in the stable range $*\leq\srscan(k)$.

The map $\rho_{r,k-1}$ induces isomorphisms on $\bF$-homology up to degree $\srscan(k-1)$, so its suspension $\Sigma^n((\rho_{r,k-1})_+)$ induces isomorphisms up to degree $\srscan(k-1)+n$. The map that this induces on $\widetilde{H}_{*+1}(-)$ is $\alpha$, which is therefore an isomorphism in the range $*\leq\srscan(k-1)+n-1$.

For the map $\beta$ consider the map of (trivial) fibre bundles
\begin{center}
\begin{tikzpicture}
[x=1mm,y=1mm]
\node (tl) at (-20,10) {$S^{n-1}\times C_{rk-i}$};
\node (tr) at (20,10) {$S^{n-1}\times C_{rk-i+1}$};
\node (b) at (0,0) {$S^{n-1}$};
\draw[->] (tl) to node[above,font=\small]{$(\mathrm{pr}_1,t_{rk-i})$} (tr);
\draw[->] (tl) to (b);
\draw[->] (tr) to (b);
\end{tikzpicture}
\end{center}
for $i=2,\dotsc,r$. Its fibre over a point in $S^{n-1}$ is the classical stabilisation map and therefore induces isomorphisms on $\bF$-homology up to degree $\sr(rk-i)$. Hence by the relative Serre spectral sequence the map $(\mathrm{pr}_1,t_{rk-i})$ also induces isomorphisms on $\bF$-homology in this range. So the map $b$ induces isomorphisms on $\bF$-homology up to degree $\min\{ \sr(rk-i) \mid 2\leq i\leq r \}$.

In general, for a map $f\colon S^d \times A \to S^d \times B$ over $S^d$, the map on homology under the K{\"u}nneth isomorphism, $f_*\colon H_*(A)\oplus H_{*-d}(A) \to H_*(B)\oplus H_{*-d}(B)$, is triangular -- more precisely the component $H_*(A)\to H_{*-d}(B)$ is zero. To see this note that a representing cycle $c$ for an element in the $H_*(A)$ component can be taken to have support in a single fibre. Since $f$ is a map over $S^d$ the image $f_{\sharp}(c)$ will also have support in a single fibre, and therefore the image $f_*([c])$ will be in the $H_*(B)$ component of the K{\"u}nneth decomposition of the right-hand side. Hence if $f$ induces an isomorphism on homology, it also restricts to isomorphisms between each of the direct summands in the K\"unneth decompositions of the source and target. Applying this fact to the map $b$ we obtain that $\beta$ is an isomorphism in the range $*\leq\min\{ \sr(rk-i) \mid 2\leq i\leq r \}$.

Hence each of $\alpha$, $\beta$ and $c_*$ are isomorphisms in the range $*\leq\lambda(k)$.

\paragraph{Commutativity.}
It therefore remains to show that the square \eqref{eSquare} commutes on $\bF$-homology. Choose a coordinate neighbourhood $U\cong \bR^n$ of $0\in M$ and define the map
\begin{equation}\label{eModuleStructure}
\zeta\colon C_r(\cyl) \times C_{k-1}(M\smallsetminus\{0\}) \longrightarrow C_{rk}(M\smallsetminus\{0\})
\end{equation}
to first apply the map $\rho_{r,k-1}$ to the configuration in $M\smallsetminus\{0\}$, i.e.\ replace each point by $r$ copies according to the vector field, then push the resulting configuration radially away from $0$ so that it is disjoint from $U$, and finally insert the configuration of $r$ points in $\cyl = U\smallsetminus\{0\}$ into the vacated space. Choosing a trivialisation of $TM$ over $U\cong \bR^n$, the vector field $v$ restricts to a map $\bR^n\to \bR^n$ which is non-vanishing on $S^{n-1}$, so we may rescale it to obtain a map $f\colon S^{n-1}\to S^{n-1}$. Recall from Lemma \ref{lFormula} that such a map induces a map $\sigma_f\colon S^{n-1}\to C_r(\cyl)$. One can then easily see that the two ways $c\circ t_{k-1}$ and $t_{rk-1}\circ b\circ a$ around the square \eqref{eSquare} are homotopic to
\[
\zeta \circ (\sigma_f \times \id) \;\text{ and }\; \zeta \circ (\sigma_{\id} \times \id) \colon\; S^{n-1}\times C_{k-1}(M\smallsetminus\{0\}) \longrightarrow C_{rk}(M\smallsetminus\{0\})
\]
respectively. It suffices to show that $\sigma_f$ and $\sigma_{\id} \colon S^{n-1}\to C_r(\cyl)$ induce the same map on $\bF$-homology, and we only need to check this on the fundamental class. Using our abuse of notation from \S\ref{ssCylinders} this means that we just need to check that the homology classes $\sigma_f$ and $\sigma_{\id}$ in $H_{n-1}(C_r(\cyl);\bF)$ are equal.

The degree of $f\colon S^{n-1}\to S^{n-1}$ is $\chi$ by the Poincar{\'e}-Hopf theorem (c.f.\ Definition \ref{dIndex}) so by Lemma \ref{lFormula} we have
\begin{align*}
\sigma_f &= r\Delta_0 + \chi r(r-1)\pi \\
\sigma_{\id} &= r\Delta_0 + r(r-1)\pi
\end{align*}
in $H_{n-1}(C_r(\cyl);\bZ)$. Their difference is $(\chi-1)r(r-1)\pi$, which is divisible by $p=\mathrm{char}(\bF)$ by hypothesis. Hence the difference $\sigma_f -\sigma_{\id}$ is indeed zero in $H_{n-1}(C_r(\cyl);\bF)$, and so the square \eqref{eSquare} commutes on $\bF$-homology.
\end{proof}

%%%%%%%%%%%%%%%%%%%%%%%%%%%%%%%%%%%%%%%%%%%%%%%%%%%%%%%%%%%%%%%%%%%%%%%%%%%%%%%%%%
\subsection{The case of the two-sphere}
%%%%%%%%%%%%%%%%%%%%%%%%%%%%%%%%%%%%%%%%%%%%%%%%%%%%%%%%%%%%%%%%%%%%%%%%%%%%%%%%%%

For $M=S^2$ we have the well-known calculation $H_1(C_k(S^2);\bZ) \cong \bZ/(2k-2)\bZ$ for $k\geq 2$ obtained from a presentation for $\pi_1(C_k(S^2))$ (see \cite{FadellVan1962}). The degree-one $\bF_p$-homology is therefore either one- or zero-dimensional, depending on whether $p\mid 2k-2$ or not. So the statement of Theorem \ref{thm:e} in degree $1$ for $M=S^2$ for mod-$p$ coefficients reduces to the following purely number-theoretic statement: if $p$ is a prime and $r$ is a positive integer such that $p\mid r-1$ then $p\mid 2k-2$ if and only if $p\mid 2rk-2$. This is of course obviously true: we have $r-1 = ap$ for some $a$, so $2rk-2 = 2k + 2kap -2 \equiv 2k-2 \;(\text{mod } p)$.

%%%%%%%%%%%%%%%%%%%%%%%%%%%%%%%%%%%%%%%%%%%%%%%%%%%%%%%%%%%%%%%%%%%%%%%%%%%%%%%%%%
\subsection{Generalisation to configurations with labels in a bundle}\label{ssDwithlabels}
%%%%%%%%%%%%%%%%%%%%%%%%%%%%%%%%%%%%%%%%%%%%%%%%%%%%%%%%%%%%%%%%%%%%%%%%%%%%%%%%%%

Theorem \ref{thm:e} generalises directly to configuration spaces $C_k(M;\theta)$ with labels in a bundle $\theta\colon E\to M$ with path-connected fibres.

\begin{rthm}[$\text{D}^\prime$]\label{thm:dprime}
Let $M$ be a closed, connected, smooth manifold with Euler characteristic $\chi$ and let $\theta\colon E\to M$ be a fibre bundle with path-connected fibres. Choose a field $\bF$ of positive characteristic $p$ and let $r\geq 2$ be an integer coprime to $p$ such that $p$ divides $(\chi-1)(r-1)$. Then there are isomorphisms
\[
H_*(C_k(M;\theta);\bF) \;\cong\; H_*(C_{rk}(M;\theta);\bF)
\]
in the range $*\leq\lambda[M;\theta](k)$.
\end{rthm}

The function $\lambda[M;\theta]$ is defined just as $\lambda[M]$, namely:
\[
\lambda[M;\theta](k) = \mathrm{min}\{ \srscan(k), \srscan(k-1)+n-1, \sr(rk-i) \mid i=2,\dotsc,r \},
\]
where $\sr=\sr[M^\star;\theta^\star]$, $\srscan=\srscan[M^\star;\theta^\star]$ and $M^\star$, $\theta^\star$ denote $M\smallsetminus\point$ and $\theta|_{M\smallsetminus\point}$ respectively. These two functions are defined analogously to Definition \ref{dStableRange}, using the stabilisation and scanning maps for configuration spaces with labels in a bundle.

In the remainder of this subsection we sketch how to generalise the proof of Theorem \ref{thm:e} to a proof of Theorem \hyperref[thm:dprime]{$\text{D}^\prime$}. The proof follows the same steps. In \S\ref{ssCofibresequence} one has to additionally choose a trivialisation of $\theta$ over the embedded disc $D\subseteq M$, and analogously to Lemma \ref{lCofibreSequence} there is a cofibre sequence
\[
C_k(M^\star;\theta^\star) \longrightarrow C_k(M;\theta) \dashrightarrow \Sigma^n((F\times C_{k-1}(M^\star;\theta^\star))_+),
\]
where $F$ is the typical fibre of $\theta$. The description of the connecting homomorphism for the long exact sequence on homology is exactly analogous, using the trivialisation of $\theta$ over $D$ to determine the label of the new point which is added to the configuration near $0\in D$.

In the diagram \eqref{eSquare} the top three spaces are replaced by their cartesian products with $F$. The maps $c_*$ and $\alpha$ are isomorphisms in the range $*\leq\lambda(k)$ for the same reasons as before, using Theorem \hyperref[thm:bprime]{$\text{B}^\prime$} instead of Theorem \ref{thm:b}. The map $b$ is a composition of maps of fibre bundles over $F\times S^{n-1}$ and the maps of fibres are classical stabilisation maps for configuration spaces with labels in a bundle, and so are isomorphisms on homology in the stable range for the stabilisation map (c.f.\ Proposition \ref{pImprovedRangeTwisted} and the appendix of \cite{KupersMiller2014Encellattachments}). The rest of the argument that $\beta$ is an isomorphism in the range $*\leq\lambda(k)$ goes through as before.

For commutativity: the map $\zeta$ can be defined similarly, using the chosen trivialisation of $\theta$ over $D$. The input is now a configuration of $k-1$ points in $M\smallsetminus\{0\}$ with labels in $\theta$ and a configuration of $r$ points in $\cyl$ with labels in the trivial bundle with fibre $F$, and the output is a configuration of $rk$ points in $M\smallsetminus\{0\}$ with labels in $\theta$. The map $f\colon S^{n-1}\to S^{n-1}$, corresponding to the restriction of the vector field to $\partial D$, induces a map $\sigma_f\colon F\times S^{n-1}\to C_r(\cyl;F)$. The two ways around the square \eqref{eSquare} are homotopic to $\zeta \circ (\sigma_f \times \mathrm{id})$ and $\zeta \circ (\sigma_{\mathrm{id}} \times \mathrm{id})$. Hence we just need to show that $\sigma_f$ and $\sigma_{\mathrm{id}}\colon F\times S^{n-1}\to C_r(\cyl;F)$ induce the same map on $\bF$-homology.

Now as in \S\ref{ssCylinders} we need to find formulas, in terms of more basic classes, for $(\sigma_{f})_*(x)$, for any class $x\in H_*(F\times S^{n-1})$. Previously we showed that when $F=*$ and $x=[S^{n-1}]$ we have
\[
(\sigma_f)_*([S^{n-1}]) = r\Delta_0 + \deg(f)r(r-1)\pi \in H_{n-1}(C_r(\cyl);\bZ).
\]
By the K\"unneth decomposition $H_*(F\times S^{n-1}) = H_*(F)\oplus H_{*-n+1}(F)$ it suffices to show that $(\sigma_f)_*(x\times [*]) - (\sigma_{\mathrm{id}})_*(x\times [*])$ and $(\sigma_f)_*(x\times [S^{n-1}]) - (\sigma_{\mathrm{id}})_*(x\times [S^{n-1}])$ are zero on $\bF$-homology for any class $x\in H_*(F)$. It is easy to see that in fact
\[
(\sigma_f)_*(x\times [*]) = (\sigma_{\mathrm{id}})_*(x\times [*]) \in H_*(C_r(\cyl;F);\bZ)
\]
and therefore also on $\bF$-homology. One can define classes $\pi(x), \Delta_i(x), \tau_i(x)$ etc.\ in $H_{*+n-1}(C_r(\cyl;F);\bZ)$ just as in \S\ref{ssCylinders} and by the same arguments as before we have
\[
(\sigma_f)_*(x\times [S^{n-1}]) = r\Delta_0(x) + \deg(f)r(r-1)\pi(x) \in H_{*+n-1}(C_r(\cyl;F);\bZ).
\]
Hence $(\sigma_f)_*(x\times [S^{n-1}]) - (\sigma_{\mathrm{id}})_*(x\times [S^{n-1}])$ is equal to $(\deg(f)-1)r(r-1)\pi(x) = (\chi-1)r(r-1)\pi(x)$ and is therefore zero on $\bF$-homology since $p$ divides $(\chi-1)(r-1)$. This completes the sketch of the proof of Theorem \hyperref[thm:dprime]{$\text{D}^\prime$}.

%%%%%%%%%%%%%%%%%%%%%%%%%%%%%%%%%%%%%%%%%%%%%%%%%%%%%%%%%%%%%%%%%%%%%%%%%%%%%%%%%%
\subsection{Combining Theorems \ref{thm:a} and \ref{thm:e}}\label{ssCorollaryE}
%%%%%%%%%%%%%%%%%%%%%%%%%%%%%%%%%%%%%%%%%%%%%%%%%%%%%%%%%%%%%%%%%%%%%%%%%%%%%%%%%%

We now prove Corollary \ref{coro:stable-homologies}, concerning the homology of configuration spaces on even-dimensional manifolds with coefficients in a field of odd characteristic. In fact, our methods also partially recover the known homological stability results for odd-dimensional manifolds and for fields of characteristic $2$ or $0$. The complete statement of what may be deduced by combining Theorems \ref{thm:a} and \ref{thm:e} is as follows:

\begin{corollary}\label{coro:stable-homologies-extended}
Let $M$ be a closed, connected, smooth manifold with Euler characteristic $\chi$ and let $\bF$ be a field of characteristic $p$. Then, in the stable range \textup{(}resp.\ the range $*\leq\lambda(k)$ for lines 4--9\textup{)}, the homology group $H_*(C_k(M);\bF)$ depends only on the quantity stated in Table \ref{table:stable-homologies}.
\end{corollary}

\begin{table}[ht]
\centering
\setlength{\tabcolsep}{1ex}
\begin{tabular}{lcllcc}
\toprule
dimension & \multicolumn{2}{c}{conditions} & $H_*(C_k(M);\bF)$ depends only on & $\sharp$ & \\
\midrule
odd & $p\neq 2$ & --- & --- & $1$ & \ref{thm:a} \\
& $p=2$ & --- & parity of $k$ & $2$ & \ref{thm:a} \\
\midrule
even & $p=0$ & --- & whether $2k=\chi$ & $1$ or $2^*$ & \ref{thm:a} \\
\cmidrule{2-6}
& $p$ odd & --- & $\mathrm{min}\{(2k-\chi)_p,(\chi)_p+1\}$ & $(\chi)_p+2$ & \ref{thm:a},\ref{thm:e} \\
& & $\chi\not\equiv 0 \text{ mod } p$ & whether $p$ divides $2k-\chi$ & $2$ & \ref{thm:a},\ref{thm:e} \\
& & $\chi\equiv 1 \text{ mod } p$ & --- & $1$ & \ref{thm:a},\ref{thm:e} \\
\cmidrule{2-6}
& $p=2$ & --- & $(k)_2$ & $\infty$ & \ref{thm:e} \\
& & $(\chi)_2\geq 1$ & $\mathrm{min}\{(k)_2,(\chi)_2\}$ & $(\chi)_2 + 1$ & \ref{thm:a},\ref{thm:e} \\
& & $(\chi)_2=1$ & parity of $k$ & $2$ & \ref{thm:a},\ref{thm:e} \\
\bottomrule
\end{tabular}
\caption{\small The second-from-right column is the maximum number of different values of $H_*(C_k(M);\bF)$ in the stable range (resp.\ the range $*\leq\lambda(k)$ in lines 4--9). The rightmost column indicates which theorem(s) each line follows from. \\ ${}^*$ There is only one ``stable homology'' on the third line, although when $\chi$ is even there may be a single exception in the stable range when $k=\frac{\chi}{2}$.}
\label{table:stable-homologies}
\end{table}

\begin{proof}
The first two lines follow directly from the odd part of Theorem \ref{thm:a}, noting that a map of spaces which induces an isomorphism with $\bZ[\frac12]$ coefficients also induces isomorphisms with coefficients in any field of characteristic different from $2$. The third line follows directly from the even part of Theorem \ref{thm:a}, since a map of spaces inducing isomorphisms on homology with $\bZ_{(p)}$ coefficients also induces isomorphisms with $\bQ$ or $\bF_p$ coefficients, and therefore with coefficients in any field of characteristic $0$ or $p$. (Given $j,k$ in the stable range and not equal to $\frac{\chi}{2}$, choose any prime $p$ which divides neither $2j-\chi$ nor $2k-\chi$ and apply Theorem \ref{thm:a} to get an isomorphism with $\bZ_{(p)}$ coefficients.)

For the fourth line there are two cases to consider. For the first case suppose that $\val{p}{2j-\chi} = \val{p}{2k-\chi} \leq \val{p}{\chi}$. In this case the result follows from Theorem \ref{thm:a}. Now suppose that $\val{p}{2j-\chi},\val{p}{2k-\chi} > \val{p}{\chi}$ and write $x^\prime$ for $x/p^{\val{p}{x}}$ for each number $x$. The assumption implies that $\val{p}{j} = \val{p}{\chi} = \val{p}{k}$ and that $p$ divides both $2k'-\chi'$ and $2j'-\chi'$, so that $j'\equiv k' \not\equiv 0 \text{ mod } p$. Hence we may choose $l$ such that $lk^\prime \equiv 1 \text{ mod } p$, so that by Theorem \ref{thm:e} we have:
\begin{align*}
p|(lj'-1) &\quad\text{so}\quad H_*(C_k(M);\bF_p) \cong H_*(C_{lj'k}(M);\bF_p) \text{ for } *\leq\mathrm{min}(\lambda(k),\lambda(lj'k)), \\
p|(lk'-1) &\quad\text{so}\quad H_*(C_j(M);\bF_p) \cong H_*(C_{lk'j}(M);\bF_p) \text{ for } *\leq\mathrm{min}(\lambda(j),\lambda(lk'j)).
\end{align*}
Since $lj'k = lk'j$ we have the required isomorphism in the intersection of these two ranges. Note that $l$ may be chosen arbitrarily large and $\lambda$ is divergent, so this is precisely the range $*\leq\mathrm{min}(\lambda(k),\lambda(j))$.

The fifth line is the special case of the fourth line when $(\chi)_p = 0$.

Sixth line: Suppose that $*\leq\mathrm{min}(\lambda(k),\lambda(j))$; we will consider three cases. First, if $j$ and $k$ are both in $p\bZ$ then $(2j-\chi)_p = 0 = (2k-\chi)_p$, so we have an isomorphism $H_*(C_j(M);\bF) \cong H_*(C_k(M);\bF)$ by Theorem \ref{thm:a}. Second, if $j$ and $k$ are both not in $p\bZ$ then $(j)_p = 0 = (k)_p$ and the isomorphism follows from Theorem \ref{thm:e} (since $p\mid \chi-1$, we may take $r$ in Theorem \ref{thm:e} to be any integer coprime to $p$). Finally, suppose that $j\in p\bZ$ and $k\notin p\bZ$. Then
\[
(2j-\chi)_p = 0 = (2(jkl+\chi)-\chi)_p \qquad\text{and}\qquad (jkl+\chi)_p = 0 = (k)_p
\]
for any $l$. Since $\lambda$ diverges and $l$ may be chosen arbitrarily large we have an isomorphism $H_*(C_j(M);\bF) \cong H_*(C_k(M);\bF)$ by Theorems \ref{thm:a} and \ref{thm:e}.

The seventh line follows directly from Theorem \ref{thm:e}: when $p=2$ we may take $r$ to be any odd integer, so there are isomorphisms in the range $*\leq\mathrm{min}(\lambda(j),\lambda(k))$ between any $j,k$ with the same $2$-adic valuation. To deduce the eighth line from the seventh we need to show that there are isomorphisms in this range whenever $(j)_2$ and $(k)_2$ are both at least $(\chi)_2$. In this case we have $(2k-\chi)_2 = (\chi)_2 = (2j-\chi)_2$, and so we can apply Theorem \ref{thm:a} (since we assumed that $\chi$ is even). Finally, the ninth line is a special case of the eighth line.
\end{proof}

\appendix
\section{The stability range for the torsion in configuration spaces}\label{appendixA}

In this appendix we show that the stable range for homological stability of unordered configuration spaces may be improved to have slope $1$ when taking $\bZ[\frac12]$ coefficients. We note that this has also recently been proved by a different method by \cite{KupersMiller2014}. We begin by proving this in a larger range when $M=\bR^n$ using Salvatore's description \cite{Salvatore:conf-sphere} of Cohen's calculations \cite{CLM}, and then use this to deduce the slope $1$ statement for general open, connected manifolds $M$. Our method for this second step is a slight variation of an argument due to Oscar Randal-Williams in \cite[\S 8]{RW-hs-for-ucs}.

From Salvatore's description \cite[page 537]{Salvatore:conf-sphere} of the homology of $C(\bR^n)$ (based on \cite[page 227]{CLM}), we obtain the following. A non-empty sequence of positive integers $I=(i_1,\ldots,i_{\ell(I)})$ with $\ell(I)\geq 0$ is said to be \emph{$(n,p)$-admissible} if it is weakly monotone and strictly bounded above by $n$. If $p$ is odd, an admissible function for $I$ is a function $\epsilon \colon \{1,\dotsc,\ell(I)\}\to \{0,1\}$ satisfying 
\begin{align*}
\epsilon(j) &\equiv i_j + i_{j-1} \text{ mod } 2 \qquad\text{for } 2\leq j\leq \ell(I).
\end{align*}
Observe that $\epsilon$ is determined by $I$ and $\epsilon(1)$. If $p=2$, define $\epsilon$ to be constant with value $0$.
 
If $p=2$, then $H_*(C(\bR^n);\bF_2)$ is isomorphic to the free commutative graded algebra generated by the symbols $Q_{\epsilon(1),I}(\iota)$ where $I$ is an $(n,p)$-admissible sequence and $\epsilon$ is an admissible function.

If $p$ is odd, then $H_*(C(\bR^n);\bF_p)$ is isomorphic to the free commutative graded algebra generated by the symbols $Q_{\epsilon(1),I}(\iota)$ where $I$ is an $(n,p)$-admissible sequence with $i_{\ell(I)}$ even and $\epsilon$ is an admissible function for $I$, and also (if $n$ is even) the symbols $Q_{\epsilon(1),I}([\iota,\iota])$ where $I$ is an $(n,p)$-admissible sequence with $i_{\ell(I)}$ odd and $\epsilon$ is an admissible function for $I$.

The homological degrees of $\iota:=Q_{\varnothing,\varnothing}(\iota)$ and $[\iota,\iota]:=Q_{\varnothing,\varnothing}([\iota,\iota])$ are $0$ and $n-1$, and the configuration degrees are $1$ and $2$. The homological and configuration degrees of the other generators are
\begin{align*}
h(Q_{\epsilon(1),(i_1,\ldots,i_k)}(\alpha)) &= ph(Q_{\epsilon(2),(i_2,\ldots,i_k)}(\alpha)) + i_1(p-1) - \epsilon(1) \\
%\nu(Q_{i_1,\ldots,i_k}(x)) &= ph(Q_{i_2,\ldots,i_k}(x)) + i_1(p-1) &
\nu(Q_{\epsilon(1),(i_1,\ldots,i_k)}(\alpha)) &= p\nu(Q_{\epsilon(2),(i_2,\ldots,i_k)}(\alpha)),
\end{align*}
where $\alpha=\iota$ or $[\iota,\iota]$. The degrees of a product of generators are:
\begin{align*}
h(xy) &= h(x)+ h(y),& \nu(xy) &= \nu(x)+ \nu(y).
\end{align*}
Multiplication by the class $\iota$ raises the configuration degree by $1$ and hence defines a homomorphism
\[
H_*(C_{k-1}(\bR^n))\lra H_*(C_{k}(\bR^n)),
\]
which is the same as that induced by the stabilisation map.

We say that a class in $H_*(C_k(\bR^n);\bF_p)$ is \emph{$p$-inceptive} if it is not in the image of the stabilisation map $C_{k-1}(\bR^n) \to C_k(\bR^n)$ on mod-$p$ homology. By the above a class is $p$-inceptive if and only if it is not in the principal ideal generated by $\iota$.

\begin{lemma}\label{lem:improved-range}
In $H_*(C_{k}(\bR^n);\bF_p)$ the first $p$-inceptive class in a fixed configuration degree $k$ is given in Table \ref{table:p-inceptive-classes}, where $a=\lfloor k/p\rfloor$ and $m_p(k)$ is the remainder after dividing $k$ by $p$. Any case not covered in the table has no $p$-inceptive classes. Hence by the above discussion the stabilisation map $C_{k-1}(\bR^n) \to C_k(\bR^n)$ induces an isomorphism on $H_*(-;\bF_p)$ for smaller homological degrees.
\end{lemma}

\begin{table}[ht]
\begin{adjustbox}{center}
\centering
\small
\begin{tabular}{ccccc}
\toprule
$p$ & $n$ & $k$ & \textbf{class} & \textbf{homological degree} \\
\midrule
even & all & even & $Q_{0,(1)}(\iota)^a$ & $a$ \\
\midrule
odd & odd & $\in p\bZ$ & $Q_{1,(2)}(\iota)^a$ & $a(2(p-1)-1)$ \\
\midrule
\begin{tabular}{c} odd \\\midrule $3,5$ \end{tabular} & \begin{tabular}{c} $\geq 6$, even \\\midrule $4$ \end{tabular} & $\geq p$, odd & $Q_{1,(2)}(\iota)^{a}[\iota,\iota]^{m_p(k)/2}$ & $a(2(p-1)-1) + (n-1)m_p(k)/2$ \\
\midrule
\begin{tabular}{c} odd \\\midrule $3,5$ \end{tabular} & \begin{tabular}{c} $\geq 6$, even \\\midrule $4$ \end{tabular} & even & $Q_{1,(2)}(\iota)^{a-1}[\iota,\iota]^{(p+m_p(k))/2}$ & $(a-1)(2(p-1)-1)+ (n-1)(p+m_p(k))/2$ \\
\midrule
$\neq 2,3,5$ & $4$ & \begin{tabular}{c} even \\\midrule $\geq p$, odd \end{tabular} & $Q_{1,(2)}(\iota)^{m_2(k)}[\iota,\iota]^{\lfloor k/2\rfloor}$ & $m_2(k)(2(p-1)-1) + (n-1)\lfloor k/2\rfloor$ \\
\midrule
odd & $2$ & even & $[\iota,\iota]^{k/2}$ & $k/2$ \\
\bottomrule
\end{tabular}
\end{adjustbox}
\caption{The first $p$-inceptive class in degree $k$.}
\label{table:p-inceptive-classes}
\end{table}

\begin{proof}
First observe that
\begin{align*}
h(Q_{\epsilon(1),(i_1,\ldots,i_k)}(\iota))&\geq h(Q_{\epsilon(i_2),(i_2,\ldots,i_k)}(\iota)^p) \\
\nu(Q_{\epsilon(1),(i_1,\ldots,i_k)}(\iota)) &= \nu(Q_{\epsilon(i_2),(i_2,\ldots,i_k)}(\iota)^p) \\
h(Q_{\epsilon(1),(i_1,\ldots,i_k)}([\iota,\iota]))&\geq h(Q_{\epsilon(i_2),(i_2,\ldots,i_k)}([\iota,\iota])^p) \\
\nu(Q_{\epsilon(1),(i_1,\ldots,i_k)}([\iota,\iota])) &= \nu(Q_{\epsilon(i_2),(i_2,\ldots,i_k)}([\iota,\iota])^p)\\
h(Q_{1,(i)}(\iota)) &\leq h(Q_{\epsilon,(j)}(\iota)) \\
\nu(Q_{1,(i)}(\iota)) &= \nu(Q_{\epsilon,(j)}(\iota)) 
\end{align*}
where $i\leq j$ in the bottom two rows. Hence the lowest $p$-inceptive class in a fixed configuration degree is a product whose factors are
\[
\begin{cases}
Q_1(\iota) & \text{ if $p=2$} \\
Q_{1,(2)}(\iota) & \text{ if $p$ odd and $n$ odd} \\
Q_{1,(2)}(\iota), [\iota,\iota] & \text{ if $p$ odd and $n$ even} \\
[\iota,\iota] &\text{ if $p$ is odd and $n=2$.}
\end{cases}
\]
This is enough to deduce the first two rows of the table, as well as the sixth. Second observe that, if $p$ is odd and $n$ is even, the first configuration degree in which a power of $Q_{1,(2)}(\iota)$ and a power of $[\iota,\iota]$ both live is $2p$, where $\nu(Q_{1,(2)}(\iota)^{2})=\nu([\iota,\iota]^p)$, and \[h(Q_{1,(2)}(\iota)^{2}) = 4p-6< p(n-1) = h([\iota,\iota]^p) \Leftrightarrow \text{ $n\geq 6$ or $n=4, p=3,5$,}\]
from which the third, fourth and fifth rows of the table follow.
\end{proof}

Lemma \ref{lem:improved-range} tells us in particular that for odd primes $p$ the stabilisation map $C_k(\bR^n) \to C_{k+1}(\bR^n)$ induces an isomorphism on homology with $\bF_p$ coefficients in the range $*\leq k$. We now show that this implies that the same is true for the stabilisation map $C_k(M) \to C_{k+1}(M)$ for any smooth, connected, open manifold $M$ of dimension at least $3$. Our method for this is a slight variation of an argument due to Oscar Randal-Williams in \cite[\S 8]{RW-hs-for-ucs}.

\begin{proposition}\label{pImprovedRange}
Let $M$ be a smooth, connected, open $n$-manifold with $n\geq 3$ and let $A$ be an abelian group. If the stabilisation map on $A$-homology
\[
H_*(C_k(\bR^n);A) \longrightarrow H_*(C_{k+1}(\bR^n);A)
\]
is an isomorphism in the range $*\leq k$ then so is the stabilisation map on $A$-homology
\[
H_*(C_k(M);A) \longrightarrow H_*(C_{k+1}(M);A).
\]
So by Lemma \ref{lem:improved-range} the stabilisation map $C_k(M)\to C_{k+1}(M)$ induces isomorphisms on homology with $\bF_p$ coefficients in the range $*\leq k$ for any odd prime $p$.
\end{proposition}

This result has also been recently proved by \cite{KupersMiller2014} using a different method along the lines of \cite{SegalRational}.

\begin{proof}
We will just write $H_*(-)$ for $H_*(-;A)$. Define $R_k(M)$ to be the homotopy cofibre of the stabilisation map $C_k(M)\to C_{k+1}(M)$. Now the stabilisation map $C_k(M)\to C_{k+1}(M)$ is split-injective on homology (see \cite[page 103]{McDuff}) so it induces an isomorphism on homology in degree $*$ if and only if $\widetilde{H}_*(R_k(M))=0$. So the hypothesis of the proposition says that $\widetilde{H}_*(R_k(\bR^n))=0$ for $*\leq k$ and we would like to show that $\widetilde{H}_*(R_k(M))=0$ for $*\leq k$. We refer to \cite{RW-hs-for-ucs} for background and any details which we omit in this proof -- the line of argument is very similar. The proof is by induction on $k$. The base case $k=0$ is obvious so we now fix $k\geq 1$ for the inductive step.

For $i\geq 0$ let $C_l(M)^i$ be the space of $l$-point subsets $c$ of $M$ together with an injection $\{0,\dotsc,i\}\to c$. These fit together to form an semi-simplicial space $C_l(M)^\bullet$ augmented by $C_l(M)$. The stabilisation map lifts to a map $C_l(M)^\bullet \to C_{l+1}(M)^\bullet$ of augmented semi-simplicial spaces. There is a fibre bundle $\pi\colon C_l(M)^i \to \widetilde{C}_{i+1}(M)$, where $\widetilde{C}$ denotes the \emph{ordered} configuration space, given by sending an injection $\{0,\dotsc,i\}\to c$ to its image and remembering the induced ordering. Its fibre over a point is homeomorphic to $C_{l-i-1}(M_{i+1})$, where $M_{i+1}$ denotes the manifold $M$ with $i+1$ points removed. Moreover the projection $\pi$ commutes with the stabilisation map $C_l(M)^i \to C_{l+1}(M)^i$ and the map of fibres over a point is the stabilisation map $C_{l-i-1}(M_{i+1}) \to C_{l-i}(M_{i+1})$. Any map of Serre fibrations over a fixed base space has an associated relative Serre spectral sequence; in this case it has second page
\[
{}^i \widetilde{E}^2_{s,t} \cong H_s( \widetilde{C}_{i+1}(M) ; \widetilde{H}_t( R_{l-i-1}(M_{i+1}) ) )
\]
and converges to $\widetilde{H}_*(R_l(M)^i)$, where $R_l(M)^i$ denotes the homotopy cofibre of the lift $C_l(M)^i \to C_{l+1}(M)^i$ of the stabilisation map.

For $1\leq j\leq k$ there are maps of augmented semi-simplicial spaces $C_{k-j}(M) \times C_j(\bR^n)^\bullet \to C_k(M)^\bullet$ defined similarly to the stabilisation map, except one stabilises by adding the given configuration in $\bR^n$ instead of just a single point. In \cite[\S 8]{RW-hs-for-ucs} it is explained how these induce maps of semi-simplicial spaces $R_{k-j-1}(M)\wedge R_j(\bR^n)^\bullet \to R_k(M)^\bullet$ for $1\leq j\leq k$. Note that when $j=k$ we have $R_{-1}(M)=S^0$ and this is just the map $R_k(\bR^n)^\bullet \to R_k(M)^\bullet$ induced by an embedding $\bR^n \hookrightarrow M$. Each semi-simplicial space has an associated spectral sequence so we obtain a map ${}^j\bar{E}\to E$ of spectral sequences whose first pages are
\begin{align*}
{}^j\bar{E}^1_{s,t} &\cong \widetilde{H}_t(R_{k-j-1}(M)\wedge R_j(\bR^n)^s) \\
E^1_{s,t} &\cong \widetilde{H}_t(R_k(M)^s).
\end{align*}
Note that these are first quadrant plus an extra column $\{s=-1,t\geq 0\}$.

The spectral sequence $E$ converges to $\widetilde{H}_{*+1}$ of the homotopy cofibre of the map $\lVert R_k(M)^\bullet \rVert \to R_k(M)$ induced by the augmentation map. Since taking homotopy cofibres commutes with taking geometric realisation of semi-simplicial spaces this space can also be obtained as follows: first take the homotopy cofibres of the maps $\lVert C_k(M)^\bullet \rVert \to C_k(M)$ and $\lVert C_{k+1}(M)^\bullet \rVert \to C_{k+1}(M)$; these are related by a map induced by stabilisation; then take the homotopy cofibre of this map. Now the augmented semi-simplicial space $C_k(M)^\bullet$ is a $(k-1)$-resolution \cite[Proposition 6.1]{RW-hs-for-ucs}, i.e.\ the map $\lVert C_k(M)^\bullet \rVert \to C_k(M)$ is $(k-1)$-connected. Hence the spectral sequence $E$ converges to zero in total degree $*\leq k-1$.

The inductive hypothesis says that
\begin{equation}\tag{IH}\label{IH}
\widetilde{H}_*(R_l(M))=0 \text{ for } *\leq l<k
\end{equation}
and the hypothesis of the proposition says that
\begin{equation}\tag{Hyp}\label{Hyp}
\widetilde{H}_*(R_l(\bR^n))=0 \text{ for } *\leq l.
\end{equation}
From (IH) we deduce that ${}^i\widetilde{E}^2_{s,t}=0$ for $t\leq l-i-1$ so the spectral sequence ${}^i\widetilde{E}$ converges to zero in total degree $*\leq l-i-1$, so
\begin{equation}\label{Eq1}
\widetilde{H}_*(R_l(M)^i)=0 \text{ for } *\leq l-i-1 \text{ and } i\geq 0.
\end{equation}
In other words:
\begin{equation}\label{Eq1a}
E^1_{s,t}=0 \text{ for } t\leq k-s-1 \text{ and } s\geq 0.
\end{equation}
Also, using the K{\"u}nneth theorem, \eqref{Eq1} and \eqref{IH} we deduce that
\begin{equation}\label{Eq2}
{}^j\bar{E}^1_{s,t}=0 \text{ for } t\leq k-s-1,
\end{equation}
where for the case $\{s=-1 \text{ and } j=k\}$ we also need to use \eqref{Hyp}. We now make the following:

\begin{claim}
For $1\leq j\leq k$ the map ${}^j\bar{E}^1_{j,k-j} \to E^1_{j,k-j}$ is surjective.
\end{claim}

The verification of this claim is delayed until the end of the proof. Now a diagram chase in the following:
\begin{center}
\begin{tikzpicture}
[x=1mm,y=1mm]
\node (t1) at (0,15) {${}^j\bar{E}^1_{j,k-j}$};
\node (t2) at (20,15) {${}^j\bar{E}^{j+1}_{j,k-j}$};
\node (t3) at (50,15) {${}^j\bar{E}^{j+1}_{-1,k}$};
\node (t4) at (70,15) {${}^j\bar{E}^1_{-1,k}$};
\node (b1) at (0,0) {$E^1_{j,k-j}$};
\node (b2) at (20,0) {$E^{j+1}_{j,k-j}$};
\node (b3) at (50,0) {$E^{j+1}_{-1,k}$};
\node (b4) at (70,0) {$E^1_{-1,k}$};
\node at (79,15) {$=0$};
\draw[->>] (t1) to (t2);
\draw[->] (t2) to node[above,font=\small]{$\bar{d}_{j+1}$} (t3);
\draw[->>] (t4) to (t3);
\draw[->>] (b1) to (b2);
\draw[->] (b2) to node[below,font=\small]{$d_{j+1}$} (b3);
\draw[->>] (b4) to (b3);
\draw[->>] (t1) to (b1);
\draw[->] (t2) to (b2);
\draw[->] (t3) to (b3);
\draw[->] (t4) to (b4);
\end{tikzpicture}
\end{center}
shows that the differential $d_{j+1}\colon E^{j+1}_{j,k-j} \to E^{j+1}_{-1,k}$ is zero for $1\leq j\leq k$.

Now we can deduce that the first differential $d_1\colon E^1_{0,t} \to E^1_{-1,t}$ is surjective in a range. First, for $t\leq k-1$ note that the differentials hitting $E^{\square}_{-1,t}$ have source $E^j_{j-1,t-j+1}$ for $1\leq j\leq t+1$. By \eqref{Eq1a} these groups are all zero, so $E^1_{-1,t} = E^{\infty}_{-1,t}$. The spectral sequence $E$ converges to zero in total degree $t-1$ so $E^1_{-1,t}=E^{\infty}_{-1,t}=0$ and so the first differential $d_1\colon E^1_{0,t} \to E^1_{-1,t}$ is vacuously surjective. For $t=k$ we use the result of the diagram chase above, which tells us that the only possible non-zero differential hitting $E^{\square}_{-1,k}$ is the first differential. We know that $E^{\infty}_{-1,k}=0$ since $E$ converges to zero in total degree $k-1$ so the first differential $d_1\colon E^1_{0,k} \to E^1_{-1,k}$ must be surjective. This can be identified as the map on homology induced by the augmentation map $R_k(M)^0 \to R_k(M)$. Hence we have established:

\begin{fact}\label{Fact1}
The augmentation map $a\colon R_k(M)^0 \to R_k(M)$ induces surjections on $A$-homology up to degree $k$.
\end{fact}

Now consider the maps $p\colon C_k(M) \to C_k(M_1)$ and $u\colon C_k(M_1) \to C_k(M)$, defined as follows. The map $p$ is defined similarly to the stabilisation map. Write $M = \mathrm{int}(\mbar)$ for a manifold $\mbar$ with non-empty boundary and choose a self-embedding $e^\prime\colon \mbar\hookrightarrow\mbar$ which is isotopic to the identity and whose image does not contain the missing point of $M_1$. Then $p$ is defined by applying $e^\prime$ to each point of the configuration. The map $u$ is simply the map induced by the inclusion $M_1 \hookrightarrow M$. Since $e^\prime$ is isotopic to the identity the composition $u\circ p$ is homotopic to the identity, and so the induced maps $u_*$ and $p_*$ on homology are semi-inverses: $u_*\circ p_*={\id}$. If we are careful to define $p$ using a self-embedding $e^\prime\colon\mbar\hookrightarrow\mbar$ whose support is disjoint from the self-embedding $e\colon\mbar\hookrightarrow\mbar$ used to define the stabilisation map $s$, then $p$ commutes on the nose with $s$ and there are induced maps $p\colon R_k(M)\to R_k(M_1)$ and $u\colon R_k(M_1) \to R_k(M)$ on mapping cones. Again we have $u\circ p \simeq {\id}$ so $u_*\circ p_*={\id}$.

The methods of the proof of Proposition 6.3 in \cite{RW-hs-for-ucs} show that
\[
\hconn_A(u\colon R_{k-1}(M_1) \to R_{k-1}(M)) \geq \hconn(s\colon C_{k-2}(M) \to C_{k-1}(M)) + \dim(M)
\]
where $\hconn_A(f)$ is the $A$-homology-connectivity of $f$, i.e.\ the largest $*$ such that $\widetilde{H}_*(\mathit{mc}(f);A)=0$, where $\mathit{mc}(f)$ is the mapping cone of $f$. By inductive hypothesis the right-hand side is at least $k-2+\dim(M)\geq k+1$ since we have assumed that $M$ is at least $3$-dimensional. Therefore the $A$-homology-connectivity of $p\colon R_{k-1}(M) \to R_{k-1}(M_1)$ is at least $k$. In particular we have:

\begin{fact}\label{Fact2}
The map $p\colon R_{k-1}(M) \to R_{k-1}(M_1)$ induces surjections on $A$-homology up to degree $k$.
\end{fact}

For our third and final fact, consider the spectral sequence ${}^0\widetilde{E}$ with $l=k$ and recall from just before \eqref{Eq1} that ${}^0\widetilde{E}^2_{s,t}=0$ for $t\leq k-1$. This is the relative Serre spectral sequence for the map of fibre bundles $C_k(M)^0 \to C_{k+1}(M)^0$ over $\widetilde{C}_1(M)=M$. The inclusion of the fibre over a point $*\in M$ is the map $C_{k-1}(M_1)=C_{k-1}(M\smallsetminus\{*\}) \to C_k(M)^0$ which adds the point $*$ to a configuration and labels it by $0$. This induces a map $f\colon R_{k-1}(M_1) \to R_k(M)^0$ on mapping cones. The map on $\widetilde{H}_*$ induced by $f$ can be identified with the composition of the edge homomorphism
\[
\widetilde{H}_*(R_{k-1}(M_1)) = {}^0\widetilde{E}^2_{0,*} \twoheadrightarrow {}^0\widetilde{E}^{\infty}_{0,*}
\]
and the inclusion
\[
{}^0\widetilde{E}^{\infty}_{0,*} \hookrightarrow \widetilde{H}_*(R_k(M)^0)
\]
given by all the extension problems in total degree $*$. But since the second page is trivial for $t\leq k-1$ there are no extension problems in total degree $*\leq k$, and so this inclusion is an isomorphism. Hence we have:

\begin{fact}\label{Fact3}
The map $f\colon R_{k-1}(M_1) \to R_k(M)^0$ induces surjections on $A$-homology up to degree $k$.
\end{fact}

The composition $s^\prime \coloneqq a\circ f\circ p\colon R_{k-1}(M) \to R_k(M)$ is defined exactly like the stabilisation map $s\colon R_{k-1}(M) \to R_k(M)$ except that it uses the self-embedding $e^\prime$ of $\mbar$ instead of $e$. Since we chose $e$ and $e^\prime$ to have disjoint support, the maps $s$ and $s^\prime$ commute. If we now ensure that we picked $e$ and $e^\prime$ to be isotopic, we have that $s$ and $s^\prime$ are homotopic. The square $s\circ s^\prime = s^\prime\circ s$ induces a map of long exact sequences:
\begin{center}
\begin{tikzpicture}
[x=1mm,y=1mm]
\node (t1) at (0,15) {$\widetilde{H}_t(C_k(M))$};
\node (t2) at (30,15) {$\widetilde{H}_t(R_{k-1}(M))$};
\node (t3) at (60,15) {$\widetilde{H}_{t-1}(C_{k-1}(M))$};
\node (t4) at (95,15) {$\widetilde{H}_{t-1}(C_k(M))$};
\node (b1) at (0,0) {$\widetilde{H}_t(C_{k+1}(M))$};
\node (b2) at (30,0) {$\widetilde{H}_t(R_k(M))$};
\draw[->] (t1) to node[above,font=\small]{$c$} (t2);
\draw[->] (t2) to (t3);
\draw[->] (t3) to node[above,font=\small]{$b=s_*$} (t4);
\draw[->] (b1) to node[below,font=\small]{$d$} (b2);
\draw[->] (t1) to node[left,font=\small]{$s^{\prime}_*$} (b1);
\draw[->] (t2) to node[right,font=\small]{$a=s^{\prime}_*$} (b2);
\draw[->] (-15,15) to node[above,font=\small]{$s_*$} (t1);
\draw[->] (-15,0) to node[below,font=\small]{$s_*$} (b1);
\draw[->] (b2) to (45,0);
\draw[->] (t4) to (110,15);
\end{tikzpicture}
\end{center}
Let $t\leq k$ -- our aim is to show that $\widetilde{H}_t(R_k(M))=0$. By Facts \ref{Fact1}, \ref{Fact2} and \ref{Fact3} above, the map $a$ in this diagram is surjective. As mentioned at the beginning of the proof, the stabilisation map is split-injective on homology in all degrees \cite[page 103]{McDuff}, so the map $b$ is injective, and so by exactness the map $c$ is surjective. Hence the composite $a\circ c$ is surjective. But
\[
a\circ c = d\circ s^{\prime}_* = d\circ s_* = 0,
\]
so its codomain $\widetilde{H}_t(R_k(M))$ must be trivial.

It now remains to prove the claim we made earlier in the proof, namely that the map
\[
{}^j\bar{E}^1_{j,k-j} = \widetilde{H}_{k-j}(R_{k-j-1}(M)\wedge R_j(\bR^n)^j) \longrightarrow \widetilde{H}_{k-j}(R_k(M)^j) = E^1_{j,k-j}
\]
is surjective. In fact we will show that the map $R_{k-j-1}(M)\wedge R_j(\bR^n)^j \to R_k(M)^j$ induces surjections on homology in degrees $t\leq k-j$. First note that
\begin{align*}
R_j(\bR^n)^j &= \mathit{mc}(C_j(\bR^n)^j \to C_{j+1}(\bR^n)^j) \\
&= \mathit{mc}(\varnothing \to \widetilde{C}_{j+1}(\bR^n)) \\
&= \widetilde{C}_{j+1}(\bR^n)_+
\end{align*}
and the map $R_{k-j-1}(M)\wedge \widetilde{C}_{j+1}(\bR^n)_+ \to R_k(M)^j$ is given by taking mapping cones of the horizontal arrows in the commutative square:
\begin{center}
\begin{tikzpicture}
[x=1mm,y=1mm]
\node (tl) at (0,15) {$C_{k-j-1}(M)\times \widetilde{C}_{j+1}(\bR^n)$};
\node (tr) at (60,15) {$C_{k-j}(M)\times \widetilde{C}_{j+1}(\bR^n)$};
\node (bl) at (0,0) {$C_k(M)^j$};
\node (br) at (60,0) {$C_{k+1}(M)^j$};
\draw[->] (tl) to node[above,font=\small]{$s\times{\id}$} (tr);
\draw[->] (bl) to node[below,font=\small]{$s$} (br);
\draw[->] (tl) to (bl);
\draw[->] (tr) to (br);
\end{tikzpicture}
\end{center}

To do this we begin by defining some more explicit models for various maps. As before, write $M=\mathrm{int}(\mbar)$ for a manifold $\mbar$ with non-empty boundary and choose two isotopic self-embeddings $e,e^\prime \colon \mbar\hookrightarrow\mbar$ which are both non-surjective and have disjoint support. Choose an embedding $\phi\colon \bR^n \hookrightarrow M\smallsetminus e^\prime(\mbar)$ and pairwise disjoint points $p_0,\dotsc,p_j \in \bR^n$. Write $M_{j+1} = M\smallsetminus\{\phi(p_0),\dotsc,\phi(p_j)\}$. We have a square of maps
\begin{equation}\label{abcd}
\centering
\begin{split}
\begin{tikzpicture}
[x=1mm,y=1mm]
\node (tl) at (0,15) {$C_{k-j-1}(M)$};
\node (tr) at (40,15) {$C_{k-j-1}(M)\times \widetilde{C}_{j+1}(\bR^n)$};
\node (bl) at (0,0) {$C_{k-j-1}(M_{j+1})$};
\node (br) at (40,0) {$C_k(M)^j$};
\draw[->] (tl) to node[above,font=\small]{$\alpha$} (tr);
\draw[->] (bl) to node[below,font=\small]{$\beta$} (br);
\draw[->] (tl) to node[left,font=\small]{$\gamma$} (bl);
\draw[->] (tr) to node[right,font=\small]{$\delta$} (br);
\end{tikzpicture}
\end{split}
\end{equation}
defined by
\begin{align*}
\alpha(c) &= (c,(p_0,\dotsc,p_j)) \\
\gamma(c) &= e^\prime(c) \\
\beta(c) &= c\cup \{\phi(p_0),\dotsc,\phi(p_j)\};i\mapsto \phi(p_i) \\
\delta(c,(q_0,\dotsc,q_j)) &= e^\prime(c)\cup \{\phi(q_0),\dotsc,\phi(q_j)\};i\mapsto \phi(q_i).
\end{align*}

Choose a point $*\in M\smallsetminus e(\mbar)$ and take an explicit model for the stabilisation map to be defined by $c\mapsto e(c)\cup\{*\}$. Since $e$ and $e^\prime$ have disjoint support this induces a map of squares from \eqref{abcd} to $\eqref{abcd}[k\mapsto k+1]$. Taking mapping cones along this map of squares gives us the following:
\begin{center}
\begin{tikzpicture}
[x=1mm,y=1mm]
\node (tl) at (0,15) {$R_{k-j-1}(M)$};
\node (tr) at (40,15) {$R_{k-j-1}(M)\wedge \widetilde{C}_{j+1}(\bR^n)_+$};
\node (bl) at (0,0) {$R_{k-j-1}(M_{j+1})$};
\node (br) at (40,0) {$R_k(M)^j$};
\draw[->] (tl) to node[above,font=\small]{$\bar{\alpha}$} (tr);
\draw[->] (bl) to node[below,font=\small]{$\bar{\beta}$} (br);
\draw[->] (tl) to node[left,font=\small]{$\bar{\gamma}$} (bl);
\draw[->] (tr) to node[right,font=\small]{$\bar{\delta}$} (br);
\end{tikzpicture}
\end{center}

We need to show that $\bar{\delta}$ induces surjections on homology in degrees $t\leq k-j$. This will follow if we can prove this for $\bar{\gamma}$ and $\bar{\beta}$. But $\bar{\gamma}$ is the composition of $j+1$ instances of the map $p$ from Fact \ref{Fact2}, and so this does induce surjections on homology up to degree $k-j$ by Fact \ref{Fact2}.\footnote{The proofs of Facts \ref{Fact2} and \ref{Fact3} earlier did not depend on the claim which we are currently proving, so this is not circular.} When $j=0$ the map $\bar{\beta}$ is surjective on homology up to degree $k$ by Fact \ref{Fact3}. Moreover, the argument proving Fact \ref{Fact3} generalises (using the spectral sequence ${}^j\widetilde{E}$ instead of ${}^0\widetilde{E}$) to prove precisely that the map $\bar{\beta}$ is surjective on homology up to degree $k-j$ in general.
\end{proof}

\begin{remark}\label{rSESofcoeffs}
When $\dim(M)\geq 3$ we have homological stability in the range $*\leq k$ for $\bQ$ coefficients (by \cite[Theorem B]{RW-hs-for-ucs}) and for $\bZ/p$ coefficients with $p$ odd (by Proposition \ref{pImprovedRange} above). Using the short exact sequences of coefficients $0\to \bZ/p \to \bZ/p^{l+1} \to \bZ/p^l \to 0$ and
\[
0\to \bZ[\tfrac12]\to \bQ\to \textstyle{\bigoplus}_{p\neq 2}\colim_{l\to\infty}\bZ/p^l \to 0
\]
this implies homological stability in the range $*\leq k-1$ for $\bZ[\frac12]$ coefficients. This recovers Theorem 1.4 of \cite{KupersMiller2014}, except without surjectivity in degree $k$.
\end{remark}

\section{Homological stability for configuration spaces with labels in a fibre bundle}\label{appendixB}

\begin{df}[Configuration spaces and stabilisation maps with labels in a fibre bundle]\label{dConfigLabels}
Let $\theta\colon E\to M$ be a fibre bundle with path-connected fibres $F$ and define
\[
C_k(M;\theta) \coloneqq \{ \{p_1,\dotsc,p_k\} \subset E \mid \theta(p_i) \neq \theta(p_j) \text{ for } i\neq j \}.
\]
Choose a self-embedding $e\colon\mbar\hookrightarrow\mbar$ which is non-surjective and isotopic to the identity. Choose an open neighbourhood $U\subseteq \mbar$ containing the support of $e$, write $V=U\smallsetminus \partial\mbar$ and choose a trivialisation $\phi\colon \theta^{-1}(V) \to V\times F$ of $E$ over $V$. Define a self-embedding $\widetilde{e}\colon E\hookrightarrow E$ by
\[
p\mapsto \begin{cases}
p & p\notin \theta^{-1}(V) \\
\phi^{-1}\circ (e\times\mathrm{id})\circ \phi(p) & p\in \theta^{-1}(V)
\end{cases}
\]
and note that $\theta\circ \widetilde{e} = e\circ\theta$. Also choose points $*\in M\smallsetminus e(\mbar) \subseteq V$ and $x\in F$. We can then define the stabilisation map $C_k(M;\theta) \to C_{k+1}(M;\theta)$ by
\[
\{p_1,\dotsc,p_k\} \mapsto \{\widetilde{e}(p_1),\dotsc,\widetilde{e}(p_k),\phi^{-1}(*,x)\}.
\]
\end{df}

We may generalise Proposition \ref{pImprovedRange} to configuration spaces with labels in $\theta$ using the following fact.

\begin{remark}\label{rTwistedCoefficients}
Configuration spaces also satisfy homological stability with respect to finite-degree twisted coefficient systems: for the case of symmetric groups this was proved by \cite{Betley2002}, and the general case was proved in \cite{Palmer2013a}. A twisted coefficient system for $M$ is a functor from the partial braid category $\cB(M)$ to $\bZ\text{-mod}$. The partial braid category $\cB(M)$ has objects $\{0,1,2,\dotsc\}$ and a morphism from $m$ to $n$ is a path in $C_k(M)$ from a subset of $\{p_1,\dotsc,p_m\}$ to a subset of $\{p_1,\dotsc,p_n\}$, up to endpoint-preserving homotopy, where $\{p_1,p_2,p_3,\dotsc\}$ is a fixed injective sequence in $M$.

If the twisted coefficient system has degree $d$ the stable range obtained is $*\leq\frac{k-d}{2}$, which arises since homological stability with untwisted $\bZ$ coefficients in the range $*\leq \frac{k}{2}$ is an input for the proof. However, if the twisted coefficient system takes values in the subcategory $\bZ[\frac12]$-mod of $\bZ$-mod and $\dim(M)\geq 3$, then we may instead input \cite{KupersMiller2014} or Proposition \ref{pImprovedRange} to obtain a stable range of $*\leq k-d$ for $C_k(M)$ with coefficients in a functor $\cB(M)\to \bZ[\frac12]\text{-mod}$ of degree $d$.
\end{remark}

\begin{proposition}\label{pImprovedRangeTwisted}
Let $M$ be a smooth, connected, open $n$-manifold with $n\geq 2$ and $\theta\colon E\to M$ a fibre bundle with path-connected fibres. Then the stabilisation map $C_k(M;\theta) \to C_{k+1}(M;\theta)$ induces isomorphisms on $H_*(-;\bZ)$ in the range $*\leq \frac{k}{2}-1$. It induces isomorphisms in the range $*\leq k$ on $H_*(-;\bQ)$, unless $M$ is an orientable surface in which case the range is only $*\leq k-1$. If $n\geq 3$ it induces isomorphisms on $H_*(-;\bZ[\frac12])$ in the range $*\leq k-1$.
\end{proposition}

The worse range $*\leq k-1$ for rational homology of configurations on orientable surfaces is necessary: for example $H_1(C_1(\bR^2);\bQ) = 0 \not\cong \bQ \cong H_1(C_2(\bR^2);\bQ)$.

\begin{remark}\label{rImprovedRangeTwisted}
A version of Proposition \ref{pImprovedRangeTwisted} is also proved in the appendix of \cite{KupersMiller2014Encellattachments}. One of the proofs given there is essentially the same as the proof we give below, and a sketch proof using semi-simplicial resolutions by collections of disjoint arcs in $M$ is also given. This latter method has the advantage that it gives a range of $*\leq \frac{k}{2}$ for $\bZ$ coefficients (at least when $M$ is orientable), rather than the smaller range $*\leq \frac{k}{2}-1$ for $\bZ$ coefficients obtained in Proposition \ref{pImprovedRangeTwisted}.
\end{remark}

\begin{proof}
This will follow by the same considerations as in Remark \ref{rSESofcoeffs} if we show that it induces isomorphisms on $H_*(-;A)$ in the range $*\leq \frac{k}{2}$ when $A=\bF_p$, in the range $*\leq k$ if either (a) $A=\bQ$ and $M$ is not an orientable surface or (b) $A=\bF_p$ for $p$ odd and $n\geq 3$, and in the range $*\leq k-1$ when $A=\bQ$ and $M$ is an orientable surface. The loss of one degree from the range occurs when going from $\bQ$ and $\bQ/\bZ$ coefficients to $\bZ$ coefficients (resp.\ $\bQ$ and $\bQ/\bZ[\frac12]$ coefficients to $\bZ[\frac12]$ coefficients).

Let $A = \bQ \text{ or } \bF_p$ for a prime $p$. There are fibre bundles $C_k(M;\theta) \to C_k(M)$, given by forgetting labels, with fibre $F^k$. The stabilisation maps $C_k(M) \to C_{k+1}(M)$ and $C_k(M;\theta) \to C_{k+1}(M;\theta)$ commute with these fibre bundles and the map of fibres is the inclusion $F^k \hookrightarrow F^{k+1}$. There is then a map of Serre spectral sequences
\begin{center}
\begin{tikzpicture}
[x=1mm,y=1mm]
\node (tl) at (0,15) [anchor=west] {$E^2_{s,t}\cong H_s(C_k(M);H_t(F^k;A))$};
\node (tr) at (80,15) {$H_*(C_k(M;\theta);A)$};
\node (bl) at (0,0) [anchor=west] {$E^2_{s,t}\cong H_s(C_{k+1}(M);H_t(F^{k+1};A))$};
\node (br) at (80,0) {$H_*(C_{k+1}(M;\theta);A)$};
\coordinate (ya) at (bl.north);
\coordinate (yb) at ($ (bl.north)+(1,0) $);
\coordinate (xa) at (tl.south);
\coordinate (xb) at ($ (tl.south)+(0,1) $);
\draw[->] (tl.south) to (intersection of ya--yb and xa--xb);
\draw[->] (tr) to (br);
\node at (60,0) {$\Rightarrow$};
\node at (60,15) {$\Rightarrow$};
\end{tikzpicture}
\end{center}
and our aim is to prove that the map in the limit is an isomorphism in a certain range depending on $A$ and $M$.

Now by Lemma 4.2 of \cite{Palmer2013a} and since $A$ is a field the assignment $k\mapsto H_t(F^k;A)$ extends to form a twisted coefficient system of degree at most $\frac{t}{h+1}\leq t$ where $h=\hconn_A(F)$. Hence the map on the second page is an isomorphism in the range $s\leq \frac{k-t}{2}$ by Theorem 1.3 of \cite{Palmer2013a}. In particular it is an isomorphism in total degree at most $\frac{k}{2}$ and therefore the same holds for the map in the limit.

To obtain the improved range in certain cases note that, by Remark 6.5 of \cite{Palmer2013a}, if homological stability with (untwisted) $A$ coefficients holds for (unlabelled) configuration spaces on $M$ in the range $*\leq f(k)$, then twisted homological stability will hold in the range $*\leq f(k-d)$ for any twisted coefficient system of degree $d$ which factors through the forgetful functor $A\text{-mod} \to \bZ\text{-mod}$ (c.f.\ Remark \ref{rTwistedCoefficients}).

If $A=\bQ$ then the above twisted coefficient system factors through the inclusion $\bQ\text{-mod} \to \bZ\text{-mod}$. By Theorem C of \cite{RW-hs-for-ucs} we may take $f(k)=k$ if $n=\dim(M)\geq 3$. For orientable surfaces we may take $f(k)=k-1$ by Corollary 3 of \cite{Church} or Theorem 1.3 of \cite{Knudsen2014}, and for non-orientable surfaces we may take $f(k)=k$ by Theorem 1.3 of \cite{Knudsen2014}. By the above paragraph the map of spectral sequences is an isomorphism on the second page in the range $s\leq f(k-t)$, and therefore in total degree at most $k$ (resp.\ total degree at most $k-1$ for orientable surfaces). Hence so is the map in the limit.

If $A=\bF_p$ for $p$ odd and $n\geq 3$ then the twisted coefficient system factors through the inclusion $\bZ[\frac12]\text{-mod} \to \bZ\text{-mod}$. By Proposition \ref{pImprovedRange} (or Theorem 1.4 of \cite{KupersMiller2014}) we may take $f(k)=k$. So as above the map of spectral sequences is an isomorphism in total degree at most $k$, and therefore so is the map in the limit.
\end{proof}

\section{Stable homology of configuration spaces with labels in a fibre bundle}\label{appendix2}
In this appendix we prove Theorem \ref{thm:McDufflabelssection4}. Another proof can be obtained adapting step by step the proof in \cite{McDuff} for trivial labels, as pointed out in the introduction to that paper. We give here a sketch of this proof with some shortcuts, taking advantage of knowing the homology stability theorem with labels (Proposition \ref{pImprovedRangeTwisted}) in the spirit of \cite{GMTW}.

\begin{df} Let $M$ be an open manifold, let $c\colon D^{n-1}\times (0,1]$ be a proper embedding and let $M_1 = M\cup_c (\mathring{D}^{n-1}\times (-1,1])$. Define $\psi^{\delta,\gamma}(M;\theta)$ to be the space whose underlying set is $C^{\delta,\gamma}(M;\theta):= \coprod_k{C_k^{\delta,\gamma}(M;\theta)}$ with the following topology: Consider the quotient $Y$ of $C^{\delta,\gamma}(M_1;\theta)$ under the relation $\sim$ where $(\q,\epsilon,\{f_q\}_{q\in\q})\sim(\q',\epsilon',\{f'_{q'}\}_{q'\in\q'})$ if and only if 
\begin{enumerate}
\item $\q\cap M = \q'\cap M$,
\item if the above intersection is non-empty, then $\epsilon=\epsilon'$,
\item if the above intersection is non-empty, then $f_q = f'_q$ for all $q\in \q\cap M$.
\end{enumerate}
The natural inclusion $C^{\delta,\gamma}(M;\theta)\to C^{\delta,\gamma}(M_1;\theta)$ induces a bijection $C^{\delta,\gamma}(M;\theta)\cong Y$, which we use to endow $C^{\delta,\gamma}(M;\theta)$ with a new topology.
\end{df}
Recall that we defined the non-linear scanning map with labels
\[\ssnl{\theta,\gamma}\colon C_k^{\delta,\gamma}(M;\theta)\lra \Gamma_c(\psi^\delta(T^{1}M;\theta))\]
in \S\ref{sUpToDim}. Following the same recipe we can define a scanning map
\[\ssnl{\theta,\gamma}\colon \psi^{\delta,\gamma}(M;\theta)\lra \Gamma(\psi^\delta(T^{1}M;\theta))\]
whose target is now the whole space of sections.

\begin{lemma}[{\cite[\S 2.3]{Hesselholt:derivative}}]\label{lemma:hesselholt}  If $M$ is connected, then the scanning map is a homotopy equivalence.
\end{lemma}
In the paper, Hesselholt considers $E$ to be a fibre bundle of based spaces. This lemma is a particular case of his theorem, taking a disjoint basepoint in each fibre and the submanifold $N$ in his theorem to be connected.

Let $\pi_1\colon M\dasharrow I$ be the partially defined function that sends a point in the image of $D^{n-1}\times I$ to the second coordinate. Define $\psi^{\delta,\gamma}(M;\theta)_\bullet$ to be the semi-simplicial space whose space of $i$-simplices is the space of tuples $(\q,\epsilon,\{f_q\}_{q\in\q},a_0,\ldots,a_i)$, where $(\q,\epsilon,\{f_q\}_{q\in\q})\in \psi^{\delta,\gamma}(M;\theta)$ and $(a_0,\ldots,a_i)\in I^{i+1}$ and $\pi_1(\q)\cap \{a_0,\ldots,a_i\}=\emptyset$. The $j$th face map forgets $a_j$, and there is an augmentation to $\psi^{\delta,\gamma}(M;\theta)$ that forgets all the $a_j$'s. 
\begin{lemma} The realization of the augmentation
\[\|\psi^{\delta,\gamma}(M;\theta)_\bullet\|\to \psi^{\delta,\gamma}(M;\theta)\]
 is a weak homotoy equivalence.
\end{lemma}
\begin{proof} This is an augmented topological flag complex \cite{GR-W2} satisfying the conditions of Theorem 6.2 in that paper, hence a weak homotopy equivalence. 
\end{proof}

\begin{proposition} If $\theta\colon E\to M$ has path-connected fibres, then the restriction of the scanning map
\[\ssnl{\gamma,\theta}\colon C_k^{\delta,\gamma}(M\setminus c;\theta) \lra \Gamma_c(\Psi^{\delta}(T^1M\setminus c;\theta))\]
is a homology isomorphism in the range in which the stabilisation map of Proposition \ref{pImprovedRangeTwisted} is a homology isomorphism. Since $M\setminus c\cong M$, the same holds for $M$.
\end{proposition}
\begin{proof}
We have constructed the following commutative diagram:
\begin{equation}\label{eq:488}
\begin{split}
\xymatrix@C-1.5ex{
\|\psi^{\delta,\gamma}(M;\theta)_\bullet\|\ar[r]\ar[d] & \psi^{\delta,\gamma}(M;\theta) \ar[r]\ar[d]& \Gamma(\Psi^\delta(T^1M;\theta)) \ar[d] \\
\|\psi^{\delta,\gamma}(D^{n-1}\times I;\theta)_\bullet\|\ar[r] & \psi^{\delta,\gamma}(D^{n-1}\times I;\theta) \ar[r]& \Gamma(\psi^{\delta}(T^1(D^{n-1}\times I);\theta)) 
}
\end{split}
\end{equation}
All the horizontal maps are homotopy equivalences, by the previous two lemmas. The rightmost vertical map is a fibration. We now choose another properly embedded ray $D^{n-1}\times (0,1]\cong L\subset \partial M\setminus c$, and take the colimit
\[P(M;\theta)_\bullet := \colim \left(\psi^{\delta,\gamma}(M;\theta)_\bullet\overset{s}{\lra} \psi^{\delta,\gamma}(M;\theta)_\bullet\lra \ldots\right)\]
with respect to the stabilisation maps $s$ that push the configurations outside $L$ and adds a point in $L$ with some prescribed label. The scanning of this operation $\ssnl{\gamma,\theta}(s)$ gives also a sequence of maps between spaces of sections, whose colimit we denote by
\[G(M;\theta) := \colim \left(\Gamma(\psi^{\delta}(T^1M;\theta))\overset{\ssnl{\theta,\gamma}(s)}{\lra} \Gamma(\psi^{\delta}(T^1M;\theta)) \lra \ldots\right)\]
Observe that the maps $\ssnl{\gamma,\theta}(s)$ increase the degree by $1$. We can consider instead the maps that push the source of the scanning map away from $L$ and glue there the \emph{reflection} of the scanning of some point in $L$, together with some prescribed label. This latter map is a homotopy inverse of $\ssnl{\gamma,\theta}(s)$, hence the maps $\ssnl{\gamma,\theta}(s)$ are homotopy equivalences. 

By Proposition \ref{pImprovedRangeTwisted}, it follows that the semi-simplicial map
\[P(M;\theta)_\bullet\lra \psi^{\delta,\gamma}(D^{n-1}\times I;\theta)_\bullet\]
satisfies the hypotheses of \cite[Proposition~4]{McDuff-Segal}, so the realization
\[\|P(M;\theta)_\bullet\|\lra \|\psi^{\delta,\gamma}(D^{n-1}\times I;\theta)_\bullet\|\]
 is a homology fibration. Its fibre over any point is the colimit of the space $C^{\delta,\gamma}(M;\theta)$ with respect to the stabilisation map $s$. The map
\[G(M;\theta)\lra \Gamma(\psi^{\delta,\gamma}(T^1M;\theta))\]
is a Serre fibration (it is a union of Serre fibrations). Its fibre over any point is the colimit of $\Gamma_c(\psi^\delta(T^1M);\theta)$ with respect to the map obtained by scanning $s$:
\begin{equation}\label{eq:489}\begin{gathered}
\xymatrix{
\colim C^{\delta,\gamma}(M;\theta)\ar[rr]\ar[d] && \colim \Gamma_c(\psi^\delta(T^1M);\theta)\ar[d] \\
\|P(M;\theta)_\bullet\|\ar[rr]^\simeq\ar[d]^{a} & & G(M;\theta) \ar[d]^b \\
\|\psi^{\delta,\gamma}(D^{n-1}\times I;\theta)_\bullet\|\ar[rr]^\simeq & & \Gamma(\psi^{\delta}(T^1M;\theta)) 
}
\end{gathered}
\end{equation}

The fibres of \eqref{eq:489} together with the maps to the fibres of \eqref{eq:489} give the following commutative diagram
\[\xymatrix{
C^{\delta,\gamma}(M;\theta)\ar[rr]^{\ssnl{\gamma,\theta}}\ar[d] && \Gamma_c(\psi^\delta(T^1M);\theta) \ar[d] \\
\colim C^{\delta,\gamma}(M;\theta) \ar[rr] && \colim \Gamma_c(\psi^{\delta}(T^1M);\theta).
}\]
The bottom map is a homology equivalence because the horizontal maps in diagram \eqref{eq:488} are homotopy equivalences. The left vertical map is a homology equivalence in the stable range of Proposition \ref{pImprovedRangeTwisted}. The right vertical map is a homotopy equivalence. As a consequence, the upper horizontal map is a homology equivalence in the range provided by Proposition \ref{pImprovedRangeTwisted}
\end{proof}
The following is proved in the same way as Theorem 1.1 at the bottom of page 34 in \cite{McDuff}. 
\begin{corollary}\label{cor:blabla} If $M$ is a manifold with empty boundary, then the non-linear scanning map
\[\ssnl{\theta,\gamma}\colon C^{\delta,\gamma}(M;\theta) \lra \Gamma_c(\psi^\delta(T^1M;\theta))\]
is a homology isomorphism in the range in which the stabilisation map is a homology isomorphism. 
\end{corollary}
We showed in Section \S\ref{sUpToDim} that the triangle
\begin{equation}
\centering
\begin{split}
\begin{tikzpicture}
[x=1mm,y=1mm]
\node (tl) at (0,15) {$C_k^{\delta,\gamma}(M;\theta)$};
\node (tr) at (40,15) {$\Gamma_c(\psi^\delta (T^1 M;\theta))$};
\node (br) at (40,0) {$\Gamma_c(\dot{T}^{1,\theta} M).$};
\draw[->] (tl) to node[above,font=\small]{$\ssnl{\theta,\gamma}$} (tr);
\draw[->] (tl) to node[anchor=north east,font=\small]{$\ssl{}$} (br);
\draw[->] (tr) to[out=300,in=60] node[right,font=\small]{$h$} (br);
\draw[->] (br) to[out=120,in=240] node[left,font=\small]{$i$} (tr);
\node at (40,7.5) {$\simeq$};
\end{tikzpicture}
\end{split}
\end{equation}
commutes, hence from Corollary \ref{cor:blabla} it follows that:
\begin{theorem}[McDuff's Theorem with labels]\label{thm:McDufflabels} The linear scanning map with labels
\[\ssl{\delta,\theta}\colon C_k(M;\theta)\lra \Gamma_c(\dot{T}^\theta(M))_k\]
induces an isomorphism on homology groups in the stable range provided by Proposition \ref{pImprovedRangeTwisted}.
\end{theorem}

\bibliographystyle{apalike}
\bibliography{hs-for-config-spaces-on-closed-manifolds}
\end{document}